\documentclass[11pt,a4paper]{amsart}
\usepackage[toc,page]{appendix}
\usepackage{a4}
\usepackage[active]{srcltx}
\usepackage{amsmath,bbm,esint}
\usepackage{amssymb}
\usepackage{accents,calc}
\usepackage{subfig}
\usepackage{caption}
\usepackage{dsfont}
\usepackage{mwe}
\usepackage{comment}
\usepackage{tikz}
\usepackage{amsfonts}
\usepackage{amsmath}
\usepackage{dsfont}
\usepackage{enumitem} 
\usepackage{xcolor}
\usepackage{mathrsfs}

\usepackage[utf8]{inputenc}
\usepackage[T1]{fontenc}
\usepackage{lmodern}

\usepackage[normalem]{ulem} 

\usepackage[colorlinks=true,urlcolor=blue,
citecolor=red,linkcolor=blue,linktocpage,pdfpagelabels,bookmarksnumbered,bookmarksopen]{hyperref}

\usepackage{xcolor}

\usepackage{collectbox}



\def\e{\varepsilon}

\def\Om{\Omega}

\def \p{\partial}

\def \0{\mathbf{0}}

\newcommand{\R}{{\mathbb R}}

\newcommand{\N}{\mathbb{N}}

\usepackage{subfig}

\newcommand{\be}{\begin{equation}}
\newcommand{\ee}{\end{equation}}

\newcommand{\ba}{\begin{array}}
\newcommand{\ea}{\end{array}}

\usepackage{subfiles}
\newtheorem{theorem}{\textbf{Theorem}}[section]
\newtheorem{remark}[theorem]{\textbf{Remark}}
\newtheorem{lemma}[theorem]{\textbf{Lemma}}
\newtheorem{example}[theorem]{\textbf{Example}}
\newtheorem{corollary}[theorem]{\textbf{Corollary}}

\newtheorem{proposition}[theorem]{\textbf{Proposition}}
\newtheorem{definition}[theorem]{\textbf{Definition}}

\newtheorem{innercustomgeneric}{\customgenericname}
\providecommand{\customgenericname}{}
\newcommand{\newcustomtheorem}[2]{%
  \newenvironment{#1}[1]
  {%
   \renewcommand\customgenericname{#2}%
   \renewcommand\theinnercustomgeneric{##1}%
   \innercustomgeneric
  }
  {\endinnercustomgeneric}
}

\newcustomtheorem{customthm}{Framework}

\newtheorem{assumption}{\textbf{Assumption}}
\numberwithin{equation}{section}

\begin{document}
\title[]{The nonlocal Stefan Problem via a Martingale transport}
\author{Raymond Chu, Inwon Kim, Young-Heon Kim, Kyeongsik Nam} 
\begin{abstract}

    We study the nonlocal Stefan problem, where the phase transition is described by  a nonlocal diffusion as well as the change of enthalpy functions. By using a stochastic optimization approach introduced in \cite{kim2021stefan}, we construct global-time weak solutions and give a probabilistic interpretation for the solutions.
    An important ingredient in our analysis is a probabilistic interpretation of the enthalpy and temperature variables in terms of a particle system. Our approach in particular establishes the connection between the parabolic obstacle problem and the Stefan Problem for the nonlocal diffusions. For the melting problem, we show that our temperature-based solution coincides with the enthalpy-based ones studied in \cite{athana_caf2010, del2017uniqueness, del2017distributional,del2021one}, and obtain a new exponential convergence result.
\end{abstract}

\address[Raymond Chu]{Department of Mathematics, UCLA, California}
\email{rchu@math.ucla.edu}

\address[Inwon Kim]{Department of Mathematics, UCLA, California}
\email{ikim@math.ucla.edu}

\address[Young-Heon Kim]{Department of Mathematics, University of British Columbia, Canada}
\email{yhkim@math.ubc.ca}

\address[Kyeongsik Nam]{Department of Mathematical Sciences, KAIST, South Korea}
\email{ksnam@kaist.ac.kr}

\maketitle


\maketitle


\section{Introduction}

The one-phase Stefan problem is a phase transition model between liquid and solid that describes the melting of ice or the freezing of supercooled water \cite{rubinshte71,meirmanov11}. In this model the phase transition is described by a diffusion in the liquid phase and a change in the latent heat between two phases.   Both effects are recorded in the {\it enthalpy} variable $h$.   For the freezing problem, it can be written as 
\begin{equation}  \tag{$St_1$} \label{St1} 
\p_t h + (-\Delta)^s \eta = 0, \quad h = \eta -1\hbox { on } \{\eta>0\}, 
\end{equation}
\noindent where $\eta$ represents the negative temperature of the fluid and $s \in (0,1]$ is fixed. 

The melting problem can be written as 
\begin{equation} \tag{$St_2$} \label{St2}
\p_t h + (-\Delta)^s \eta = 0, \quad h = \eta +1 \hbox { on } \{\eta>0\}.
\end{equation}
where  $\eta$ now  represents the temperature of the fluid.
\medskip

While the classical models consider the local diffusion $s=1$, we allow the diffusion to be either local or nonlocal by setting $s\in (0,1]$. This extension is natural from the viewpoint of applications (for instance see \cite{liu2004exact,voller2014fractional,zhang2017review}).

\medskip

 The nonlinear dynamics in the Stefan problem occurs at the interface $\Gamma= \partial\{\eta>0\}$. Below we will regard the set  $\{\eta>0\}$ as the water/liquid region and $\{\eta=0\}$ as the ice/solid region. We will see that the set  $\{\eta=0\}$ increases in time for \eqref{St1}, and decreases in time for \eqref{St2}. The challenge in studying either  type of Stefan problem lies in the unknown regularity of both $\Gamma$ and the enthalpy variable $h$ across $\Gamma$.

\medskip

While formally \eqref{St1} and \eqref{St2} fully describe the evolution of $h$ and $\eta$, it is challenging to describe the evolution of the problem as a system, especially in the nonlocal case where $h$ does not stay constant in the zero set of $\eta$; the reason being that the underlying particles could jump. The difficulty lies in the description of $h$ on the interface $\partial\{\eta>0\}$.  For the nonlocal case, any type of regularity properties of $\Gamma$  is unknown for $0<s<1$ for either \eqref{St1} or \eqref{St2}, and likely $\Gamma$ may be highly irregular due to the nonlocal diffusion present in the equation. For the local case, \eqref{St1} is well-known to have non-unique evolution with jump-discontinuities \cite{sherman1970general,chayes2012supercooled},  though recently a physically natural notion of minimal-jump solutions was shown to generate unique solutions in one space dimension \cite{delarue}. It is likely that the solutions have very little regularity for the interface \cite{ckk24,misha}. On the other hand, the local version of \eqref{St2} is well-known to generate stable weak solutions with regularizing properties, extensively studied in the literature (see e.g. \cite{acs,choikim,  figalli_stefan, hadzic, meirmanov11}). For the nonlocal case, most available result regarding \eqref{St2} focus on  well-posedness result for the enthalpy-version of the problem (see \eqref{Ste} below): see  \cite{athana_caf2010,athana_caf2022,del2017distributional, del2017uniqueness, del2021one}. Later our analysis will in particular show that our temperature-based solution of $\eqref{St2}$ is equivalent to the enthalpy-based solution of $\eqref{Ste}$, when  $(h-\eta)(0,x)\in \{0,1\}$ (see Theorem~\ref{melting_theorem} ($c$)). We also point out that our notion of weak solutions for $\eqref{Ste}$ is the same as those studied in the aforementioned references.

\medskip

 Our goal in this paper is to characterize the enthalpy variable in both \eqref{St1} and \eqref{St2} with a particle dynamics approach, aiming for a better understanding of enthalpy behavior across the free boundary, in particular in the event of an irregular interface. More precisely, our main result is construction of global-in-time solutions of \eqref{St1} and \eqref{St2} based on the temperature variable $\eta$ (see Theorems~ \ref{meltingtheorem} and \ref{freeze:theorem}). As we will explain, this is a much more delicate task than the local case $s=1$, where $h$ is constant in the zero set of $\eta.$ For \eqref{St2} our solutions are also the enthalpy-based weak solutions in the literature (Theorem~\ref{meltingtheorem}), yet they come with stronger properties. For instance, using the new characterization of enthalpy for \eqref{St2} and a variational approach to the associated particle dynamics, we show that the time integral of the temperature variable $\eta$ in \eqref{St2} solves a parabolic nonlocal obstacle problem (see \eqref{obstacle_p} in Theorem \ref{melting_theorem}). This connection is established for the first time in the nonlocal case $0<s<1$ (for the local case this is a classical and very useful fact, see \cite{duvaut1973resolution, kim2010homogenization}).
Our approach also yields a new quantitative convergence result for the enthalpy variable (Theorem \ref{melting_theorem} ($d$)), answering one of the open questions addressed in \cite{del2021one}.

 \medskip
 
  To achieve our result, we use a particle dynamics interpretation, which was considered in \cite{ghoussoub2019pde} and adapted by  \cite{kim2021stefan} to study the local Stefan problem. We will show that both \eqref{St1} and \eqref{St2}
 can be written in terms of the temperature variable as 
\begin{equation} \tag{$St$} \label{St}
    \p_t \eta +(-\Delta)^{s} \eta = -\rho, \quad  \rho([0,\infty),\cdot)  = \nu(\cdot).
\end{equation} 

Here, while $\eta$ denotes the temperature or the distribution of active heat particles as in the aforementioned problems, $\rho$ is a space-time measure that represents the distribution of \textit{stopped particles} at time $t$, as the particles change their  state from mobile to immobile: see Appendix \ref{measures} for the definition of $\rho([0,t), x)$, which can be formally considered as $\int_0^t \rho(s,x)ds$. 
The measure $\nu$ corresponds to the spatial capacity of the phase transition, or the final accumulated distribution of stopped particles, which represents the final outcome of the phase transition. The particles follow stochastic process $X_t$ corresponding to the fractional Laplacian $(-\Delta)^s$, called $2s$-stable L\'evy process, where particles can jump; see Definition~\ref{stable}. 
 Extending the analysis of \cite{kim2021stefan}, we will rigorously justify this characterization of $\rho$ and $\nu$, using the particle dynamics based on an optimization problem we describe below;  see \eqref{free} and  Theorem~\ref{th:free-target}.\\

To formulate \eqref{St1} and \eqref{St2} in terms of the variables in \eqref{St},  we  will verify that the enthalpy is given by the following formula:
\begin{equation} \label{enthalpy0} \tag{1.1}
   h(t,x) := \begin{cases}
    \eta(t,x) + \rho([0,t),x) - \nu(x) &\text{ for } \eqref{St1}, \\
    \eta(t,x) +  \rho([0,t),x) &\text{ for } \eqref{St2}.
\end{cases} 
\end{equation} 
For \eqref{St2}, $h$ denotes the distribution of total, active and stopped, heat particles. For \eqref{St1}, $h$ also reflects  the  enthalpy loss in the active (liquid) phase. This formula is motivated by the fact that the stopped distribution $\rho$ in \eqref{St} occurs where the moving particles of $\eta$ are stopped by a space-time barrier that has a certain time-monotonicity (corresponding to \eqref{St1} or \eqref{St2}). As we will see in Section~
\ref{sec:target} such a barrier as well as the distribution $\nu$ will be uniquely determined by $\eta$, so that every term in the formula \eqref{enthalpy0} is determined by $\eta$.

 \medskip

  Unlike the local case of the Stefan problem, in the nonlocal case the particles jump,
which generates a significant difference in the relation between $\eta$ and $h$.  In particular, while $h = \eta \pm \chi_{\{\eta>0\}}$ for the local case $s=1$, for the nonlocal case $0<s<1$  the variables $h$ and $\eta$ have  less explicit relation due to nonlocal effects.
Our idea is 
to study properties of the rather singular measure $\rho$ to show that $h$ given by the formula \eqref{enthalpy0} paired with $\eta$ indeed solves \eqref{St1}-\eqref{St2}.  This is a delicate task we carry out in Section \ref{sec:Stefan}, with a careful analysis on the properties of $\rho$. We will discuss more about the differences of the local and non-local case in 
Section \ref{sec 1.1}.

\medskip

 While there are parallel parts in our analysis of the optimization problem compared to the local case $s=1$ in \cite{kim2021stefan}, there are significant challenges that are new for the nonlocal case $0<s<1$. A fundamental difference lies in the associated stochastic process: while Brownian motion for $s=1$ has continuous sample paths, the symmetric  $2s$-stable L\'evy process for $0<s<1$ is a pure jump process with arbitrarily large jump sizes. As a consequence, our enthalpy function is no longer  constant in the solid phase. This makes our particle-based definition of enthalpy more useful even for the stable melting problem  \eqref{St2}.  The jumps also bring interesting new challenges already in early parts of the analysis, as we will discuss in the next subsection: for instance the stopped particles $\rho$ are no longer concentrated on $\Gamma$ (see Figure \ref{fig:Stopped_Active_Region}). 
 This leads to the lack of compactness for the existence property of  {our optimization problem}
 $\mathcal{P}_f(\mu)$ ({see Section \ref{sec 1.1} below for the definition}),   which we overcome with the idea of \emph{compact active region}, see the discussion in Section \ref{sec2}. Combining  this compactness with the results in  \cite{ghoussoub2021optimal} leads to  the well-posedness of $\mathcal{P}_f(\mu)$.

\medskip

 Let us mention that our methods of constructing solutions to \eqref{St} do not rely on special properties of the $2s$-stable processes. The key properties of the  stable process we exploit are that its sample path is c\`adl\`ag,  i.e.
 right continuous with left limits, the infinitesimal generator is  uniformly elliptic, and that the process is transient (see Assumption \ref{assume1}). In particular, we expect that  our methods can be generalized to  other various types of non-local operators and even mixed local and non-local diffusion, e.g. $(-\Delta)^s - \Delta$.  It can also be easily extended to domains in Riemanninan manifolds.

\subsection{Summary of main ideas and results}\label{sec 1.1}

Throughout the paper, we denote by $(X_t)_{t\ge 0}$ the isotropic symmetric  $2s$-stable L\'evy process (see Definition \ref{stable}  for a precise definition).
In what is below much of the terminology and notions are from \cite{kim2021stefan}.
For the optimization problem to be stated below, we consider costs of the form
\begin{equation} \label{Cost} \tag{1.2} \mathcal{C}(\tau) := \mathbb{E} \left[ \int_0^{\tau} L(s,X_s) ds \right], 
\end{equation}
 where $\tau$ denotes a randomized stopping time (see Definition \ref{randomized} for details) and $L(t,x)$ is a  non-negative bounded function that is strictly monotone in time. We  say that $L$ is {\it of Type (I) (resp. Type (II))} if $L$ is strictly increasing (resp. strictly decreasing) in time. 
\medskip

For a non-negative function $f \in L^{\infty}(\R^d)$ and a non-negative, bounded, compactly supported, and absolutely continuous measure $\mu$ on $\R^d$,  we consider the following optimization problem introduced in \cite{kim2021stefan}:
\begin{equation} \label{free} \tag{1.3}
    \mathcal{P}_f(\mu) := \inf_{ (\tau,\nu) } \{ \mathcal{C}(\tau) : X_0 \sim \mu, \ X_{\tau} \sim \nu \text{ and } \nu(\cdot) \leq f(\cdot) \},
\end{equation}
where $\nu \ll \text{Leb}$ and $\nu(\cdot)$ represents its Lebesgue density, and $\nu$ has an active compact region (see Definition \ref{active_region}). 
This is a free-target variant of the more classical problem,  called the optimal Skorokhod problem in the literature, given as 
\begin{equation} \label{Var_Skoro} \tag{1.4} \mathcal{P}_0(\mu,\nu) := \inf_{ \tau \in \mathcal{T}_r(\mu,\nu) } \mathcal{C}(\tau),
\end{equation} 
where  
$\mathcal{T}_r(\mu,\nu)$ is the set of randomized stopping times $\tau$ that satisfies  $ X_0 \sim \mu$ and $ X_{\tau} \sim \nu$, which stays less than or equal to  $\tau_r := \inf\{t \ge 0: X_t \notin B_r(0)\}$. Here, the randomized stopping time is a relaxed version of stopping time by introducing an external randomization, which will be precisely defined in Definition \ref{randomized} later. The notion of a randomized stopping time has appeared in the literature 
\cite{baxter,rost,beiglbock2017optimal}.

\medskip

Let us briefly describe the literature in the local case ($s=1$), i.e. $(X_t)_{t\ge 0}$ is the Brownian motion. 
In this case, it was shown that the optimal stopping time of \eqref{Var_Skoro} can be determined by solving the local parabolic obstacle problem (see \cite{gassiat2015root} for $d=1$, and \cite{ghoussoub2019pde,dweik2021stochastic} for $d>1$). For more general processes, we refer to \cite{ghoussoub2021optimal,gassiat2021free, rost1976skorokhod}.
The connection between the free target problem  $\mathcal{P}_f(\mu)$ and the Stefan problems has been established in \cite{kim2021stefan} for the local case $s=1$.

 \medskip
 
We extend the analysis of \cite{kim2021stefan,ghoussoub2019pde} for the case $s=1$ to construct solutions of the \emph{nonlocal} Stefan problems, i.e. $0<s< 1$,  using the free target problem   $\mathcal{P}_f(\mu)$. Below we will illustrate the similarity and novelty of the results in the nonlocal case.

\medskip

As mentioned above, the presence of unbounded jumps in the stable processes has made the study of its Skorohord embedding problem $\mathcal{P}_0(\mu, \nu)$ in \eqref{Var_Skoro} more challenging compared to that of the Brownian motion. For instance, conditions under which  one can find a stopping time $\tau$ with $\mathbb{E}[\tau]<\infty$ such that $X_0 \sim \mu$ and $X_{\tau} \sim \nu$ is well known for Brownian motion (see for instance \cite{obloj2004skorokhod} or \cite[Theorem 1.5]{ghoussoub2020optimal}). In the case of one dimensional L\'evy processes, necessary and sufficient conditions for $\mathbb{E}[\tau]<\infty$   were derived in the recent paper \cite{doring2019skorokhod}. On the other hand, by using potential theory,  \cite{gassiat2021free} established a connection for general Markov processes between the Skorohord embedding problem and the fractional parabolic obstacle problem for Type (I), though a similar analysis for Type (II) appears to be  much more challenging.  

\medskip

Under a series of compactness assumptions, the Skorohord embedding problem  $\mathcal{P}_0(\mu, \nu)$ was considered for a class of general processes in \cite{ghoussoub2021optimal}, using a dual problem approach. In particular, by restricting ourselves to target measures $\nu$ that have a \emph{compact active region}, we are able to apply their results to characterize the optimal stopping time as a hitting time to  a certain space-time barrier $R$ (with some potential randomization at the initial time for Type (II)). We further proceed to show that the barrier is closed, with the aid of potential theory and elliptic regularity. 

 \medskip
 
We first state the assumptions required for our results. Throughout this paper, we assume that Lagrangian $L$ in the cost function, $f$, and $s$ satisfy certain conditions (see  Assumptions~\ref{assume1}, \ref{assum2}, and \ref{assume3} for details), and the initial measure $\mu$ is given by $\mu_0dx$, where $\mu_0\in L^{\infty}(\R^d)$ has a non-empty and compact support. Also with a slight abuse of notation, we interchangeably refer to a measure and its corresponding Lebesgue density. With these assumptions, we were able to obtain the following result, a nonlocal extension of the
similar result in \cite{kim2021stefan}.

\begin{theorem}[Theorem \ref{barrier_existence}, Theorem \ref{exist_unique_Pf}, Theorem \ref{saturation}, and Theorem \ref{nu_identify}]\label{th:free-target} Once it has an admissible solution, the variational problem
$\mathcal{P}_f(\mu)$ in \eqref{free} has a unique optimal stopping time $\tau$ and an optimal target measure $\nu$ with bounded density. Moreover, 
\begin{itemize}
    \item[(a)] $\tau$ is the hitting time  to  a certain space-time barrier $R$ (with a possible randomization on $\{\tau=0\}$ for Type (II)).
    \item[(b)] $R$ is closed after a modification on a set of space-time measure zero.
    \item[(c)]  $\nu$ is independent of the choice or type of the cost.
    \item[(d)] $\nu = f$ in the set $\{x\in \R^d: (t,x)\notin R \hbox{ for some } t>0$\}.  
    \item[(e)] $\nu = \mu - (-\Delta)^s u$, where $u$ solves the elliptic obstacle problem \eqref{initial_set_expansion}. 
    \end{itemize} \label{Pf_solutions}
\end{theorem}

A key step for proving $(b)$ in Theorem \ref{Pf_solutions} is to characterize  the set $R$ as a space-time barrier generated by the potentials, which is a coincidence set of the  parabolic obstacle problem for both Type (I) and Type (II), similar to the strategy taken by   \cite{kim2021stefan,gassiat2021free}. 

This characterization of the barrier set would provide a connection between Stefan problems and parabolic obstacle problems as we will see below. In particular, the relation between parabolic obstacle problems for Type (II) and \eqref{Var_Skoro} appears to  be new  for $2s$-stable processes. For Type (I),  $(b)$ still provides an improvement for the result of \cite{gassiat2021free} where the barrier is shown to be only \emph{finely} closed. The saturation property $(d)$ is important: it will be used along with a characteristic function  $f$ to derive a solution of the Stefan problem.
 
\medskip

To proceed with the PDE formulation of the problem, we {use} the {\it Eulerian variables} associated with initial data $\mu$ and a stopping time $\tau$,  introduced in \cite{ghoussoub2019pde}. Namely we denote $\eta$ and $\rho$ to be the distribution of respectively active particles and stopped particles associated to $(\mu,\tau)$ (see Section \ref{Eulerian_variable_section}, in particular Lemma \ref{Eulerian_Variables}, for the full definition). Let us mention that, in terms of the barrier set $R$ in Theorem \ref{Pf_solutions}, $\{\eta>0\}$ coincides with $R^c$ (see Figure \ref{fig:Stopped_Active_Region} and Lemma \ref{open_barrier}).

\begin{figure}
    \centering
    \begin{minipage}{.5\textwidth}
        \centering
        \includegraphics[width=\textwidth]{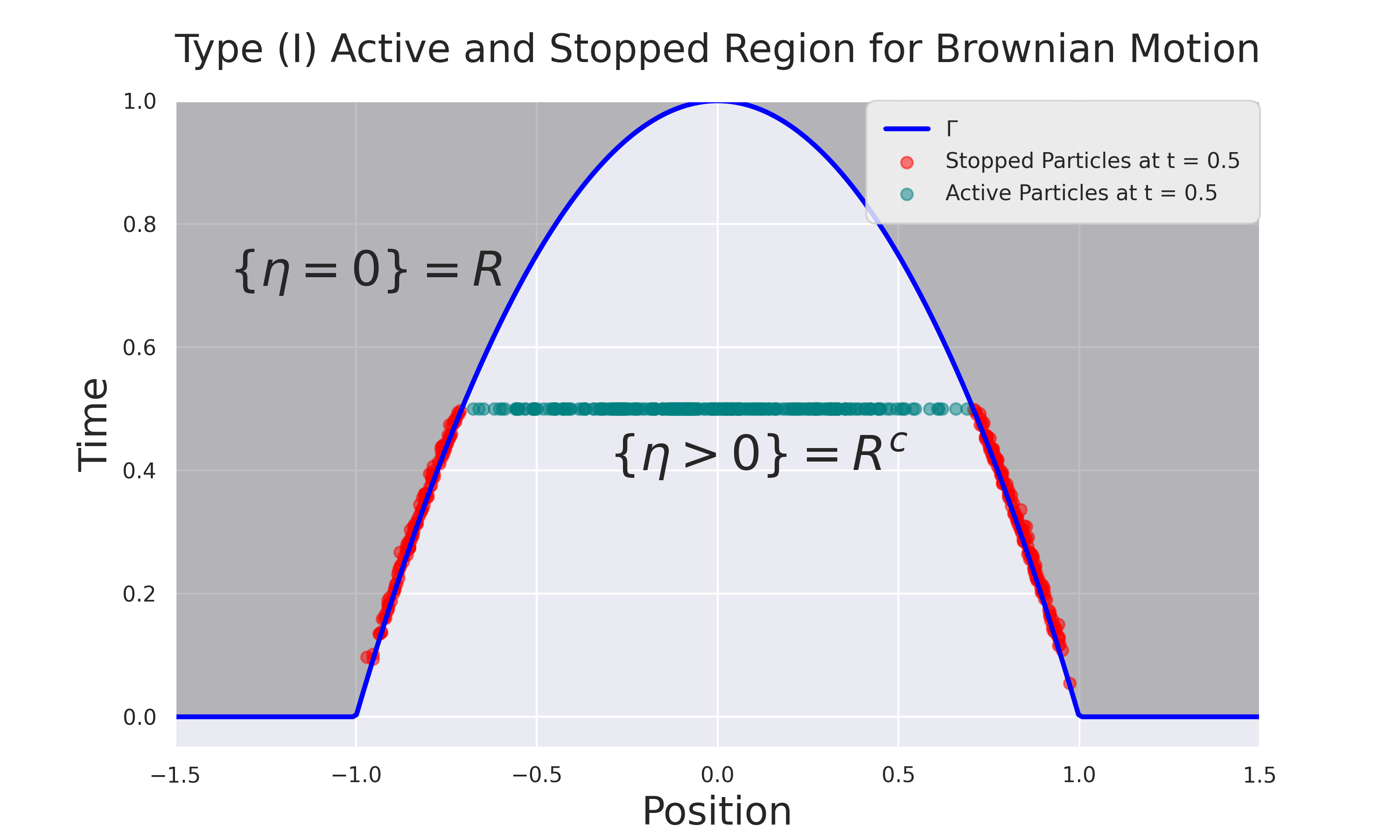}
        \caption*{(a) Brownian motion}
    \end{minipage}%
    \hfill
    \begin{minipage}{.5\textwidth}
        \centering
        \includegraphics[width=\textwidth]{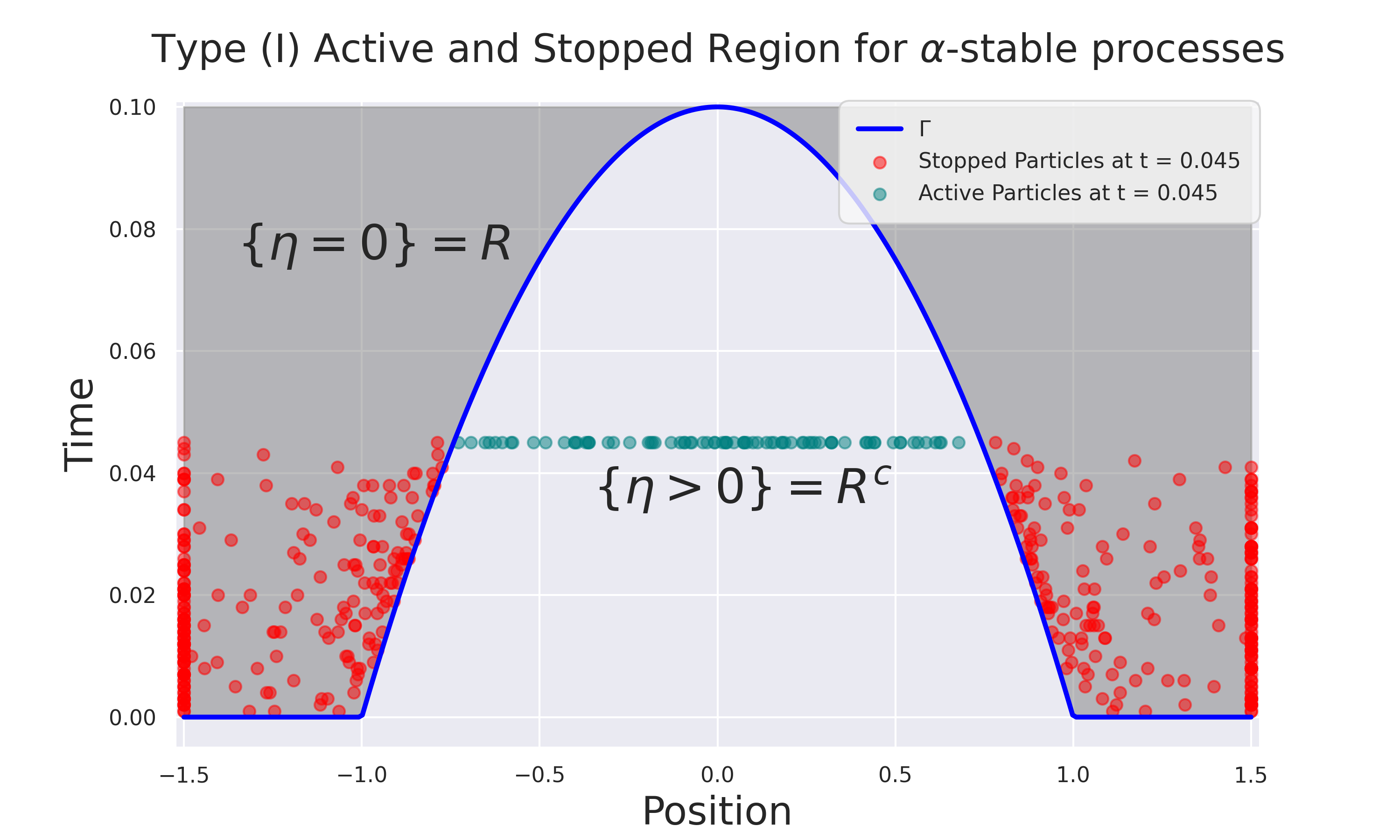}
        \caption*{(b) $\alpha$-stable process}
    \end{minipage}
    \caption{Comparison of the Brownian motion (a) and $\alpha$-stable process (b). Each dot represents $(X_{t \wedge \tau}, t \wedge \tau)$, where $\tau$ is the first entry time to the ice region $R$, and particles are normalized to the maximum norm of 1.5. Observe that $\rho$ is supported on $\Gamma = \partial R$ for the Brownian motion, whereas $\rho$ is supported on $R$ for the $\alpha$-stable processes.}
    \label{fig:Stopped_Active_Region}
\end{figure}

\medskip

Let us first discuss our result for \eqref{St2}.  As mentioned above, this problem has well-posed weak solutions for the enthalpy-based form of the equation: 
\begin{equation} \tag{$St_h$} \label{Ste}
    \p_t h +  (-\Delta)^{s} (h-1)_+ = 0.
\end{equation} 
For the notion of our weak solutions, we refer to Definition \ref{weighted_stefan_def} - \ref{def_St} for \eqref{St2} and to Definition \ref{enthalpy_weighted_def} for  \eqref{Ste}. We will show that our unique weak solution $\eta$ for the temperature-based equation \eqref{St2}  generates the unique weak solution $h$ of \eqref{Ste}. Moreover, thanks to the probabilistic interpretation of $h$ (Theorem \ref{stefan_enthalpy}), we provide a new exponential convergence result for \eqref{Ste}.

\begin{theorem}[Melting Stefan problem \eqref{St2}: Theorem \ref{melting_solve}, Theorem \ref{associated_obstacle_melt}, and Theorem~\ref{stefan_enthalpy}]\label{meltingtheorem}
Suppose that $\mu>1$ on its support. Let $(\nu,\tau)$ be the pair of optimal target measure and optimal stopping time for $\mathcal{P}_f(\mu)$ with $f \equiv 1$ and with a Type (II) cost.  Let $(\eta,\rho)$  be the Eulerian variables associated to $(\mu,\tau)$, and define $h(t,x) := \eta(t,x) + \rho([0,t),x)$. Then the following statements hold:
\begin{itemize}
\item[(a)] The pair $(h,\eta)$ is the unique weak solution to  \eqref{St2} with initial data $h(0,x) = \mu(x)$ and $\eta(0,x) = (\mu(x)-1)_+$, with the property that the set $\cap_{t>0} \{\eta(t,\cdot)>0\}$  coincides with the support of $\mu$. 
\item[(b)]  $w(t,x) := \int_0^t \eta(s,x) ds$  is the unique weak solution to   the parabolic obstacle problem
\begin{equation}\label{obstacle_p} \tag{1.5}
    \min \{ \p_t w + (-\Delta)^s w + 1 - \mu, w \} = 0 \text{ in } (0,\infty) \times \R^d \text{ with } w(0,\cdot)=0. 
    \end{equation}
 \item[(c)]  $h$ is the unique weak solution to \eqref{Ste} with initial data $\mu$. 

 \item[(d)] There exist  constants $\gamma,C>0$ depending only on $\mu,d$ and $s$ such that for any $t \ge 0,$
\[ ||h(t,\cdot)-\nu||_{L^1(\R^d)} \leq Ce^{-\gamma t}.  \]

     \end{itemize} \label{melting_theorem} 

\end{theorem}\vspace{-5mm}


\vspace{0mm}
The key challenge that we face in the nonlocal case is to track the distribution of $\rho$, as mentioned earlier. See Section~\ref{sec:Stefan} for more discussion on this and a comparison to the local case. The connection  to the parabolic obstacle problem (Theorem~\ref{melting_solve})  is new for $0<s<1$: we hope that the regularity analysis of \eqref{obstacle_p} provides a better understanding on \eqref{St2}, see Section~\ref{connection_stefan} for the further discussion on this.

\medskip

Note that our choice of initial data in Theorem~\ref{meltingtheorem} covers the general initial data considered in the literature with no initial {\it mushy region}, namely $\{x\in \R^d: 0<(h-\eta) (0,x)<1\} = \emptyset$. This restriction comes from our approach where we view $h$ as a function of $\eta$ and $\rho$, which does not account the initial mushy region.
 \medskip

For Type (I) cost, while we can address the general initial data $\mu$, the initial domain (or the {\it initial trace}) $E:=\bigcup_{t>0} \{\eta(t,\cdot)>0\}$ cannot be arbitrarily chosen.  Indeed  we show that for a given initial data $\mu$ and a set $G$ that contains $\{x\in \R^d: 0<\mu(x) \leq 1\}$, there is a unique solution $\eta$ of \eqref{St1} that generates $E$ and it matches $G$ as its {\it insulated region}, namely 
$$
\Sigma(\eta):=\bigcap_{t > 0}\{\eta(t,\cdot)>0\} = G.
$$
We will see that $E$ is always larger than the support of $\mu$, and strictly larger in most cases. Hence our solutions go through an instant expansion of its active region at initial time, from which the freezing process starts.

\begin{figure}[ht]
    \centering
    \includegraphics[width=0.40\textwidth]{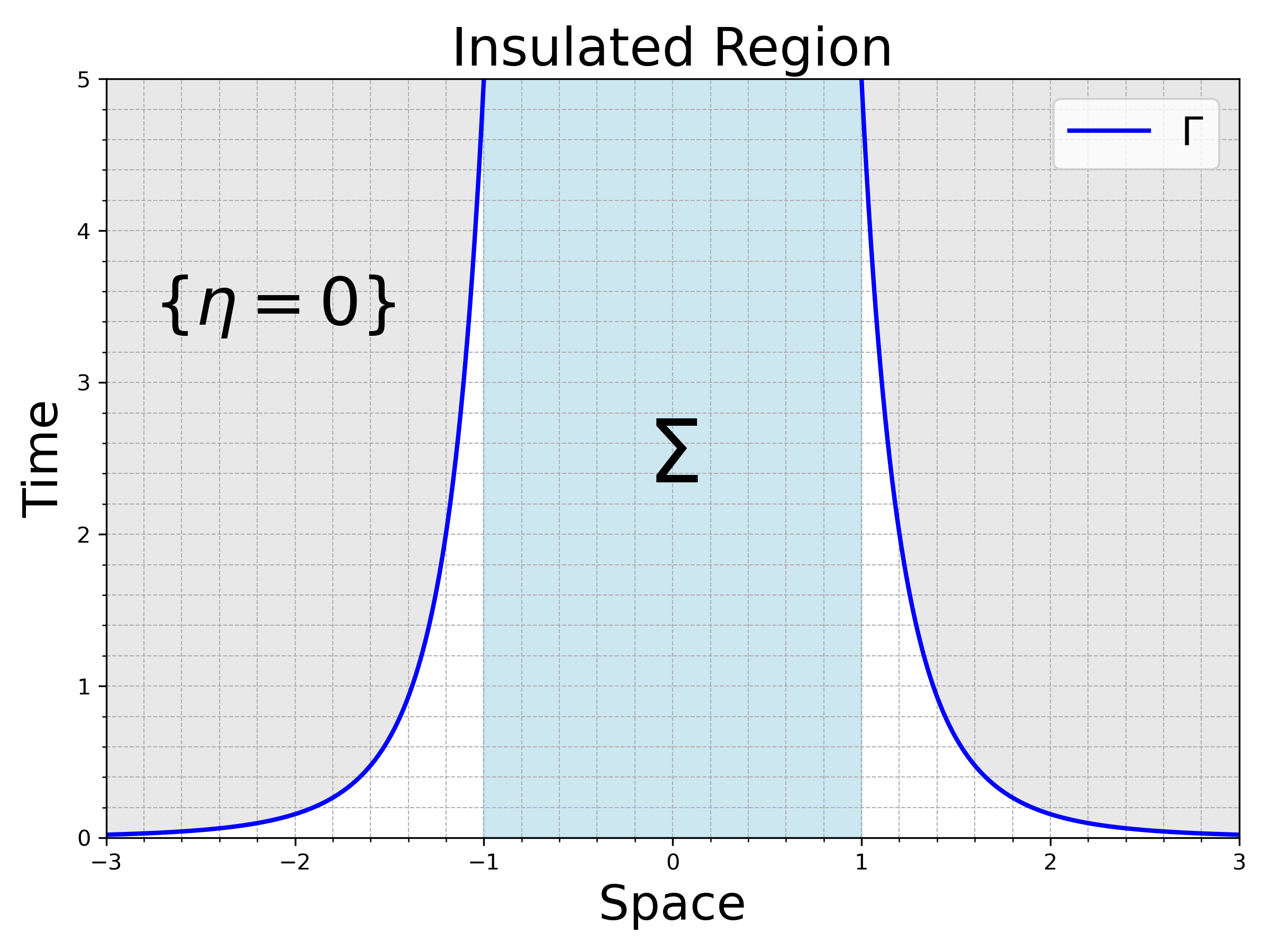}
    \caption{An illustration of the insulated region.}
    \label{fig:insulated_region}
\end{figure}

We now state our well-posedness result for the supercooled nonlocal Stefan problem (see Definition \ref{weighted_stefan_def} -\ref{def_St} for the notion of our weak solutions \eqref{St1}):
\begin{theorem}[Freezing Stefan Problem \eqref{St1}: Theorem \ref{associated_freezing_obstacle}, Theorem \ref{Solve_Insulated}, Theorem \ref{uniqueness_St1}, and Corollary \ref{vanish_finite_time}]\label{freeze:theorem} 
Let $G$ be a bounded open set in $\R^d$ that contains  $\{x\in \R^d: 0 < \mu(x) \leq 1\}$.
Let $(\eta, \rho)$ be as in Theorem \ref{melting_theorem}, but with a Type (I) cost and with $f:= 1-\chi_{G}$. Then the following statements hold.
\begin{itemize}
\item[(a)] The pair $(h,\eta)$ is the unique weak solution of \eqref{St1}  satisfying $\eta(0,\cdot)=\mu$ and $\Sigma(\eta)=G$.
\item[(b)] The initial domain $E:=\cup_{t>0} \{\eta(t,\cdot)>0\}$ is given as a positivity set  of the solution to the elliptic obstacle problem \eqref{initial_set_expansion} with $f = 1-\chi_{G}$. Also, $E$ contains the support of $\mu$.
\item[(c)] $w(t,x):= \int_t^{\infty} \eta(s,x) ds$ solves the parabolic obstacle problem
$$
\min \{ \p_t w + (-\Delta)^s w +\nu  ,w\}=0 \hbox{ in } \R^+ \times \R^d \text{ with } w(0,\cdot) = (-\Delta)^{-s}(\mu-\nu),
$$  where $\nu$ is the optimal target measure for $\mathcal{P}_f(\mu)$ with $f = 1 - \chi_G$. 
\item[(d)] In particular, if $\mu$ is greater than $1$ on its support, then one can choose $G$ to be an empty set and $f\equiv 1$, for which $(h,\eta)$ is the unique weak solution of \eqref{St1} with initial data $\mu$ that  vanishes in a finite time.
\end{itemize} \vspace{-5mm}
\end{theorem}
It is worth pointing out that $G$ can be chosen without any topological assumption, whereas in the local case $G$ must be simply connected. This is because, due to the nonlocal diffusion, particles can jump and stop anywhere outside of the active region. 

\medskip

While we do not pursue showing further stability of this class of solutions, we expect parallel results hold as in the local problem \cite{kim2021stefan}, where stability and comparison principle holds for the corresponding solutions with respect to perturbation of $\mu$ and $\Sigma$.
Lastly we point out that, after completion of our manuscript, a global existence result for general initial measures $0 \le \mu<1$ and initial domains  is  recently achieved for the local Stefan problem \cite{ckk24}. Given this, it is likely that the corresponding result holds for our problem, but we do not explore it further here.

\medskip

{\bf Acknowledgement}  I.K and R.C are partially supported by NSF grant DMS-2153254. Y.K is partially supported by Natural Sciences and Engineering Research Council of Canada (NSERC), Discovery Grant (RGPIN-2019-03926), as well as  Exploration Grant NFRFE-2019-00944 from the New Frontiers in Research Fund (NFRF). Y.K is a member of the Kantorovich Initiative (KI) that is supported by PIMS Research Network (PRN) program. We thank PIMS for their generous support.
K.N is partially supported by  the National Research Foundation of Korea (NRF-2019R1A6A1A10073887, RS-2019-NR040050).  We would like to thank Xavier Ros-Oton and Hugo Panzo for helpful discussions. 
 Part of this work is completed during Y.K’s visit at KAIST, and he thanks KAIST's hospitality and excellent environment.

\section{Preliminaries} \label{sec2}

Let us start with reviewing preliminary results that we will use in this paper.
We first define $\alpha$-stable L\'evy  processes (see \cite{applebaum2009levy} for more details about L\'evy  processes).

\begin{definition} [$\alpha$-stable L\'evy 
 process] \label{stable}
A continuous–time process 
$X= (X_t)_{t\ge 0}$  taking values in $\R^d$ is called a \emph{L\'evy process} if the following holds:
\begin{itemize}
    \item Its sample
paths are right-continuous and have left limits at every time.
\item   It has
stationary and independent increments.
\end{itemize}  
For $0<\alpha<2,$ the  L\'evy  process $X= (X_t)_{t\ge 0}$
 with initial value $X_0=0$ is called  \emph{$\alpha$-stable process} if it satisfies the self-similarity property  
\begin{align*}
 X_t\sim t^{1/\alpha} X_1,\quad \forall t>0.
\end{align*}
 (we say that $X\sim Y$ if $X$ and $Y$ have the same law). Furthermore, the $\alpha$-stable process is called \emph{symmetric} if $X_t\sim -X_t$ for any $t>0.$  
\end{definition}
If $X= (X_t)_{t\ge 0}$ is a symmetric $\alpha$-stable L\'evy  process,  
then there exists a constant $c > 0$ such that for any $\theta\in \R$,
\begin{align}
    \mathbb{E} [e^{i \theta X_t}] = e^{-ct|\theta|^\alpha},\quad \forall t > 0.
\end{align}
We say that  $X= (X_t)_{t\ge 0}$ is an \emph{isotropic} symmetric $\alpha$-stable L\'evy  process if $c=1$. Notice that Brownian motion corresponds to the case $\alpha=2$ and $c=\frac{1}{2}$.

Throughout the paper, we set the parameter $s\in (0,1)$ as $s := \alpha/2$. We work with the isotropic symmetric $2s$-stable L\'evy  process (from now on, we abbreviate it to the  $2s$-stable  process) $X= (X_t)_{t\ge 0}$, where $s$ satisfies 
\begin{align} \label{s}
    \begin{cases} s \in (0,\frac{1}{2}) &\text{ if } d = 1, \\ s \in (0,1) &\text{ else}.
    \end{cases}
\end{align}
The above assumption is due to the fact that $X_t$ is a recurrent process if $d=1$ and $s  \ge  1/2$  (see \cite{bertoin1996levy}). We restrict ourselves to the regime where  $X_t$ is transient because we will use the associated potential (see Section \ref{barrier_property}) to analyze $X_t$ and transisence is a typical assumption when one works with potentials \cite[Remark 3.9]{gassiat2021free}.

\medskip

Now we assume the same assumptions as in \cite{ghoussoub2021optimal} about our probability space. Let   $(\Om,\mathcal{F}, (\mathcal{F}_t)_{t \ge 0}, \mathbb{P})$ be a  filtered probability space  which  supports the  $2s$-stable  process 
$X= (X_t)_{t\ge 0}$ in $\R^d$.
We assume that  $\Om$ is a Polish space and $X : \Om \rightarrow D(\R^+;\R^d)$ (here we set $\R^+ := [0,\infty)$) is a continuous map onto the Skorokhod space of c\`{a}dl\`{a}g paths (i.e. $\{X_t(\omega)\}_{t \in \R+}$ is right continuous with left limits) with respect to the  Skorokhod path space metric. We also  assume that the filtration   $(\mathcal{F}_t)_{t\ge 0}$  is right continuous and complete such that the process $X$ is adapted to the filtration.

\medskip

{For our further connections with a nonlocal Stefan problem, we allow  the law of the initial state $X_0$ to be any finite Radon measure on $\R^d.$ First for any $x\in \R^d$, we denote by $ (X^x_t)_{t\ge 0} $  the  process initially starting from $x$ (i.e. translation of  the  $2s$-stable process by $x$). Now for any finite Radon measure $\mu$ on $\R^d$, $X_0\sim \mu$ means that the law of the process $(X_t)_{t\ge 0}$ is given by $\int  \mathbb{P}( (X^x_t)_{t\ge 0} \in \cdot)  d\mu(x)$}.

\medskip

A random variable $\tau : \Om \rightarrow \R^+$ is called a \emph{stopping time}  if  $\tau^{-1}([0,t]) \in \mathcal{F}_t$ for  any $t \geq 0$. Stopping times can be regarded as a rule that prescribes when each particle of the process stops moving. Also the particles
 are determined 
to stop based 
 only on  the information in the present or past. 
 
\medskip

It turns out that a relaxed notion of the stopping time, called a \textit{randomized stopping time}, is useful to analyze the Skorokhod's embedding problem. Randomized stopping time assigns a probability distribution to each particle on when the particle should stop moving. Let $\mathcal{M}_1(\R^+)$ be the collection of probability measures on $\R^+.$

\begin{definition} [Randomized stopping time]  \label{randomized} (See \cite{baxter,rost,beiglbock2017optimal}.)
A randomized stopping time $\tau : \Om \rightarrow \mathcal{M}_1(\R^+)$ is a measure-valued random variable such that 
for any $\omega \in \Om$,  $\tau_{\omega} := \tau(\omega)$ is a  probability measure on  $\R^+$ and 
for each $t \geq 0$,
\[ \omega \in \Omega \mapsto \tau_{\omega}([0,t]) \text{ is } \mathcal{F}_t\textup{-measurable}. \]
We let  $\mathcal{S}$ 
denote
the collection of randomized stopping times. 
\end{definition}

For a randomized stopping time $\tau$ and $\varphi :\R^+ \to \R,$ we interpret 
\[ \mathbb{E}[\varphi (\tau)] = \mathbb{E} \left[ \int_{\R^+} \varphi(t) \tau(dt) \right]. \]

\subsection{Fractional Subharmonic Ordering}

For $s\in (0,1)$, the fractional Laplacian  $(-\Delta)^{s}$ is a non-local generalization of the standard Laplacian.
For  Schwartz functions $\psi$ of rapid decay,
\begin{equation} (-\Delta)^{s} \psi(x) := C_{d,s} \text{P.V.} \int_{\R^d} \frac{\psi(x)-\psi(y)}{|x-y|^{d+2s}}  dy = C_{d,s} \lim_{\e \rightarrow 0} \int_{\R^d \setminus B_{\e}(x) }  \frac{\psi(x)-\psi(y)}{|x-y|^{d+2s}} dy,
\label{frac_laplacian}  \end{equation} 
where  $|\cdot|$ denotes the $\ell^2$-norm and   the normalizing constant $C_{d,s}$ is chosen so that $(-\Delta)^{s}$ is a Fourier multiplier by $|\xi|^{2s}$ (see \cite[Proposition 3.3]{di2012hitchhiker} for details).
We state the integration by parts formula for the fractional Laplacian: for    Schwartz functions  $\psi$ and $\varphi$,
\[ \int_{\R^d} \varphi \cdot  (-\Delta)^s \psi = \int_{\R^d} (-\Delta)^{\frac{s}{2}} \varphi \cdot (-\Delta)^{\frac{s}{2}} \psi = \int_{\R^d} \psi \cdot (-\Delta)^s \varphi,  \] which can be easily deduced with 
an application of  Plancherel's theorem.

\medskip

The (negation of) fractional Laplacian  $-(-\Delta)^{s}$ is the  infinitesimal generator of the   $2s$-stable   process $X = (X_t)_{t\ge 0}$ (see \cite{applebaum2009levy}):
\begin{align*}
    -(-\Delta)^{s}\psi(x) = \lim_{t\to 0} \frac{\mathbb{E} [\psi(X^x_t)] - \psi(x)}{t} .
\end{align*}

Next, we define  $s$-subharmonic functions. {We denote by $ C_b^2(\R^d)$ the collection of functions $f\in C^2(\R^d)$ such that  $f,\nabla f, \nabla^2 f$ are all bounded.}
\begin{definition}[Fractional subharmonic]  A function $\varphi \in C_b^2(\R^d)$ is called a  \emph{$s$-fractional subharmonic} function if $(-\Delta)^{s} \varphi \leq 0$ on $\R^d$.
\end{definition}

In addition, we define the ordering of measures with respect to the fractional Laplacian $(-\Delta)^{s}$. This notion of ordering is crucially related to  
the existence of a randomized stopping time in the Skorokhod embedding problem.
\begin{definition}[Fractional Subharmonic Ordering] \label{deffso}
Two finite Radon measures $\mu$ and $\nu$  on $\R^d$ are in \emph{$s$-fractional subharmonic ordering} if for any $s$-fractional subharmonic function $\varphi \in C_b^{2}(\R^d)$,  
\[ \int_{\R^d} \varphi(x) \mu(dx) \leq \int_{\R^d} \varphi(x) \nu(dx). \] We denote this relationship as {$\mu \leq_{s\textup{-SH}} \nu$}. Note that this in particular implies $\mu(\R^d)=\nu(\R^d)$.
\end{definition}

\noindent  It is known that the fractional subharmonic ordering ensures the existence of a randomized stopping time (see \cite[Theorem 4.3]{gassiat2021free} and \cite[Theorems 1 and 3]{rost1976skorokhod}).

\begin{proposition}[Theorems 1 and 3 in \cite{rost1976skorokhod}]  \label{prop 2.1}
Assume that $\mu$ and $\nu$ are finite Radon measures on $\R^d$ that satisfy $\mu \leq_{s\textup{-SH}} \nu$. Then,  there exists a randomized stopping time $\tau$ such that $X_0 \sim \mu$ and $X_{\tau} \sim \nu$. In addition,  $\tau$ can be chosen so that it has a ``minimal residual expectation". That is, if $\sigma$ is any randomized stopping time such that  $X_0 \sim \mu$ and $X_{\sigma} \sim \nu$, then  \label{subharmonic_order}
\[ \mathbb{E}[F(\tau)] \leq \mathbb{E}[F(\sigma)] \]
for any non-decreasing convex function $F: \R^+ \to \R$.  
\end{proposition}

Conversely,  the existence of a randomized stopping time $\tau$ with $\mathbb{E} [\tau] < \infty$ satisfying $X_0 \sim \mu$ and $X_{\tau} \sim \nu$ implies $\mu \leq_{s\textup{-SH}} \nu$.
To see this,   by  It\^{o}'s formula (see Proposition \ref{ito} below), 
for any $s$-fractional subharmonic function $\varphi \in C^2_b(\R^d)$,
\[ \int_{\R^d} \varphi(x) \mu(dx) \leq \int_{\R^d} \varphi(x) \nu(dx). \] This indeed implies that $\mu \leq_{s\textup{-SH}} \nu$. 

\begin{proposition} [It\^{o}'s Formula] \label{ito} 
For any  $\varphi \in C^2_b(\R^d)$  and 
 a randomized stopping time $\tau$ with $\mathbb{E}[\tau]<\infty$,
\[ \mathbb{E}[\varphi(X_{\tau})] =  \mathbb{E}[\varphi(X_0)] + \mathbb{E} \left[ \int_0^{\tau} -(-\Delta)^s \varphi(X_s) ds \right].  \] 
\end{proposition}

\begin{remark} We remark that the above It\^{o}'s formula is more commonly referred to as Dynkin's formula in the literature and that it holds for more general Feller processes (see \cite[Exercise 6.29]{le2016brownian} or \cite[Lemma 19.21]{kallenberg1997foundations}).
\end{remark}

\section{Optimal Skorokhod problem with compact active region}

In this section, we introduce precise settings for the variational problem \eqref{Var_Skoro}, where one of the notable new components lies in considering the particular class of  stopping times $\tau$.  In the nonlocal case $0<s<1$,  the randomized stopping time $\tau$ from Proposition \ref{prop 2.1} could satisfy $\mathbb{E}[\tau] = \infty$.  We refer to \cite{doring2019skorokhod} for discussions of necessary and sufficient conditions for the integrability  of $\tau$ in the particular case $d=1$. As we will see in Section \ref{Eulerian_variable_section}, we  need a stronger condition  $\mathbb{E}[\exp(\gamma \tau)]<\infty$ ($\gamma>0$ is some constant) to obtain enough compactness in the time variable of the space-time process $(t \wedge \tau,X_{t \wedge \tau}  )$.

\medskip

These observations motivate us to introduce the following class of stopping times:

\begin{definition}[Compact Active Region]\label{active_region} For  a randomized stopping time $\tau$, we say that a pair $(X,\tau)$, which induces a  process $(X_{t \wedge \tau})_{t\ge 0}$, has a \emph{compact active region}  if for some $r>0$,
\begin{align} \label{stopping ball}
    \tau \leq \tau_r:=  \inf \{t \ge 0 : X_t \notin B_r(0)\},
\end{align}
where $B_r(0) :=\{x \in \R^d: |x|<r \}$.
We interchangeably say that $\nu$ has a compact active region  if there exists $r>0$ and a  {randomized}
stopping time $\tau\le \tau_r$ such that  $X_{\tau} \sim \nu$. For such $r$, we say that $(X,\tau)$ or $\nu$ has a compact active region in $B_r(0)$. 
\end{definition}

Intuitively speaking, $\tau\le \tau_r$ implies that particles are only allowed to actively move from inside $B_r(0)$, and even though they are allowed to jump to outside of the ball they  must stop  immediately once  they  arrived  outside $B_r(0)$.

\begin{remark}[Exponential Moments of $\tau_r$]\label{exp_moments_exit_time}
The condition $\tau \leq \tau_r$ implies that $\tau$ not only has a  finite expectation but also has finite exponential moments, as long as $\mu$ is a finite measure. Indeed, by \cite[page 82]{bogdan2009potential},  if $\lambda > 0$ denotes the principal eigenvalue of the Dirichlet problem
\begin{equation}  \label{poincare}
\begin{cases}  (-\Delta)^{s} u = \lambda u &\textup{ in } B_r(0),
\\ u = 0 &\textup{ on } B_r(0)^c,
\end{cases} 
\end{equation} then 
\begin{equation} \lambda = -\lim_{t \rightarrow \infty} \frac{1}{t} \log\mathbb{P}(\tau_r > t | X_0 = x),\quad  \forall x \in B_r(0). \label{tail_tau_r} \end{equation}
Also, $\mathbb{P}(\tau_r = 0 | X_{0}=x)=1$ if $x \notin B_r(0)$. Therefore, for any $\gamma \in (0, \lambda)$, 
\[ \mathbb{E}[ \exp(\gamma \tau_r) ] = \int_{\R^d} \mathbb{E}[ \exp(\gamma \tau_r) | X_0=x] d\mu(x) \leq \int_{\R^d} \mathbb{E}[ \exp(\gamma \tau_r) | X_0=0] d\mu(x)  < \infty.   \]  We refer the reader to Appendix \ref{Hs_review} for more information about $\lambda$ and its relation to a Poincare's inequality for the fractional Laplacian.

\end{remark}

\begin{remark}[Fractional Subharmonic Ordering in $B_r(0)$] For $\mu \leq_{s\textup{-SH}} \nu$, if $\nu$ has a compact active region in $B_r(0)$, then   for any $\varphi \in C^2_b(B_r(0)) \cap C_b(\R^d)$ such that $(-\Delta)^s  \varphi \leq 0$ in $B_r(0)$,
\[ \int_{\R^d} \varphi(x) \mu(dx) \leq \int_{\R^d} \varphi(x) \nu(dx), \] due to It\^{o}'s Formula. This is a stronger condition than a fractional sub-harmonic ordering in the sense of Definition  \ref{deffso}, since fractional sub-harmonic functions on $B_r(0)$ are not necessarily fractional sub-harmonic on $\R^d$.
   
\end{remark}

Next, we show that the target measure $\nu$ associated with stopping times having a compact active region possesses a nice decay property.  In the  local case, the compact active region condition is equivalent to the target measure being compactly supported.  Let us first  consider the following example, showing that the distribution of $X_{\tau_r}$  (see \eqref{stopping ball} for the definition of $\tau_r$) is not compactly supported, though it obviously has a compact active region $B_r(0)$.

\begin{example} \label{exit_ball_density}
 For $s\in (0,1)$ satisfying  the condition \eqref{s}, let $r>0$ and $X = (X^x_t)_{t\ge 0}$ be a $2s$-stable process with  $X_0 \sim \delta_x$ and  $|x| < r$. Define the exit time $\tau_r^x := \inf \{t \geq 0: X_t^x \notin B_r(0) \}$. Then the distribution of $X_{\tau^x_r}$ is absolutely continuous with respect to the Lebesgue measure on $\R^d$, whose  density   is given by (see \cite{blumenthal1961distribution})
 \begin{align} \label{hm ball}
     \nu_r^x(y) = C_{r,s,d} \left( \frac{r^2-|x|^2}{|y|^2-r^2} \right)^{s} \frac{1}{|x-y|^d} \chi_{  \{|y| > r\}}.
 \end{align}
  In addition, if $X_0 \sim \mu$, where $\mu$ is bounded and has a compact  support contained in $B_{r/2}(0)$, then the density $\nu_r$ of $X_{\tau_r}$ is given by
\begin{equation}
	\nu_r(y) = \int_{\R^d} \nu_r^x(y) d\mu(x) \leq C  \frac{1}{ (|y|^2-r^2)^{s}}   \frac{1}{(|y|-\frac{r}{2})^d }   \chi_{ \{|y| > r\} } \in L^1(\R^d).\label{nu_R_bound}
\end{equation} 
This in particular implies
\begin{align} \label{decay}
    \nu_r(y) = O( |y|^{-2s-d}), \quad |y|\to \infty.
\end{align}  
\end{example}

Now, we state a decay property on target measures having a compact active region in $B_r(0)$.

\begin{lemma} 
Let $X = (X_t)_{t\ge 0}$ be a $2s$-stable process and $\tau$ be a randomized stopping time which has a compact active region in $B_r(0)$ (i.e. $\tau \le \tau_r$).  Let $X_{\tau} \sim \nu$, then $\nu \leq \nu_r$ on  $B_r(0)^c$, where $\nu_r$ denotes the distribution of $X_{\tau_r}$ from Example \ref{exit_ball_density}. \label{nu_decay}
\end{lemma}

\begin{proof} 
Let $A \subset B_r(0)^c$ be any Borel set. Under the event $X_{\tau} \in A$,  by the definition of $\tau_r$, we have $\tau\geq\tau_r$, which together with the condition $\tau\leq\tau_r$ implies  $\tau=\tau_r$. Hence, $\{\omega: X_{\tau} \in A\} \subset \{ \omega : X_{\tau_r} \in A\},$  which implies the desired bound.
\end{proof}

From now on, we work with measures that are absolutely continuous with respect to the Lebesgue measure and we will abuse notation by using the same symbol for both the measure and density.
For $M,R>0$ and a finite Radon measure $\mu$ on $\R^d$, define
\begin{equation} \mathcal{A}_{\mu,M,R} := \{ \nu \in \mathcal{M}(\R^d)  : \mu \leq_{s\textup{-SH}} \nu, \  \nu \text{ has a compact active region in } B_R(0) \text{ and } \nu \leq M  \} \label{A_mu_e_R}\end{equation}
($\nu \le M$ means that $\nu$ has a density w.r.t. Lebesgue measure, which  is  bounded by $M$).

 \begin{lemma}  
 Let $M>0$ and $\mu$ be a  compactly supported  measure on $\R^d$ such that  $\mu \in L^1(\R^d)$. Then, $\mathcal{A}_{\mu,M,R} \neq \emptyset$ for large enough $R>0$ (depending only on $M$, the support of $\mu$, and the $L^1$ norm of $\mu$). In particular, there exists $\nu \in \mathcal{A}_{\mu,M,R}$ such that $\nu$ is supported on $B_R(0)^c$ for large enough $R>0$. \label{non_empty_P_f}
 \end{lemma}
 In the case of Brownian motion starting from the origin, one can construct an element of $\mathcal{A}_{\mu,M, R}$ using the fact that $\delta_0$ is in subharmonic order with the uniform probability measure on any annuli centered at 0. We extend this argument to the case of  $2s$-stable processes.
 
 \begin{proof}

 First let us consider the case when  $\mu=\delta_x$ for some $x\in \R^d$. 
 Let $R>0$ such that $|x| \leq R/2$.
By It\^{o}'s formula,  for any $s$-fractional subharmonic function $\varphi \in C^2_b(\R^d)$ and $r\ge R,$
\begin{equation} \label{FSO_1} 
\varphi(x) \leq \mathbb{E}[\varphi(X_{\tau_r^x})] =  \int_{\R^d}  \varphi(y) \nu_r^x(y) dy
\end{equation} 
(recall that $\nu_r^x$ is defined in \eqref{hm ball}).
By averaging in $r$ over $[R,2R],$
\begin{equation}\label{FSO_2} 
\varphi(x) \leq 
\frac{1}{R}\int_{\R^d}  \int_{R}^{2R}  \varphi(y) \nu_r^x(y) dr dy.
\end{equation}
Hence, using the formula \eqref{hm ball}, setting
\begin{equation} \tilde{\nu}_R^x(y)  : = \frac{1}{R} \int_{R}^{2R} \nu_r^x(y) dr=\frac{C_{s,d}}{R} \int_{R}^{2R} \left( \frac{r^2-|x|^2}{|y|^2-r^2} \right)^{s} \frac{1}{|x-y|^d} \chi_{\{|y|>r\} }  dr, \label{tilde_nu_x} \end{equation} 
we have $\delta_x  \leq_{s\textup{-SH}} \tilde{\nu}_R^x$. 

~
We show that for any $\e>0$,  $\tilde{\nu}_R^x \leq \e$  for large enough $R>0$. Let us divide into the cases  $|y| \geq 4R$ and $R \leq |y| \leq 4R$ (by \eqref{tilde_nu_x}, $\tilde{\nu}^x ( y)=0$ for $|y| < R$).
First when $|y| \geq 4R$, using the fact $|x-y|^d \geq ||y|-|x||^d \geq 3^dR^d$ and $|y|^2-r^2 \geq 12R^2$ for $r \in [R,2R]$,
we see from \eqref{tilde_nu_x} that
\begin{equation} \sup_{|y| \geq 4R} \tilde{\nu}_R^x(y) \leq C \frac{1}{R} \int_{R}^{2R}  \frac{R^{2s}}{ R^{2s+d} } dr = C R^{-d}. \label{nu_large_y}  \end{equation}  
In addition when $R \le |y| \leq 4R$, using $|x-y|^d \geq | |y|-|x| |^d \geq R^d/2^d$, 
\begin{align}\label{nu_small_y}
    \tilde{\nu}_R^x(y) &\leq \frac{C}{R^{d+1} }  \int_{R}^{2R} \frac{R^{2s}}{ (|y|^2-r^2)^{s} } \chi_{  \{|y|> r\}}  dr \leq \frac{C}{R^{d+1}} \int_{R}^{|y|} \frac{R^{2s}}{ (|y|-r)^{s} (|y|+r)^{s} }  dr \nonumber \\& \leq  \frac{C}{R^{d+1-s}} \int_{R}^{|y|} \frac{1}{ (|y|-r)^{s} } dr = \frac{C}{R^{d+1-s}} (|y|-R)^{1-s}  \le CR^{-d}.
\end{align}
Therefore, by \eqref{nu_large_y} and \eqref{nu_small_y}, we conclude the following: 
\begin{equation}\label{claim_1}
\hbox{For any }\e>0, \quad \tilde{\nu}_R^x(y) \leq \e \hbox{ for } |x|\leq R/2, \hbox{ if } R\hbox{ is sufficiently large.} 
\end{equation}

\medskip

We now extend this to general compactly supported finite measures $\mu$.  Let $R>0$ be such that $\text{supp}(\mu) \subset B_{R/2}(0)$. Then,  by integrating \eqref{FSO_2} with respect to $\mu$, we deduce $\mu \leq_{s\textup{-SH}} \tilde{\nu}_R$ with  
\begin{equation} \tilde{\nu}_R(y)  := \int_{\R^d} \tilde{\nu}_R^x(y) d\mu(x)=  \frac{1}{R} \int_{R}^{2R} \int_{\R^d} \nu_r^x(y) d\mu(x) dr.\label{tilde_nu} \end{equation}
From \eqref{claim_1}, for any $\e>0$, we have $\tilde{\nu}_R(y) \leq\e ||\mu||_{L^1}$ for large enough $R>0.$ This yields $\tilde{\nu}_R \le M$  by taking  small $\e>0$. 

\medskip

Finally, we show that $\tilde{\nu}_R$ has a compact active region in $B_{2R}(0)$. Indeed, let $\tilde{\tau}$ be a randomized stopping time such that 
\begin{align} \label{uniform}
    \tilde{\tau} \sim   \tau_U,
\end{align}
where $U$ denotes the uniform distribution on $[R,2R]$ independent of the process.  Then we have $X_{\tilde{\tau}} \sim \tilde{\nu}_R$, since
for any Borel  $A \subset \R^d$,
\[ \mathbb{P}( X_{\tilde{\tau}} \in A) = \frac{1}{R} \int_{R}^{2R} \mathbb{P}(X_{\tilde{\tau}} \in A \mid  U=r) dr = \frac{1}{R} \int_{R}^{2R} \nu_r[A] dr \overset{\eqref{tilde_nu}}{=} \tilde{\nu}_R[A] . \] 
Since $\tilde{\tau} \leq \tau_{2R}$ by \eqref{uniform}, we  conclude the proof.

 \end{proof}

\subsection{Assumptions} In this section, we state our assumptions in the  optimal Skorokhod's embedding problem \eqref{Var_Skoro}. 
\begin{assumption} \label{assume1} We assume that $s \in (0,1)$ for $d \geq 2$ and $s \in (0,1/2)$ for $d=1$. This condition ensures that the $2s$-stable process  $X = (X_t)_{t\ge 0}$ is transient. 

In addition, the measures $\mu,\nu\in \mathcal{M}(\R^d)$ satisfy 
\begin{enumerate}
\item The source measure is non-trivial i.e. $\mu(\R^d)>0$.
    \item  $\mu \leq_{s\textup{-SH}} \nu$.
    \item   
$\mu$ and $\nu$  are absolutely continuous  with respect to the Lebesgue measure, whose densities satisfy $\mu \in L^\infty(\R^d)$ and $\nu \in L^{\infty}(\R^d)$.
\item  There exists $r>0$ such that $\mu$ is supported in $  B_r(0) $ and {$\nu$ has a compact active region in $ B_r(0) $}.

\end{enumerate}
\end{assumption}
By Proposition \ref{prop 2.1}, there exists a randomized stopping time $\tau$ such that $X_0\sim \mu$ and $X_\tau \sim \nu$. The last condition above implies that such $\tau $ satisfying $\tau \le \tau_r$ exists. Then by Remark \ref{exp_moments_exit_time}, $\tau$ has finite exponential moments. 

\medskip

The following is the assumption on our cost functional.
\begin{assumption} \label{assum2}
The Lagrangian $L$ satisfies
\begin{enumerate}
    \item $L \in C_b(\R^+ \times \R^d)$ is non-negative with $\p_t L \in C_b(\R^+ \times \R^d)$.
    \item $L$ is strictly monotone in time. We say that $L$ is of \emph{Type (I)} if $t \mapsto L(t,x) $ is strictly increasing for any $x$, and is of \emph{Type (II)} if it is strictly decreasing  for any $x$. 
\end{enumerate} 
\end{assumption}

Finally, we induce a constraint on the upper bound $f$ in the variational problem $\mathcal{P}_f(\mu)$ to ensure that the optimization set of $\mathcal{P}_f(\mu)$ is non-empty:

\begin{assumption} \label{assume3} We assume that $f \in L^{\infty}(\R^d)$ and  there exists $\delta>0$ and a compact set $K \subseteq \R^d$ such that $f(x) \geq \delta$ for $x \notin K$.
\end{assumption}

By Lemma \ref{non_empty_P_f},  the optimization set for the variational problem $\mathcal{P}_f(\mu)$ is non-empty under Assumption \ref{assume3}. 

\medskip

For the pair of measures $(\mu,\nu)$ and the Lagrangian $L$ satisfying Assumptions \ref{assume1} and \ref{assum2}, the existence and uniqueness of the optimal stopping time for the variational problem \eqref{Var_Skoro} follows from \cite{ghoussoub2021optimal}.

\subsection{Eulerian Variables} 
\label{Eulerian_variable_section}

In this section, we introduce {\it Eulerian variables} $(\eta, \rho)$ that correspond to the optimal Skorokhod problem \eqref{Var_Skoro}. These play a central role in connecting our problem with the PDE formulation later. The local case $s=1$ was 
 introduced in  \cite{ghoussoub2019pde}. While our analysis largely follow the local case, we present the full proof for the nonlocal case due to the subtle differences  of 
function spaces caused by jumps in the process.

\medskip

 Let   $(\mu,\nu)$ be a pair satisfying Assumption \ref{assume1}  and  $X = (X_t)_{t \geq 0}$ be the   $2s$-stable  process   with $X_0 \sim \mu$. From now on, we regard $r>0$ from Assumption \ref{assume1} as a fixed quantity and use the abbreviated notation $$B := B_r(0).$$ Let $\overline{B} $ be the (topological) closure of $B$. Define
\begin{align} \label{st}
    \mathcal{T}_r(\mu,\nu) := \{  \tau \in \mathcal{S} : X_0 \sim \mu, X_{\tau} \sim \nu \text{ and } \tau \leq \tau_r \}
\end{align}
(recall that $\mathcal{S}$ denotes the set of randomized stopping times defined in Definition \ref{randomized}, and $\tau_r$ denotes the exit time defined in \eqref{stopping ball}). Note that $X_{\tau} \sim \nu$ means that
\[
    \mathbb{E}[g(X_{\tau})] = \int_{\Om} \int_{\R^+} g(X_{t}(\omega)) \tau_{\omega}(dt) \mathbb{P}(d \omega) = \int_{\R^d} g(y) \nu(dy),\quad \forall g \in C_b(\R^d).
\] Observe that $\mathcal{T}_r(\mu,\nu)$ is a convex set. \\

Next, we define function spaces as in \cite{ghoussoub2019pde}. For $\gamma >0$, define 
$$C_{-\gamma}(\R^+ \times \overline{B}) :=\{w \in C(\R^+ \times \overline{B}): \sup_{x\in \overline{B}} e^{-\gamma t} |w|(t,x)  \to 0
 \text{ as } t \rightarrow \infty\},$$ with the norm   {$$||w||_{C_{-\gamma}(\R^+ \times \overline{B})} := \sup_{(t,x) \in \R^+ \times \overline{B}} |e^{-\gamma t} w(t,x)|.$$} 
In addition, let $C^{1,2}_{-\gamma}(\R^+ \times \overline{B})$ be the set of functions {in $C_{-\gamma}(\R^+ \times \overline{B})$}  whose first-order time derivative and second-order spatial derivative belong to $C_{-\gamma}(\R^+ \times \overline{B})$.

We also define some notation regarding the set of measures. For $\gamma >0$, let
  $$\mathcal{M}_{\gamma}(\R^+ \times \overline{B}):= \{\mu\in \mathcal{M}(\R^+ \times \overline{B}): \lim_{t\to \infty} e^{\gamma t} |\mu|(\overline{B}) = 0 \}. $$
One can see that the dual space of $ C_{-\gamma}(\R^+ \times \overline{B})$ is $\mathcal{M}_{\gamma}(\R^+ \times \overline{B})$.

\medskip

In addition, we define time-dependent weighted  Sobolev spaces. Let
 $$L^2_{\gamma}(\R^+ ; H^s_0(B)) := \Big\{ u : u(t,x) = 0 \text{ for } x \notin B \text{ and  } \int_{\R^+} e^{\gamma t} [u(t,\cdot)]^2_{H^s(\R^d)} dt < \infty\Big\} ,$$  where $[\cdot]_{H^s(\R^d)}$ denotes the Gagliardo seminorm (see Appendix \ref{hs_info} for more details about $H^s$). Lastly, define 
 $$
 \mathcal{X}:=\{f \in L^2_{-\gamma}(\R^+;H_0^s(B)):  \p_t f \in L^2_{-\gamma}(\R^+;H^{-s}(B)) \},
 $$where $H^{-s}(B)$ denotes the dual space of $H^s_0(B)$. The associated   norm is  
 \begin{align} \label{x norm}
     ||f||^2_{\mathcal{X} } := \int_{\R^+} e^{- \gamma t } \left[ ||f(t,\cdot)||^2_{H^s_0(B)} + || \p_t f(t,\cdot)||^2_{H^{-s}(B)} \right].
 \end{align}

From now on we assume that $\gamma \in (0, \lambda)$, where $\lambda$ is defined in Remark \ref{exp_moments_exit_time}. 

\medskip

We now introduce the notion of \emph{Eulerian Variables} $(\eta, \rho)$.

 \begin{lemma}[Eulerian variables]\label{Eulerian_Variables}
For any $\tau \in \mathcal{T}_r(\mu,\nu),$  there exist $\eta \in \mathcal{M}_{\gamma}(  \R^+ \times \overline{B} )$ and $\rho \in \mathcal{M}( \R^+ \times \R^d )$ such that for any  $\varphi \in C_{-\gamma}(\R^+ \times \overline{B})$ and $\psi \in C_{c}(\R^+ \times \R^d)$,
\begin{equation} \begin{cases} \mathbb{E}\left[ \int_0^{\tau} \varphi(s,X_s) ds \right] = \iint_{ \R^+ \times \overline{B}} \varphi(s,x) \eta(ds,dx), \\ \mathbb{E}[ \psi(\tau,X_{\tau}) ] = \iint_{ \R^+ \times \R^d} \psi(s,x) \rho(ds,dx) .
\end{cases}\label{def_eta_rho}  \end{equation} 
{Also, $\eta$ is absolutely continuous with respect to the space-time Lebesgue measure.}

In addition for any test function $w \in C^{\infty}_c( \R^+ \times B)$,
\begin{equation} \iint_{ \R^+ \times B  } w(t,x) \rho(dt,dx) = \int_{B} w(0,x) \mu(dx) + \iint_{\R^+ \times B } \left( \p_t w(t,x) - (-\Delta)^s w(t,x) \right) \eta(dt,dx) \label{non_local_heat_weak}. 
\end{equation} 
Hence if $\tau>0$ almost surely, then $(\eta,\rho)$ solves as distributions
\begin{equation} \label{non_local_heat} \begin{cases} \partial_t \eta + \rho =  -(-\Delta)^s \eta  &(t,x)\in  (0,\infty) \times B, \\ \eta(t,x) = 0 & (t,x)\in  (0,\infty) \times B^c, \\ \eta(0,x) = \mu(x)\quad & {x\in \R^d}.  \end{cases} 
\end{equation}

\end{lemma}

\begin{proof} 
 Since $\tau \leq \tau_r$, $\tau$ has finite exponential moments (see  Remark \ref{exp_moments_exit_time}). Hence,  $\varphi \mapsto \mathbb{E} \left[ \int_0^{\tau} \varphi(s,X_s) ds \right]$ is a continuous linear functional on $C_{-\gamma}(\R^+ \times \overline{B})$. Thus, by  Riesz representation, there exists   $\eta \in \mathcal{M}_{\gamma}(\R^+ \times \overline{B})$ such that for any  $\varphi\in C_{-\gamma}(\R^+ \times \overline{B})$,
\[ \mathbb{E} \left[ \int_0^{\tau} \varphi(t,X_t) ds \right] = \iint_{\R^+ \times \overline{B}} \varphi(t,x) \eta(dt,dx). \] 
As $\tau \leq \tau_r$, {$\eta$ is supported on $\R^+ \times \overline{B}$.} 

\medskip

{Now we show that $\eta$ is absolutely continuous with respect to the space-time Lebesgue measure. Let $A \subset \R^+ \times \R^d$ be a null set, then by Fubini's theorem, its time section $A_t := \{ x\in \R^d : (t,x) \in A \}$ satisfies $|A_t|=0$ for $t$-a.e. Also \eqref{def_eta_rho} along with Fubini's theorem imply that  
\[ \eta[A] = \int_0^{\infty} \mathbb{P}[ t<\tau, (t,X_t) \in A ] dt = \int_0^{\infty} \mathbb{P}[ t<\tau, X_t \in A_t ] dt.  \] 
Since the law of $X_t$ is absolutely continuous with respect to the spatial Lebesgue measure and $|A_t|=0$, we have $0=\mathbb{P}[ X_t \in A_t] \geq \mathbb{P}[t<\tau,X_t \in A_t]$ for a.e. $t \geq 0$. Hence, $\eta[A]=0$, which verifies the absolutely continuity.} 

\medskip

Next, define $\rho \in \mathcal{M}(\R^+ \times \R^d)$ to be the law of $(\tau,X_{\tau})$, i.e. for any $\psi \in C_c(\R^+ \times \R^d)$,
\[ \mathbb{E}[\psi(\tau,X_{\tau})] = \iint_{\R^+ \times \R^d} \psi(t,x) \rho(dt,dx). \]
Then, \eqref{non_local_heat_weak} is obtained as a direct consequence of It\^{o}'s formula: 
\begin{equation} \mathbb{E} \left[ w(\tau,X_{\tau}) \right] = \int_{\R^d} w(0,x) \mu(dx)  + \mathbb{E} \left[ \int_0^{\tau} \left[ \partial_t w(t,X_t) - (-\Delta)^s w(t,X_t)\right]  dt \right]. \label{Ito_Eulerian} \end{equation} 

{Finally to establish \eqref{non_local_heat} under the condition $\tau>0,$ it remains to obtain \eqref{non_local_heat_weak} with a time interval $\R^+$ replaced by $(0,\infty)$, i.e. we rule out the possibility of  integrals in \eqref{non_local_heat_weak}  over $\R^+ \times \R^d$ having a mass at $t=0$.}
Observe that
\begin{equation} \iint_{\R^+ \times \R^d} w(t,x) \rho(dt,dx) = \iint_{ (0,\infty) \times \R^d } w(t,x) \rho(dt,dx) + \int_{\R^d} w(0,x) \rho(\{0\},dx). \label{rho_time} \end{equation} 
Since $\tau>0$ a.s., $\rho(\{0\} \times \R^d)=\mathbb{P}(\tau=0)=0$, implying that the last term above is zero. 
 Recalling that {$\eta$ is absolutely continuous with respect to the space-time Lebesgue measure, from  \eqref{non_local_heat_weak}, we obtain}
{\[ \iint_{ (0,\infty) \times B  } w(t,x) \rho(dt,dx) = \int_{B} w(0,x) \mu(dx) + \iint_{(0,\infty)  \times B } \left( \p_t w(t,x) - (-\Delta)^s w(t,x) \right) \eta(dt,dx).
\]} 
{This implies that $\eta(0,\cdot)=\mu(\cdot)$ in the sense of distributions, and thus we deduce \eqref{non_local_heat}.} 
 
\end{proof}

\begin{remark}  \label{eul_initial_data}
Note that \eqref{def_eta_rho} defines $\eta$ and $\rho$ respectively as the distribution of active and stopped particles associated with $\tau$. 
If $\tau=0$ with a positive probability, then the initial data in \eqref{non_local_heat} has to be revised as $\mu(x) - \rho(\{0\},x)$. The instantly stopped distribution, $\rho(\{0\},\cdot)$, is defined in Definition \ref{disintegration}.   
\end{remark}

\begin{lemma}[Parabolic regularity]\label{regularity_lemma}
Let $(\eta,\rho)$ solve \eqref{non_local_heat} in the sense of distributions. Then for any $\gamma \in (0, \lambda)$, $(\eta,\rho) \in L^2_{\gamma}(\R^+;H_0^s(B)) \times \mathcal{X}^*$. Moreover there is $C=C(\gamma,d,s,B)$ such that 
\[ ||\eta||_{L^2_{\gamma}(\R^+;H_0^{s}(B))} \leq C||\mu||_{L^2(\R^d)} \text{ and } ||\eta(t,\cdot)||_{L^2(B)} \leq e^{- \gamma t} ||\mu||_{L^2(\R^d)}. \]

\end{lemma}

\begin{proof} 
As the above estimates are direct consequence of regularity theory of fractional heat equations, we only briefly sketch the proof (we refer to \cite{ghoussoub2019pde} in the case of heat equations). We first mollify $(\eta,\rho)$ in space-time to construct smooth approximations $\{(\eta_{\e},\rho_{\e})\}_{\e>0}$. Then we have $\p_t \eta_{\e} + \rho_{\e} = -(-\Delta)^s \eta_{\e}$ classically. Then we multiply this equation by $e^{\gamma t} \eta_{\e}$ and proceed with integration by parts. Afterwords, Poincare's inequality (see Appendix \ref{Hs_review}) yields the desired bounds in the limit $\e \to 0$.
\end{proof}

Let us finalize this section by characterizing the optimal stopping time associated to \eqref{Var_Skoro} as a hitting time to some space-time barrier set. The result follows from \cite{ghoussoub2021optimal}, see Appendix \ref{opt_sko_embed}.

\begin{definition} [Forward/backward Barrier] A set  $R \subset \R^+ \times \R^d$ is called a \emph{forward barrier} if $(t,x) \in R$, then $(t+\delta,x) \in R$ for any $\delta>0$. 	Conversely,  $R$ is called a \emph{backward barrier} if $(t,x)\in R$, then $(t-\delta,x) \in R$ for any $\delta \in [0,t]$.
\end{definition}

\begin{theorem}[Optimal Stopping time as a Hitting time] \label{barrier_existence}
$\empty$
Let $\mu,\nu \in \mathcal{M}(\R^d)$ satisfy Assumption \ref{assume1} and the Lagrangian $L$ satisfies Assumption \ref{assum2}. Then there is a unique optimal stopping time $\tau^*$ of \eqref{Var_Skoro} that has a compact active region in $B_r(0)$ (where $r$ is from Assumption \ref{assume1}). Moreover the following holds for the   optimal stopping time  $\tau^*$: 
 \begin{itemize}
\item[1.] When L is Type (I), there exists a forward barrier $R_1^* = \{ (t,x) \in \R^+ \times \R^d: t \geq s^*_1(x) \}$ with $s^*_1:\R^d \rightarrow \R^+$ measurable such that $\tau^* = \inf\{ t \ge 0: (t,X_t) \in R_1^* \}$. In particular, $\tau^*$ is a non-randomized stopping time (of the space-time process).

\medskip

\item[2.]When $L$ is Type (II), there exists a backward barrier $R_2^* = \{ (t,x) \in \R^+ \times \R^d : t\le  s_2^*(x)  \}$ with  $s_2^*:\R^d \rightarrow \R^+$ measurable such that  $\tau^* = \inf \{ t \ge 0 : (t,X_t) \in R^*_2 \}$ on $\{\tau^*>0\}$. In particular $\tau^*$ is only randomized on $\{\tau^*=0\}$.\end{itemize}
\end{theorem}
The proof of Theorem \ref{barrier_existence} will be presented in Appendix \ref{opt_sko_embed}.

\section{Properties of the Barrier Set}
\label{barrier_property}
In this section we show that the optimal barriers $R$ from Theorem \ref{barrier_existence} are in fact closed in the usual topology on $\R^+ \times \R^d$. Although the barrier for Type (I) has been explored in both local and non-local contexts in \cite{gassiat2021free}, our approach yields stronger regularity results on $R$ by leveraging the  elliptic regularity theory. While we use potential theoretic arguments similar to \cite{kim2021stefan,gassiat2021free}, our proof strongly utilizes the probabilistic interpretation of the Eulerian variables. \\

From now on, we assume that $s \in (0,1)$ satisfies the condition   \eqref{s}.

\subsection{Potential Flow} \label{pot_flow} 
Define the Riesz kernel of order $2s$
\begin{equation} \label{riesz}
	N(y) :=  C_{s,d} \frac{1}{|y|^{d-2s}},
\end{equation} 
where $C_{s,d}>0$ is a normalizing constant (see \cite{stinga2019user}).
Also for $m \in \mathcal{M}(\R^d)$,  define the potential of $m$ as
\begin{align}\label{U-m}
 U_m(y) := \int_{\R^d} N(x-y) m(dy).
 \end{align}
Then, $(-\Delta)^s U_m = m$.

In the next lemma, we state the spatial regularity of the potential $U_m$.

\begin{lemma} [Spatial regularity] \label{spatial_reg_1} Assume that   there exist $C,\e>0$ such that $|m(x)| \leq C/(1+|x|^{2s+\e})$ for all $x\in \R^d$. Then, $U_m \in C^{\alpha}(\R^d)$ for any $\alpha \neq 1$ with  $0 < \alpha < 2s$ and satisfies
$$||U_m||_{C^{\alpha}} \leq C (||U_m||_{L^{\infty}} + ||m||_{L^{\infty}}).$$ 
Here   for $\alpha\geq 1$, $C^{\alpha}$ denotes $C^{1, \alpha-1}$.
\end{lemma} 

\begin{proof} Note  that $(-\Delta)^s U_m = m \in L^{\infty}(\R^d)$ and the  decay conditions on $m$ implies $||U_m||_{L^{\infty}(\R^d)} \leq C$. By the elliptic regularity estimate \cite[Proposition 2.9]{silvestre2007regularity}
\[ ||f||_{C^{\alpha}(\R^d)} \leq C \left[ ||f||_{L^{\infty}(\R^d)} + ||(-\Delta)^s  f ||_{L^{\infty}(\R^d)} \right], \]
we conclude the proof.
\end{proof}

\medskip
 
From now on, for measures $\mu,\nu \in \mathcal{M}(\R^d)$ satisfying Assumption \ref{assume1}, we assume that $\tau \in \mathcal{S}$ is a randomized stopping time such that $X_0 \sim \mu$ and $X_{\tau} \sim \nu$. Let $(\eta,\rho)$ be the 
 Eulerian variables associated  to $\tau$.  Then, the distribution of $X_{t \wedge \tau}$, which we call $\mu_t$, can be  decomposed into the active and stopped mass up to time $t$, namely 
 \begin{align} \label{dec}
     \mu_t (x)= \eta(t, x) + \rho([0,t), x),
 \end{align}   
where $\rho([0,t),\cdot)$ is defined in Definition \ref{disintegration}. In the next lemma, we rigorously verify this decomposition.

\begin{lemma}[Characterization of $\mu_t$] 
{Let $\tau\in \mathcal{S}$ be such that $X_0 \sim \mu$ and $X_{\tau} \sim \nu$, and let $(\eta,\rho)$ be  the associated  Eulerian variables. For $t\ge 0,$ define $\mu_t$ to be the distribution of $X_{t \wedge \tau}$.}
Then, for any $\varphi \in C^{\infty}_c(\R^+ \times \R^d)$,
\[ \iint_{\R^+ \times \R^d} \varphi(t,x) \mu_t(dx) dt = \iint_{\R^+ \times \R^d} \varphi(t,x) \eta(t,x) dtdx + \iint_{\R^+ \times \R^d} \varphi(t,x) \rho([0,t],x)dtdx. \]  
 In other words,  \eqref{dec} holds $(t,x)$-a.e. 
\label{mu_t_formula}
\end{lemma}

\begin{proof} By the definition of $\mu_t$,  
\begin{align*}
     \int_{\R^+} \int_{\R^d} \varphi(t,x) \mu_t(dx) dt &= \int_{\R^+} \mathbb{E}[\varphi(t,X_{t \wedge \tau})] dt\\
 &= \mathbb{E} \left[  \int_{\R^+}  \varphi(t,X_{t}) \chi_{ \{t < \tau\}}  dt \right] +  \int_{\R^+}  \mathbb{E}[\varphi(t,X_{\tau}) \chi_{ \{\tau \leq t \}}] dt .
\end{align*}
By the definition of $(\eta,\rho)$ from Lemma \ref{Eulerian_Variables}, the above quantity is equal to 
\[  \iint_{\R^+ \times \R^d} \varphi(t,x) \eta(t,x) dtdx + \int_{\R^+} \left[ \iint_{\R^+ \times \R^d} \varphi(t,x) \chi_{ \{s \leq t \} } \rho(ds,dx) \right]dt \]
\[ = \iint_{\R^+ \times \R^d} \varphi(t,x) \eta(t,x) dtdx + \iint_{\R^+ \times \R^d} \varphi(t,x) \rho([0,t],x) dtdx.  \] In particular,  this implies that  $dt \otimes \mu_t(dx) $ is absolutely continuous w.r.t. Lebesgue measure on $\R^+ \times \R^d$. Finally we point out that  for any $x\in \R^d,$ $\rho([0,t),x)=\rho([0,t],x)$ $t$-a.e. due to  Lemma \ref{lambda_equality}, and thus we deduce \eqref{dec} for $(t,x)$-a.e.

\medskip

\end{proof}

\begin{definition}[Potential Flow] {For $\tau\in \mathcal{S}$ and $t\ge 0$, let $\mu_t$ be the distribution of $X_{t \wedge \tau}$.} Then,  $(U_{\mu_t})_{t\ge 0}$  (see \eqref{U-m})  is called the potential flow associated to $\tau$. \label{potential_flow}
    
\end{definition}
 
\begin{lemma} Let $\tau \in \mathcal{S}$ be such that $\tau \leq \tau_r = \inf\{t : X_t \notin B_r(0) \}$. Then, there exists  $C>0$ such that $\mu_{t}(x)  \leq  C(1+|x|^{d+2s})^{-1}$ for all $x\in 
 \R^d$ and $t\ge 0$. In particular, $\sup_{t\ge 0} ||U_{\mu_{t}}||_{C^{\alpha}(\R^d)} < \infty$ for any $0 < \alpha < 2s$ with $\alpha \neq 1$. \label{spatial_reg}
\end{lemma}

\begin{proof} First recall from Assumption \ref{assume1} that $\mu$ is compactly supported on $B:=B_r(0) $. Let $(\eta,\rho)$ be the Eulerian variables associated to $\tau$, as given in  Lemma \ref{Eulerian_Variables}. This then implies from \eqref{non_local_heat} and the comparison principle imply that $\eta \leq u$, where $u$ solves  
\[ \begin{cases} \p_t u = -(-\Delta)^s u &(t,x)\in  (0,\infty) \times B   , \\ u (t,x)= 0 &(t,x)\in (0,\infty) \times B^c , \\ u(0,x) = \mu(x) &x\in  B
\end{cases}\]  
(note that \eqref{non_local_heat}  is applicable even if $\tau=0$ with a  positive probability since  by Remark \ref{eul_initial_data}, the initial data of $\eta$ is at most $\mu$).
{By the comparison principle, $0 \leq u \leq ||\mu||_{L^{\infty}}$ and $u$ is compactly supported in $\overline{B}$. As $\eta\le u$, it follows that $\eta \leq ||\mu||_{L^{\infty}} \chi_B$.} 

Next, Lemma \ref{nu_decay}, \eqref{decay}, and the fact $\nu \in L^{\infty}$ implies that
\[ \rho([0,t),x) \leq \nu(x) \leq \frac{C}{1+|x|^{d+2s}} . \]
Thus Lemma \ref{mu_t_formula} yields that $\mu_{t}(x) \leq C(1+|x|^{d+2s})^{-1}$.  Therefore, as $\mu_t$ satisfies the condition in Lemma \ref{spatial_reg_1} uniformly in $t$, we conclude the proof.
\end{proof}

Parallel arguments to  \cite[Lemma 4.8]{kim2021stefan} yield a time regularity  for the potential flow:

\begin{lemma} For $\tau \in \mathcal{S}$ with $\mathbb{E}[\tau]<\infty$, let  $(U_{\mu_t})_{t\ge 0}$   be the associated potential flow. Then there exists $C>0$ such that for any $0 \leq t \leq t'$,
\[ -C(t'-t) \leq  U_{\mu_{t'}}(x) - U_{\mu_t}(x) \leq 0, \] 
in particular the potential flow is uniformly Lipschitz in time. \label{time_reg}
\end{lemma}

As a consequence of Lemmas \ref{spatial_reg} and  \ref{time_reg}, we obtain the following fact.
 
\begin{lemma} For $\tau \in \mathcal{S}$ with $\mathbb{E}[\tau]<\infty$,  the associated  potential flow $(U_{\mu_t})_{t\ge 0}$ satisfies 
\begin{enumerate}[label=(\alph*)]
	\item $\lim_{t \rightarrow 0+} U_{\mu_t} = U_{\mu}$ \text{uniformly}.
	\item $\lim_{t \rightarrow \infty} U_{\mu_t} = U_{\nu}$ \text{uniformly}.
\end{enumerate} \label{unif_limits}
\end{lemma}

We now introduce the  notion of forward and backward barrier associated to the potential flow. This  notion has been used in \cite{gassiat2015root} for the one-dimension space and in \cite{kim2021stefan, gassiat2021free} for general dimensions.

\begin{definition} [Barrier associated to potential flows] \label{barrier def}
 For  $\tau \in \mathcal{S}$, 
 let  $(U_{\mu_t})_{t\ge 0}$ be the associated potential flow. The forward/backward barrier functions and regions are defined as
	\[ s^{U,f}(x) := \inf \{ t \ge 0 : U_{\mu_t}(x) = U_{\nu}(x) \},\quad  R^{U,f} := \{ (t,x)\in \R^+ \times \R^d: t \geq s^{U,f}(x) \} \] and
 		\[ s^{U,b}(x) := \sup \{ t \ge 0 : U_{\mu_t}(x) = U_{\mu}(x) \} ,\quad  R^{U,b} := \{ (t,x)\in \R^+ \times \R^d: t \leq s^{U,b}(x) \}. \] 
 The corresponding
	stopping times are defined as
	\[ \tau^{U,f} := \inf \{t  \ge 0: U_{\mu_t}(X_t) = U_{\nu}(X_t) \} ,\quad   \tau^{U,b} := \inf\{t > 0: U_{\mu_t}(X_t) = U_{\mu}(X_t) \}.  \]    
   \end{definition}

Equivalently, by the monotonicity of $U_{\mu_t}$,
		\[ \tau^{U,f} = \inf \{t \ge 0: t \geq s^{U,f}(X_t) \},\quad \tau^{U,b} = \inf \{ t > 0 :  t \leq s^{U,b}(X_t) \}. \] 
By Lemmas \ref{spatial_reg} and \ref{time_reg}, $s^{U,f}$ and $-s^{U,b}$ are lower semi-continuous. Hence,  barriers $R^{U,f}$ and $R^{U,b}$ are closed in $\R^+ \times \R^d$.

\subsection{Potential analysis on the barrier}  We begin with our analysis on the barriers associated to  potential flows. 
Throughout this section, we assume that  $\tau \in \mathcal{S}$ satisfying  $X_0 \sim \mu$ with $X_{\tau} \sim \nu$ is  given by
\begin{align} \label{condition}
    \tau = \begin{cases}
        \inf \{t : t \geq s(X_t) \} &(\text{called Type (I)}), \\
      \inf\{ t : 0 < t \leq s(X_t) \} &(\text{called Type (II)}),
    \end{cases}
\end{align}
 for some measurable function $s: \R^d \rightarrow \R^+$.
The  terminologies Type (I) and Type (II) are inspired by  Theorem \ref{barrier_existence}. \\

  Our goal is to establish that the potential barrier induced by $\tau$ (in the sense of Definition \ref{barrier def}) are essentially the same as the original barrier generated by a function $s$.
By a parallel argument as in \cite[Proposition 4.12]{kim2021stefan}, we obtain the following result:

\begin{lemma} 
Let $\tau$ be as given in  \eqref{condition}.  Then 
\[ | \{ x\in \R^d : s(x) < s^{U,f}(x) \} | = 0 \text{ for Type (I) and }  |\{x\in \R^d :  s(x) > s^{U,b}(x) \}| = 0 \text{ for Type (II)}.  \] \label{pot_barrier_biggest}
\end{lemma}

In particular, Lemma \ref{pot_barrier_biggest} says that the potential forward/backward barrier is the largest such barrier that embeds $\mu$ to $\nu$. Using the short hand $\tau^U := \tau^{U,f}$ with $R^U := R^{U,f}$ for Type (I) and $\tau^U := \tau^{U,b}$ with $R^U := R^{U,b}$ for Type (II), we conclude the following:
\begin{corollary} \label{cor 4.7} 
Let $\tau$ be the optimizer of  $\mathcal{P}_0(\mu,\nu)$ in \eqref{Var_Skoro} (assume additionally that  $\tau>0$ a.s. for costs of Type (II)). Then 
$\tau^{U} \leq \tau$ a.s.  
\end{corollary} 
\begin{proof}
{By Theorem \ref{barrier_existence}, the optimizer $\tau$ is given by the space-time stopping time \eqref{condition} for some measurable function $s:\R^d \rightarrow \R^+$.} Let us first consider the Type (I) case. By Lemma \ref{pot_barrier_biggest} and recalling $X_\tau \sim \nu$ and $\nu \ll \text{Leb}$, we have    $s(X_{\tau}) \geq s^{U,f}(X_{\tau})$ a.s. {In addition, by Lemma \ref{stop_inside_barrier}, the stopped particles are inside the barrier, i.e. $\tau \geq s(X_{\tau})$ a.s.  In summary, $\tau \geq s^{U,f}(X_{\tau})$, implying that  $\tau \geq \tau^{U,f}$ a.s.}

The proof is similar for Type (II), except that when applying {Lemma \ref{stop_inside_barrier}}, we need the condition $\tau>0$ a.s.
\end{proof} 

With this result we can show that the active region is open:

\begin{lemma} \label{open_barrier}
Let $\tau$ be as given in  \eqref{condition} and $(\eta,\rho)$ be the associated Eulerian variables. Denote by $R^U$ the potential barrier in the sense of Definition \ref{barrier def}.  Then up to a zero-measure set (w.r.t. space-time Lebesgue measure),
\[ \{\eta>0\} = (R^U)^c. \] In particular, recalling that $R^U$ is closed, the positive set of $\eta$ is open (up to a zero-measure set).	 
\end{lemma}

\begin{proof} We follow the proof in \cite{kim2021stefan}.  Let
\begin{align} \label{w}
    w(t,x) :=  \begin{cases} U_{\mu_t}(x)-U_{\nu}(x) \quad \text{ for Type (I)}, \\ U_{\mu}(x) - U_{\mu_t}(x) \quad  \text{ for Type (II)}.\end{cases}
\end{align}
Then  by Lemma \ref{time_reg}, $w \geq 0$. Also  by definition of $R^U$, we have {$ (R^U)^c = \{w>0\}$}. Hence, it suffices to prove that up to a zero-measure space-time set,
\begin{align} \label{relation}
    \{\eta>0\}=\{w>0\}.
\end{align}

For a test function $g \in C^{\infty}_c( (0,\infty) \times \R^d)$, let $\varphi(t,\cdot):= N*g(t,\cdot)$ (recall that  $N$ denotes the Riesz kernel, see \eqref{riesz}), which satisfies $(-\Delta)^s  \varphi(t,\cdot)=g(t,\cdot)$ and decays at infinity. By Lemmas \ref{mu_t_formula} and \ref{diff_rho_t}, in the case of  Type (I),  
\[ \iint g ( \p_t w )  dtdx = \iint \varphi (\p_t \mu_t) dtdx = \iint \varphi(\p_t \eta + \rho ) dtdx= \iint -\varphi \cdot  (-\Delta)^s  \eta dtdx = \iint -g \eta  dtdx. \] By a similar computation for Type (II), along with the Lipschitz regularity   of $w$ (in time), 
\begin{equation} \p_t w = \begin{cases} -\eta & \text{ for Type (I)}, \\ \eta & \text{ for Type (II)}  .\end{cases}\label{w_time_deriv}
\end{equation}
Integrating in time and using Lemma \ref{unif_limits},
\begin{equation} 
w(t,x) = \begin{cases} \int_t^{\infty} \eta(s,x) ds &\text{ for Type (I)}, \\ \int_0^t \eta(s,x) ds &\text{ for Type (II)}.  
\end{cases}
\label{w_explicit} 
\end{equation}

We claim that this implies that $\{\eta>0\} \subset \{w>0\}$ up to a 
 zero-measure set. Indeed, for Type (I), recalling $\{w(t,x)=0\}=\{t \geq s^U(x)\}$, we have $0=w(s^U(x),x)=\int_{s^U(x)}^{\infty} \eta(t,x) dt$ for every $x\in \R^d$, implying that $\eta(t,x)=0$ for a.e. $t \geq s^U(x)$.  Hence,  we
 verify the claim.

\medskip   

 For the reverse direction, observe that by Lemma \ref{pot_barrier_biggest},  $\rho [ (R^U)^c]=0$. Hence $\eta$ solves the fractional heat equation on the open set $(R^U)^c$ with a zero boundary data and initial data $\mu$. By the strong maximum principle,  $\eta>0$ on $(R^U)^c \cap \{t>0\}$. Therefore, we establish \eqref{relation}.
\end{proof}

\begin{remark} \label{modififcation}
{By Lemma \ref{open_barrier}, one can modify $\eta$ on a set of space-time Lebesgue measure zero so that $\{\eta>0\}$ is open in $\R^+ \times \R^d$ (in particular, $  \{\eta>0\}=\{w>0\}$). From now on, we  work with this modification of $\eta$.}
\end{remark}

Let $\tau$ be the optimizer  of $\mathcal{P}_0(\mu,\nu)$. Our goal is to show that in the case of Type (I) cost,  $\tau=\tau^U$ a.s.  
{Once this is verified, one can assume that the optimal stopping time of $\mathcal{P}_0(\mu,\nu)$ in Theorem \ref{barrier_existence} is $\tau^U$, which allows us to abuse the notation} $s=s^U$ and $R=R^U$.
\begin{theorem} 
Let $\tau$ optimize $\mathcal{P}_0(\mu,\nu)$  with a Type (I) cost. Then $\tau=\tau^U$ a.s.
\label{equality_hitting_time_type_1}
\end{theorem}

\begin{proof}
By Corollary \ref{cor 4.7}, $\tau^U \leq \tau$. In particular, as $\tau \leq \tau_r = \inf\{t  \ge 0: X_t \notin B_r(0) \}$, we see that $(X,\tau^U)$ has an active compact region in $B_r(0)$.

Let  $(\eta^U,\rho^U)$ be the    Eulerian variables associated to $\tau^U$.  
  We claim that $\eta^U[ R^U ]= 0$ and $\rho^U[R^U]=1$. Indeed, by Fubini's Theorem,
\[ \eta^U[R^U] = \int_0^{\infty} \mathbb{P}(s < \tau^U, (s,X_s) \in R^U)ds.  \] 
Since $(s,X_s) \notin R^U$ for  $0\le s < \tau^U$, we have $\eta^U[R^U]=0$. Also, since $R^U$ is closed and $t\mapsto X_t$ is  right-continuous, $(\tau^U,X_{\tau^U}) \in  R^U$ a.s., which implies $\rho^U[R^U]=1$. 

Next, let $(\eta,\rho)$ be the Eulerian variables associated to $\tau$ from Lemma \ref{Eulerian_Variables}. By  Lemma \ref{open_barrier}, $\eta[R^U]=0$. In addition, by Lemma \ref{pot_barrier_biggest}, $\rho[R^U]=1$. Hence, by Theorem \ref{uniqueness}, we obtain that $\rho=\rho^U$. In particular, as $\rho \sim (\tau,X_{\tau})$ and $\rho^U \sim (\tau^U,X_{\tau^U})$, we conclude that $\tau \sim \tau^U$. Since $\tau^U \leq \tau$ by Corollary \ref{cor 4.7}, we deduce that $\tau=\tau^U$ a.s. 
\end{proof}

Before we discuss the corresponding result for Type (II), we introduce the PDE characterization of the potential variable $w$. This result will be important later in Section \ref{connection_stefan} where we derive the connection between the nonlocal Stefan problem and the parabolic obstacle problem for Type (II) (see Theorems \ref{associated_obstacle_melt} and \ref{associated_freezing_obstacle}). 
\begin{theorem}  \label{theorem 3.13}
Let $\tau$ optimize $\mathcal{P}_0(\mu,\nu)$  with a Type (I) or (II) cost (assume that $\tau>0$ a.s. in Type (II) case). 
Let $w$ be defined in \eqref{w}, where the potential flow is associated to $\tau$. Then, $w$ solves the following parabolic obstacle problem:
\begin{equation} \label{w_obstacle}
\begin{cases} \min \{ \p_t w + (-\Delta)^s w +\nu  ,w\}=0 \text{ with } w(0,\cdot) = U_{\mu} - U_{\nu} \quad & \text{ for Type (I)}, \\ \\ \min \{\p_t w + (-\Delta)^s w + \nu - \mu,w\} = 0 \text{ with } w(0,\cdot) = 0 &\text{ for Type (II)}. \end{cases}
\end{equation} 
\end{theorem}

\begin{proof} 
We just consider the Type (II) case since parallel arguments work for Type (I) case.   Let $(\eta,\rho)$ be the Eulerian variables associated to the optimal stopping time $\tau$.
By taking the fractional Laplacian in \eqref{w}, Lemma \ref{mu_t_formula} and \eqref{w_time_deriv} gives
\begin{equation}  \label{Theorem_4.13_identity} 
\p_t w(t,x) + (-\Delta)^s w(t,x) =  \mu(x)  - \rho([0,t),x). \end{equation}

By Lemmas \ref{spatial_reg_1} and \ref{time_reg},  $w$ is continuous. Also,  Lemma \ref{open_barrier}  (in particular, \eqref{relation}) and Lemma \ref{pot_barrier_biggest} yields that  if $ w(t,x)>0$ (equivalently, $  t > s^U(x) $),  then $\rho([0,t),x)=\nu(x)$. Hence,
\[ \p_t w + (-\Delta)^s w + \nu - \mu = 0 \quad \text{ on } \{w>0\}. \]
In addition,  since  $\nu(x) \geq \rho([0,t),x)$ for any $x\in \R^d$ and $t\ge 0$,  
\[ \p_t w + (-\Delta)^s w + \nu - \mu \geq \p_t w + (-\Delta)^s w + \rho([0,t),\cdot ) - \mu = 0, \]  where we used \eqref{Theorem_4.13_identity} in the final equality.
Thus we conclude that $w$ solves \eqref{w_obstacle}.

\end{proof} 

For Type (II), we present a direct  probabilistic  argument that does not rely on Theorem~\ref{uniqueness}, taking advantage of the nonlocal nature of  diffusion. While the argument is more intuitive and simpler, it achieves a slightly weaker result, namely we only show that $s=s^U$ a.e. Because of this reason we will not pursue our method for Type (I), which appears to require more careful analysis. We first state a consequence of Theorem~\ref{theorem 3.13} for Type (II).

\begin{corollary} Let $\tau$ be the optimal stopping time for $\mathcal{P}_0(\mu, \nu)$ with a Type (II) cost, and assume that $\tau>0$ a.s. For $t\ge 0$ and  $\delta>0$, assume that $A$ is a Borel set in  $\R^d$ of positive Lebesgue measure such that $A \subset \{ x\in \R^d: (t+\delta,x) \in R^U \}$. Then, $\rho( [t,t+\delta] \times A)>0$.
\label{rho_supported_RU}
\end{corollary}

\begin{proof} 
Let $w$ be the potential for Type (II), defined in  \eqref{w}.   Recalling that $w$ is non-decreasing and  Lipschitz in time (see Lemma \ref{time_reg}), \eqref{w_time_deriv} implies that $\p_t w = \eta$ a.e. As $\tau>0$, $\eta(0,\cdot) = \mu(\cdot)$. Since $\mu$ is non-trivial and  $\{\eta(t,\cdot)>0\}$ is non-decreasing in $t$, we have $\p_t w(t,\cdot) > 0$ on a set of positive measure for any   $t>0$. 

\medskip

If  $(t+\delta,x) \in R^U$ (i.e. $w(t+\delta,x)=0$), then $w(t,x)=0$. Thus,
\begin{equation} 
(-\Delta)^s w(t,x) = -C_{d,s} \int_{\R^d} \frac{w(t,y)}{|x-y|^{d+2s}} dy > -C_{d,s} \int_{\R^d} \frac{w(t+\delta,y)}{|x-y|^{d+2s}} dy = (-\Delta)^s w(t+\delta,x). \label{fsign} \end{equation} Note that $\p_t w = \eta = 0$ a.e. on $R^U$ and $\mu = 0$ a.e. on $\{ x\in \R^d : (t,x) \in R^U \text{ for some }t \}$. Hence by \eqref{Theorem_4.13_identity},  
\[ -\rho( [t,t+\delta),x) = (-\Delta)^s w(t+\delta,x) - (-\Delta)^s w(t,x) < 0,   \]
where  we used \eqref{fsign} in the last inequality. {Now by integrating over $A$, we conclude the proof.}
\end{proof}

\begin{theorem} Let $\mu,\nu$ and $\tau$ be as given in Corollary~\ref{rho_supported_RU} for a Type (II) cost. Then $R=R^U$ a.e. in $\R^+ \times \R^d$ and $s=s^U$ a.e. in $\R^d$. 
\label{equality_of_barrier}
\end{theorem}

\begin{proof} 
Note that $R = R^U$ a.e. implies that the barrier functions coincide, i.e. $s=s^U$ a.e. Hence it suffices to show $R=R^U$ a.e. Lemma \ref{pot_barrier_biggest} implies that $R \subset R^U$ a.e., and thus it remains to show that $|R^U \setminus R|=0$. 
By  {Lemma \ref{stop_inside_barrier}}, $\rho$ is supported on $R$, and thus $\rho( R^U \setminus R ) = 0$. 

\medskip

Assume for the sake of a contradiction that $|R^U \setminus R|>0$. Set  $Z := R^U \setminus R$ and define the time section $Z_t := \{x\in \R^d: (t,x)\in Z\}$. By Fubini's theorem, there is $t>0$ such that $|Z_t|>0$.
Since $R = \{(t,x) \in \R^+ \times \R^d: t \leq s(x) \}$ and $\{t\} \times Z_t \notin R$, we have $t>s(x)$ for $x\in Z_t$. Thus, recalling $|Z_t|>0$, there is  $n$ with $t>1/n$ such that $Z_t^{(n)} := \{ x \in Z_t : t - \frac{1}{n}  > s(x) \}$ has a positive Lebesgue measure. Since $[t-\frac{1}{n},t] \times Z_t^{(n)}  \subset R^U \setminus R$ and $\rho( R^U \setminus R ) = 0$, we have  $\rho( [t-\frac{1}{n},t] \times  Z_t^{(n)} ) = 0$, contradicting Corollary \ref{rho_supported_RU}. 
\end{proof}

Now by combining Theorems \ref{equality_hitting_time_type_1} and \ref{equality_of_barrier} along with our assumption that $\nu \ll \text{Leb}$, we conclude the following.

\begin{theorem} Under the assumptions of Corollary \ref{cor 4.7}, we have for both types that $s=s^U$ $\nu$-a.e. For Type (I), we further have that $\tau = \tau^U$ a.s. \label{equality_hitting_time}
\end{theorem}
 
We conclude this section by introducing the identity for $(-\Delta)^s w$, which will be useful when studying the associated Stefan problem in Section \ref{sec:Stefan}.

\begin{proposition}\label{w_s_laplacian_remark}
 Let $\tau$ and $w$ be as in Theorem~\ref{theorem 3.13} and let $(\eta,\rho)$ be the Eulerian variables associated to  $\tau$.
Then for any $t \geq 0$, 
\begin{equation}    \label{Type1Laplacian}
    \left[ (-\Delta)^s w(t,x) \right] \cdot  \chi_{ \{\eta=0\} }(t,x)   =    \rho((t,\infty),x) \cdot \chi_{ \{\eta=0\} }(t,x)  
\end{equation} 
for Type (I), and 
\begin{equation}
       -\left[ (-\Delta)^s w(t,x) \right]  \cdot  \chi_{ \{\eta=0\} } (t,x) =   \rho((0,t),x)  \cdot \chi_{ \{\eta=0\} }(t,x) \label{Type2Laplacian}
\end{equation}
for Type (II) (if $\tau>0$ a.s.).
\end{proposition} 

\begin{proof}  By  \eqref{w} and Lemma \ref{mu_t_formula},
\begin{equation} (-\Delta)^s  w(t,x) = \begin{cases}  \mu_t(x)-\nu(x)= \eta(t,x) - \rho((t,\infty),x)  \quad &\text{ for Type (I)},\\  \mu(x)-\mu_t(x) = \mu (x)-\eta(t,x)-\rho([0,t),x) \quad &\text{ for Type (II)}.
\end{cases}\label{w_s_laplacian} \end{equation} 
We just consider the Type (II) case since parallel arguments work for Type (I) case. 
Since $\tau>0$ a.s.,   $\rho([0,t),x) = \rho( (0,t),x)$ $x$-a.e. for any $t > 0 $. In addition, $\tau>0$ a.s. implies that  $\{\mu(\cdot)>0\} \subset \{\eta(t,\cdot)>0\}$ for any $t \geq 0$, since $\{\eta(t,\cdot)>0\}$ is non-decreasing in $t$. This implies that $(-\Delta)^s w(t,x) \cdot \chi_{ \{\eta=0\} }(t,x) = - \rho( (0,t),x) \cdot \chi_{ \{\eta=0\} }(t,x)$. 
\end{proof}

\section{Associated Stefan Problem}\label{sec:Stefan}
Suppose that $\tau$ is the optimal stopping time of the variational problem $\mathcal{P}_0(\mu,\nu)$ in \eqref{Var_Skoro}, and let $(\eta,\rho)$ be the  associated Eulerian variables  as in Lemma~\ref{Eulerian_Variables}.   In this section, we establish that $\eta$ solves the nonlocal Stefan problem with an initial distribution $\mu$ and the weight $\nu$. Before introducing the definition of solutions, let us present a heuristic discussion on the  characterization of the enthalpy variable in terms of the Eulerian variables. 

\medskip

$\circ$ {\it Review for the case $s=1$:} Let us briefly recall the proof for the local case $s=1$. In this case,  \eqref{St1} and \eqref{St2} can be written as 
\begin{equation}\label{eta}
\partial_t h -\Delta \eta = 0,
\end{equation}

\noindent where the enthalpy $h$ is given as a function of $\eta$ by  
\begin{align*}
        h = \begin{cases}
         \eta - \chi_{\{\eta>0\}} &\text{ for } (St_1), \\ 
         \eta + \chi_{\{\eta>0\}}  &\text{ for } (St_2).
    \end{cases}
\end{align*}
The discontinuity of $h$ represents the unit amount of heat energy change associated with the phase transition. To illustrate our interpretation  more clearly, we consider a generalized version of this problem, the equation \eqref{eta} that represents a {\it weighted} rate of energy change $\nu$, namely $h$ is given by 
\begin{align*}
        h = \begin{cases}
         \eta - \nu \chi_{\{\eta>0\}} &\text{ for } \eqref{St1nu}, \\ 
         \eta + \nu\chi_{\{\eta>0\}}  &\text{ for } \eqref{St2nu}
    \end{cases}
\end{align*}
(see below for the definition of  \eqref{St1nu} and \eqref{St2nu}).

It was shown  \cite{kim2021stefan}, following \cite{ghoussoub2019pde}, that in the case of the Brownian motion, this can be solved with the aid of optimal Eulerian variables $(\eta, \rho)$  of \eqref{St} associated to the target measure $\nu$. A crucial step in the analysis is the particle interpretation of the enthalpy variable  in terms of the distributions of particles at $(t,x)$. Precisely, \cite{kim2021stefan} deduced that  the Eulerian variables generated by the cost of Type (I)  and   Type (II) yield the  solutions to \eqref{St1nu} and \eqref{St2nu} respectively, with {the enthalpy}:
\begin{equation}\label{enthalpy}
    h(t,x) = \begin{cases}
         \eta(t,x) - \rho((t,\infty),x) &\text{ for } \eqref{St1nu}, \\ 
         \eta(t,x) + \rho([0,t),x) &\text{ for }  \eqref{St2nu},
    \end{cases}
\end{equation}
where $\rho((t,\infty),x):= \nu(x) - \rho([0,t],x).$ Note that this {formula for the enthalpy} is consistent with the {enthalpy formula} \eqref{enthalpy0} since $$\rho( [0,t),x) - \nu(x) = \rho( [0,t),x ) - \rho( [0,\infty),x) = -\rho([t,\infty),x) = -\rho( (t,\infty),x) ,$$ where the final equality holds $t$-a.e. thanks to Lemma \ref{lambda_equality}.   We remark that the choice of an open left end point on $\rho( (t,\infty),x)$ and  a closed right point on $\rho([0,t),x)$ is so that their values at time $t=0$ equals their limit as $t \rightarrow 0^+$, which will be used in our definition of weak solutions to the Stefan problem.

\medskip

In the local case, it is crucial that $\rho$ can be written in terms of $\eta$, for instance  $\rho([0,t),x)=\nu\chi_{\{\eta>0\}}$  for Type (II). This allows us to write the Eulerian PDE  solely in terms of $\eta$, which then leads to the characterization of the equation as the Stefan problem.

\medskip

$\circ$ {\it Heuristics for the nonlocal case.} Our task is to extend the characterization  of the enthalpy \eqref{enthalpy0} to the non-local case. We aim to represent $\rho$ in terms of $\eta$  using the Eulerian equation \eqref{non_local_heat} and the forward/backward time-monotonicity of the barrier set $\{ \eta =0\}$ for the Type (I)/(II) cases respectively.
 Our description of $\rho$ should now reflect the fact that the stopped particles are no longer concentrated on the boundary, due to the jumps in the $2s$-stable process (see Figure \ref{fig:Stopped_Active_Region}). In other words $\rho$  is  supported on the whole $\{\eta=0\}$ instead of only on the free boundary.

\medskip

We present the heuristic argument for the representation of $\rho$. We consider Type (II) with $\tau>0$. As $\tau$ is the   hitting time to the backward barrier {$R$ that is a.e. equal} to $\{\eta=0\}$  (see Theorem~\ref{barrier_existence} (2)), we have 
 $$\rho([0,t),x) = \rho([0,\infty),x)=\nu(x) \hbox{ on }\{\eta>0\}.
 $$
 On the other hand,  from Proposition \ref{w_s_laplacian_remark}, we have  
 $$-(-\Delta)^sw(t,x) = \rho( [0,t),x ) \hbox{ on } \{\eta=0\} \hbox{ with } w(t,x) = \int_0^t \eta (a,x) da.$$ 
 Putting these two identities together, we have that
\begin{equation}\label{main_rho_2} 
 \rho([0,t), x)  = \nu(x)\chi_{\{\eta(t,x)>0\}} + \kappa_2 (t,x)\chi_{ \{\eta(t,x)=0\} }
\end{equation}
where we define 
\begin{equation}\label{Type_II_jump}
\kappa_2(t,x) := -(-\Delta)^s w(t,x) = -\int_0^t (-\Delta)^s \eta(a,x) da.
\end{equation}


It is crucial to note that in contrast to the local case,  $\kappa_2 \neq 0$ on Int$(\{\eta=0\})$ unless $\eta$ is identically zero. 
 Together with \eqref{non_local_heat} this, at least formally, show
that our optimal Eulerian variables for the cost of Type (II) satisfy  the nonlocal melting Stefan problem 
\begin{equation}
\p_t h + (-\Delta)^s \eta = 0, \quad \quad h =h_2(\eta):=\eta + \nu\chi_{\{\eta>0\}} +  \kappa_2 \chi_{ \{\eta=0\} }.\tag{$St_{2,\nu}$}     \label{St2nu}
\end{equation}
Note that this expression of $h$ is consistent with \eqref{enthalpy0}. 

Similarly for Type (I) costs, at least formally, we have  that 
\begin{equation}\label{main_rho}
\rho((t,\infty), x) =\nu (x)\chi_{\{\eta(t,x)>0\}} + \kappa_1 (t,x)\chi_{\{\eta(t,x)=0\}}
\end{equation}
with 
\begin{equation}\label{Type_I_jump}
\kappa_1(t,x)  := -(-\Delta)^sw(t,x) = -\int_t^\infty [(-\Delta)^s \eta(a,x)]  da  ,
\end{equation}
and establish that $\eta$ solves the nonlocal freezing Stefan problem  
\begin{equation}
      \partial_t h   +(-\Delta)^s \eta = 0, \qquad h=h_1(\eta):=\eta -\nu \chi_{\{\eta>0\}} - \kappa_1 \chi_{ \{\eta=0\} }, \label{St1nu} \tag{$St_{1,\nu}$}
\end{equation}
 thus verifying the formula \eqref{enthalpy0}. 

 \begin{remark} \label{kappa_remark}
For $\eta$ and its associated potential $w$, Proposition \ref{w_s_laplacian_remark} establishes that for any $i=1,2$ and $t \geq 0$,  we have $0 \leq \kappa_i(t,\cdot) \chi_{ \{ {\eta=0} \} }(t,\cdot) \leq \nu(\cdot) \in L^1 \cap L^{\infty}$. 
 \end{remark}

\medskip
Below we give rigorous justification of the above heuristic arguments. To this end we introduce a weak solution to \eqref{St1nu} 
 and \eqref{St2nu} based on the above discussions. For $u \in L^1(\R^+;\R^d)$ for which there is a version  $\tilde{u}$  of $u$  (in $\R^+ \times \R^d$) such that $\{ \tilde{u}(t,\cdot)>0\}$ is non-decreasing or non-increasing in $t$ and $\{ \tilde{u}>0\}$ is open in $\R^+ \times \R^d$, we define the {\it initial domain} of $u$ as  
\begin{equation}\label{initial}
E(u) := \lim_{t\to 0^+} \{ u(t,\cdot)>0\}.
\end{equation} 
In the following remark, we verify that the initial domain $E(u)$ is well-defined (up to zero-measure).
\begin{remark}  \label{inital_domain_well_posed}
 We show that the initial domain is uniquely determined a.e. in $\R^d$. We only consider the case when $\{\tilde{u}(t,\cdot)>0\}$ is non-increasing in $t$, since the non-decreasing case similarly follows. 
We claim that if $u_1$ and $u_2$ are two versions of $u$ with $\{u_i(t,\cdot)>0\}$ non-increasing in $t$ and $\{u_i>0\}$ open ($i=1,2$),  then for any $t>0$,
\begin{equation} | \{ u_1(t,\cdot) > 0 \} \setminus \{ u_2(t,\cdot) > 0 \} | = 0.  \label{level_set_0_measure} \end{equation} 
 Suppose that the above does not hold for some $t>0.$ Let $A:= \{u_1(t,\cdot)>0\} \cap  \{ u_2(t,\cdot)=0 \}$.
As $\{u_1>0\}$ is open, for all $x \in A$, there is $r_x>0$ such that {$[t,t+r_x] \times \{x\} \subset \{u_1>0\}$}. Setting $A_k := \{ x \in A: r_x \geq 1/k \}$, there exists $n$ such that $|A_n|>0.$ Then $|[t,t+1/n] \times A_n|>0$ and by monotonicity of positivity sets, $[t,t+1/n] \times A_n \subset$ $\{u_1>0\} \cap \{u_2=0\}$.  This contradicts the fact that both $u_1$ and $u_2$ are  versions of $\eta$. \\
As $\{\tilde{u_i}(t,\cdot)>0\}$ is non-increasing in $t$, the initial domain $E(u_i)$   can be written as
\[ E(u_i) = \bigcup_{n \in \N} \{ u_i(1/n,\cdot)>0 \}. \] 
Hence  by \eqref{level_set_0_measure},  we deduce that $ E(u_1)= E(u_2)$ up to zero-measure.
\end{remark}

We now define the notion of weak solutions  to \eqref{St1nu} 
 and \eqref{St2nu}  for a non-negative initial data  $u_0\in L^1(\R^d)\cap L^{\infty}(\R^d)$ and an initial domain  $E$ containing the support of $u_0$.

In the definition below, we introduce the variables $u$ and $v_i$ ($i=1,2$) for the definition of general weak solutions for weighted Stefan problems, with the understanding that they will later correspond to the Eulerian variable $\eta$ and its associated potential $w$ respectively, generated by $\mathcal{P}_0(\mu,\nu)$.

\begin{definition}[Weighted Stefan Problem]  \label{weighted_stefan_def}
Let $\nu$ and $u_0$ be two bounded and nonnegative functions on $\R^d$, and let $E \subset \R^d$ be a measurable set. Let $u \in L^1 \cap L^{\infty} $ be a non-negative  function on $\R^+ \times \R^d$, and define its time integrated version:
\[ v_1(t,x) := \int_t^{\infty} u(a,x) da \,\,\text{ and }\,\, v_2(t,x) := \int_0^t u(a,x) da. \] 

\begin{itemize}
    \item [(1)]  We say that $u$ is a weak solution to \eqref{St1nu} with initial data $(u_0,E)$ if $E(u)=E$ and 
\begin{enumerate}
    \item[(a)] $\{u>0\}$ is open in $\R^+ \times \R^d$,
    \item[(b)] $\{u(t,\cdot)>0\}$ is uniformly bounded and  non-increasing in $t$,
    \item[(c)]   for any  $t \geq 0$, $0 \leq [-(-\Delta)^s v_1 \cdot \chi_{ \{u=0\} }](t,\cdot) \leq \nu(\cdot)$,
    \item[(d)] setting $h=h_1(u) := u - \nu \chi_{ \{u>0\} } + [(-\Delta)^s v_1] \chi_{ \{u=0\} }$, for any $\varphi \in C^{\infty}_c(\R^+ \times \R^d)$,
\begin{equation}\label{weak_eq_St_1}
    \iint (-\p_t \varphi \cdot  h + \left[ (-\Delta)^s \varphi\right]  u ) dtdx  = \int \varphi(0,\cdot) (u_0 - \nu\chi_E +  [(-\Delta)^s v_1](0,\cdot)\chi_{E^c}) dx.
\end{equation}

\end{enumerate}

\item[(2)] We say that $u$ is a weak solution to \eqref{St2nu} with initial data $(u_0,E)$ if $E(u)=E$ and
\begin{enumerate}
\item[(a')]  $\{u>0\}$ is open in $\R^+ \times \R^d$,
\item[(b')] $\{u(t,\cdot)>0\}$ uniformly bounded and non-decreasing in $t$, 
\item[(c')] for any $t \geq 0$, $ 0\leq [-(-\Delta)^s v_2 \cdot \chi_{ \{u=0\} }](t,\cdot) \leq \nu(\cdot)$.
\item[(d')] setting $h = h_2(u) = u + \nu \chi_{ \{u>0\} } - [(-\Delta)^s v_2] \chi_{ \{u=0\} }$,
 for  any  $\varphi \in C^{\infty}_c(\R^+ \times \R^d)$,
\begin{equation} \label{weak_eq_St_2}
    \iint (-\p_t \varphi \cdot h + \left[ (-\Delta)^s\varphi \right] u) dtdx = \int \varphi(0,\cdot) (u_0 + \nu \chi_{E}) dx.
\end{equation}
\end{enumerate}
\end{itemize}
We say that $(h_i,u)$ solves $(St_{i,\nu})$ if $u$ solves it with $h_i = h_i(u)$ as defined above, for $i=1,2$.


\end{definition}

  Note that, by expanding out the definition of the enthalpy variable $h$,  \eqref{weak_eq_St_1}  reads as
\begin{multline*}
    \iint (-\p_t \varphi \cdot (u - \nu \chi_{ \{u>0\} }  + (-\Delta)^s v_1 \cdot \chi_{ \{u=0\} } ) + u \cdot (-\Delta)^s \varphi )dtdx \\
    =  \int \varphi(0,\cdot) (u_0 - \nu \chi_E + (-\Delta)^sv_1(0,\cdot) \cdot \chi_{E^c}  ) dx, 
\end{multline*}
and \eqref{weak_eq_St_2} reads as
\[ \iint (-\p_t \varphi \cdot (u + \nu \chi_{ \{u>0\} }  - (-\Delta)^s v_2 \cdot \chi_{ \{u=0\} } )    +  u \cdot (-\Delta)^s \varphi)dtdx =  \int \varphi(0,\cdot) (u_0 + \nu \chi_E ) dx.  \]




\medskip

Using the  above definition of a solution to the weighted Stefan problem, we now define the  solution to the Stefan problem \eqref{St1} and \eqref{St2}.

\begin{definition}\label{def_St}
 We say that $u$ solves \eqref{St1} (resp. \eqref{St2}) if it solves \eqref{St1nu}  (resp. \eqref{St2nu}) with some (non-negative) $\nu  \leq 1$ that is  equal to $1$ on $\{x \in \R^d:u(t,x)>0 \hbox{ for some } t>0\}$. We also say that $(h_i,u)$ solves $(St_{i})$ if $u$ solves it with $h_i = h_i(u)$ (and with corresponding $\nu$) as defined in Definition \ref{weighted_stefan_def}, for $i=1,2$.
\end{definition} 

The condition $\nu\leq 1$ reflects the fact that the maximal capacity of the inactive region to hold ice particles in our phase transition process in \eqref{St1} or \eqref{St2} equals one. Note that, in contrast to the local case \cite{kim2021stefan},  $\nu$ can not be chosen to be a characteristic function due to the jumps of the underlying process.

\medskip

In Section \ref{solution_stefans_1}, we  show that our notion of solutions corresponds to those generated by the Eulerian variables $(\eta,\rho)$ from Lemma \ref{Eulerian_Variables} with a target measure $\nu$. In particular, the solution is in the stronger regularity class, namely $\eta\in L^2(\R^+; H^s(\R^d))$ (see Lemma \ref{regularity_lemma}). In Section \ref{enthalpy_section} we further show that for the equation \eqref{St2}, our notion of solution coincides with that of  \eqref{Ste} which was considered in several recent papers \cite{del2021one, del2017distributional, del2017uniqueness}. This provides a new Sobolev regularity for solutions of \eqref{Ste}, which was only known to be continuous \cite{del2021one}. 

\medskip

In our proof with the Eulerian variables $(\eta,\rho)$, the characterization of $\rho$ in terms of $\eta$  is a main step in showing  that  $\eta$ is a weak solution to the weighted Stefan problems. Namely, we deduce \eqref{main_rho} for Type (I) and \eqref{main_rho_2} for Type (II).

\medskip

\subsection{Identification of  $\rho$}
From now on, $\eta$ denotes the modification of the Eulerian variable such that $\{\eta>0\}$ is open in $\R^+ \times \R^d$ (see 
Remark \ref{modififcation}).
\begin{theorem}[$\rho$ for Type (I)]   \label{rho_Type_1}
Let $\tau$ be the optimal stopping time for $\mathcal{P}_0(\mu,\nu)$ for a Type (I) cost and $(\eta,\rho)$ be the associated Eulerian variables.  Then, for any  $\varphi \in C_c^{\infty}(\R^+ \times \R^d)$, 
\begin{multline}\label{decomposition}  
    \iint_{\R^+ \times \R^d} \varphi(t,x) \rho(dt,dx)  \\
    = \iint_{\R^+ \times \R^d} \left( \p_t  \varphi \cdot   \nu \chi_{ \{\eta>0\}} + \p_t \varphi \cdot 
 \rho((t,\infty), x)  \chi_{ \{\eta=0\}} \right)   dtdx+ \int_{\R^d} \varphi(0,x) \nu(x) dx. 
\end{multline} 

\noindent If $\tau>0$ a.s., then \eqref{main_rho} holds $(t,x)$-a.e. 

\end{theorem}

\begin{proof} 
Using the disintegration formula in Definition \ref{disintegration} (the map $x\mapsto\rho_x \in \mathcal{M}(\R^+)$ denotes the disintegration of $\rho$ with respect to the spatial measure $\nu$),  
\begin{align}\label{rho_1} 
     \iint_{\R^+ \times \R^d} \varphi(t,x) &\rho(dt,dx) = \int_{\R^d}\int_{0}^{\infty} \varphi(t,x) \rho_x(dt) \nu(x)dx \nonumber\\
     &=\int_{\R^d} \int_0^{\infty} \left[ \int_0^t \p_t \varphi(w,x) dw \right] \rho_x(dt) \nu(x) dx + \int_{\R^d} \int_0^{\infty} \varphi(0,x) \rho_x(dt) \nu(x) dx.
\end{align} 
 By Fubini's theorem and noting that $\rho_x[\R^+]=1$ $\nu$-a.e., \eqref{rho_1} reduces to
\begin{equation} 
\int_{\R^d} \int_0^{\infty} \p_t \varphi(w,x) \left[ \int_w^{\infty} \rho_x(dt) \right] \nu(x) dw dx  + \int_{\R^d} \varphi(0,x) \nu(x) dx. \label{rho_2} \end{equation} Observe that by Lemma \ref{lambda_equality}, $\int_w^{\infty} \rho_x(dt) = \rho_x((w,\infty))$ $\nu$-a.e.   By Theorems \ref{barrier_existence} and \ref{equality_hitting_time}, 
there is $s^U:\R^d \to \R^+$  such that  $\tau$ is the hitting time to the barrier $\{(t,x) \in \R^+ \times \R^d:t \geq s^U(x)\}$. {By Lemma \ref{stop_inside_barrier},}  $\tau \geq s^U(X_{\tau})$ a.s. This implies that  for $t<s^U(x)$ (equivalently, $\eta(t,x)>0$ by Lemma \ref{open_barrier}), 
$\rho_x((t,\infty)) = \mathbb{P}( \tau > t \mid  X_{\tau}=x) =  1$. Therefore,  
\[ \iint \varphi(t,x) \rho(dt,dx) = \iint (\p_t \varphi \cdot \nu \chi_{ \{\eta>0\}}  +  \p_t \varphi \cdot  \nu  \rho_x((t,\infty)) \chi_{ \{\eta=0\}} ) dtdx + \int \varphi(0,x) \nu(x)dx. \] 
Again by the disintegration formula, we obtain \eqref{decomposition}.

\medskip

Now because we verified \eqref{decomposition}, we will use it to prove \eqref{main_rho}. If $\tau>0$ a.s. then $\rho(\{0\} \times \R^d) = 0$. Thus by Lemmas  \ref{lambda_equality} and \ref{diff_rho_t},
{\begin{equation} 
 \iint \varphi(t,x) \rho(dt,dx) = -  \iint \p_t \varphi \cdot  \rho((0,t),x) dtdx =   -\iint (\p_t \varphi) \left[ \nu(x) - \rho((t,\infty),x) \right] dtdx . \label{rho_time_deriv_e} \end{equation}} 
Combining this with \eqref{decomposition},
\[ \iint \p_t \varphi   \cdot \rho((t,\infty),x)  dtdx = \iint \left( \p_t \varphi 
 \cdot \nu \chi_{ \{\eta>0\}  } + \p_t \varphi  \cdot \rho( (t,\infty) ,x)  \chi_{ \{ \eta =0 \} } \right) dtdx.   \] Note that by Proposition \ref{w_s_laplacian_remark},
 \begin{equation} 
   \rho( (t,\infty),x) \chi_{ \{\eta=0\} } = -[(-\Delta)^s w(t,x)] \chi_{ \{\eta=0\} }  = \kappa_1(t,x) \chi_{ \{\eta=0\} } . \\
 \end{equation}
 Hence, we deduce \eqref{main_rho}. 
\end{proof}

Similarly we prove \eqref{main_rho_2} for Type (II).

\begin{theorem}[$\rho$ for Type (II)] 
Let $\tau$ be the optimal stopping time for $\mathcal{P}_0(\mu,\nu)$ with a Type (II) cost and $(\eta,\rho)$ be the associated Eulerian variables.
If $\tau>0$ a.s., then for any $\varphi \in C^{\infty}_c( \R^+ \times \R^d)$,
 \begin{align} \label{442}
     \iint_{  \R^+ \times \R^d } \varphi(t,x) \rho(dt,dx) = - \iint_{  \R^+ \times \R^d }  ( \p_t \varphi \cdot  \nu  \chi_{ \{\eta>0\} } +\p_t  \varphi  \cdot  \rho((0,t),x) \chi_{ \{\eta=0\} } ) dtdx.
 \end{align}
In particular, \eqref{main_rho_2} holds  $(t,x)$-a.e.
\label{rho_Type_2}
\end{theorem}

\begin{proof}
Again by the disintegration formula,
\begin{align*}
    \iint_{  \R^+ \times \R^d } \varphi(t,x) \rho(dt,dx) &= \int_{\R^d} \int_{0}^{\infty} \varphi(t,x) \rho_x(dt) \nu(x)dx  \\
    &= -\iint_{  \R^+ \times \R^d } \left[ \int_t^{\infty} \p_t \varphi(w,x) dw \right] \rho_x(dt) \nu(x)dx \\
&=  -\int_{\R^d} \int_0^{\infty} \p_t \varphi(w,x) \left[ \int_0^{w} \rho_x(dt) \right] \nu(x) dw dx.
\end{align*} Recall that by Lemma \ref{lambda_equality} and $\tau>0$ a.s. that $\int_0^{w} \rho_x(dt) = \rho_x((0,w) )$ a.e.  Now {as $\tau>0$ a.s.},  {Lemma \ref{stop_inside_barrier}} implies that $\tau \leq s(X_{\tau})$ a.s.  Since  $s=s^U$ $\nu$-a.e. by  Theorem \ref{equality_hitting_time}, we have 
 $\tau \leq s(X_{\tau}) = s^U(X_{\tau})$ a.s.  Hence for $t>s^U(x)$ (equivalently,  $  \eta(t,x)>0$ by Lemma \ref{open_barrier}),  $\rho_x( (0,t) ) = 1$ $\nu$-a.e., which implies \eqref{442}.

This together with by Lemmas  \ref{lambda_equality} and \ref{diff_rho_t} imply
\[ \iint_{\R^+\times \R^d} \p_t \varphi \cdot \rho( (0,t),x) dtdx = \iint_{\R^+\times \R^d} (\p_t \varphi \cdot \nu(x) \chi_{\{\eta>0\} } +  \p_t \varphi \cdot \rho((0,t),x) \chi_{ \{\eta=0\} }) dtdx.  \] {As in Theorem \ref{rho_Type_1}, $\rho((0,t),x) \chi_{ \{\eta=0\} } = \kappa_2(t,x) \chi_{ \{\eta=0\} }$, and thus we deduce 
\eqref{main_rho_2}.}

\end{proof}

 Theorem \ref{rho_Type_1} implies that it suffices to  verify the initial condition $\kappa_1(0,\cdot) \chi_{E^c} = \nu \chi_{E^c}$ to establish that the Eulerian variable $\eta$ (with Type (I) cost)  solves \eqref{St1nu}. Indeed, from \eqref{non_local_heat} and \eqref{decomposition}, we have that for any $\varphi \in C^{\infty}_c( \R^+ \times \R^d)$,
    \[ \iint -\p_t \varphi( \eta - \nu \chi_{ \{\eta>0\} } - \rho( (t,\infty),x) \chi_{ \{\eta=0\} }) + \iint (-\Delta)^s \varphi \cdot \eta = \int \varphi(0,x)(\mu(x)-\nu(x)). \] By Proposition \ref{w_s_laplacian_remark}, this expression is equal to
  \[ \iint -\p_t \varphi( \eta - \nu \chi_{ \{\eta>0\} } + (-\Delta)^sw \cdot \chi_{ \{\eta=0\} }) + \iint (-\Delta)^s \varphi \cdot \eta  = \int \varphi(0,x)(\mu(x)-\nu(x)). \] 
  Hence in order to obtain that $\eta$ is a weak solution to \eqref{St1nu} with initial data $\mu$ and initial domain $E$, it remains to verify $\kappa_1(0,\cdot) \chi_{E^c} = \nu \chi_{E^c}$. We will prove this in Lemma \ref{kappa_1(0,x)}. \\
  
  Similarly Theorem \ref{rho_Type_2} implies that it suffices to  verify the initial condition 
  $\nu \chi_E = 0$ to establish that  the Eulerian variable $\eta$ (with Type (II) cost)  solves \eqref{St2nu}. Indeed, as above,  Proposition \ref{w_s_laplacian_remark} implies $\rho( (0,t),x) \cdot \chi_{ \{\eta(t,x)=0\} } = -(-\Delta)^sw  \cdot \chi_{ \{\eta(t,x)=0\} }$, and thus by \eqref{442} and \eqref{non_local_heat},  for any test function $\varphi \in C^{\infty}_c(\R^+ \times \R^d)$,
  \[ \iint -\p_t \varphi(\eta + \nu \chi_{ \{\eta>0\}  } - (-\Delta)^sw \cdot \chi_{ \{\eta=0\} }) + \iint (-\Delta)^s \varphi \cdot \eta = \int \varphi(0,x) \mu(x). \] We will prove that $\nu \chi_E = 0$ in Lemma \ref{nu_1_E}.
    \\

Now, we verify the initial condition.
Recall that $E= E(\eta)$ denotes   the initial domain of $\eta$, defined in \eqref{initial}.
\begin{lemma}[Initial data for Type (I)]  \label{kappa_1(0,x)}
Let $\tau$ be the optimal stopping time of $\mathcal{P}_0(\mu,\nu)$ for Type (I) and  $(\eta,\rho)$ be the associated Eulerian variables.  Then $\kappa_1(0,\cdot) = \nu(\cdot)$ a.e. outside of $E$ if $\tau>0$ a.s. 
\end{lemma}

\begin{proof} Since $\{\eta(t,\cdot)>0\}$ is non-increasing in $t$, we have $ \R^+ \times E^c  \subset \{\eta=0\}$.  
Also by Proposition \ref{w_s_laplacian_remark}, { $-(-\Delta)^s w(t,x) =  \rho( (t,\infty),x)$ on $\R^+ \times E^c$}. Noting that the condition $\tau>0$ a.s. implies $\nu(\cdot) = \rho([0,\infty),\cdot)=\rho((0,\infty),\cdot) $, for $x\in E^c,$ we have
\begin{align*}
    \nu(x) = \rho((0,\infty),x) = -(-\Delta)^s w(0,x) = \kappa_1(0,x),
\end{align*} 
where the last identity follows from \eqref{Type_I_jump}.
\end{proof}

\begin{lemma}[Initial data for Type (II)]\label{nu_1_E}
With the same notation and assumptions as Theorem \ref{rho_Type_2} for Type (II), we have  $\nu=0$ a.e. in $E$ if $\tau>0$ a.s.
\end{lemma}

\begin{proof} 
 Let $s:\R^d \rightarrow \R^+$ be the barrier function associated  to $\tau$. Then {since $\tau>0$ a.s.,}  by  {Lemma \ref{stop_inside_barrier}}, we have $\tau \leq s(X_{\tau})$.   Also by Theorem \ref{equality_hitting_time},  $s(X_{\tau})=s^U(X_{\tau})$ a.s. Hence,  as $\tau  \leq s^U(X_{\tau})$, we have  $(\tau,X_{\tau}) \in \{w=0\} = \{\eta=0\}$ a.s., implying that $\eta(\tau,X_{\tau})=0$ a.s. 
Since  $\eta(t,x)>0$ for any $x \in E$ and $t\in \R^+$, we deduce that  $X_{\tau} \notin E$ a.s. Recalling $X_{\tau} \sim \nu$, we conclude the proof.
\end{proof}

Recalling \eqref{main_rho} and \eqref{main_rho_2} along with using
  Theorem \ref{rho_Type_1} and Lemma \ref{kappa_1(0,x)} for Type (I) and Theorem \ref{rho_Type_2} with Lemma \ref{nu_1_E} for Type (II), we conclude:

\begin{theorem}[Solutions to the Fractional Stefan Problem]\label{solution_eta_stefan}
Let $(\eta,\rho)$ be the  Eulerian variables associated to the optimal stopping time $\tau$ of $\mathcal{P}_0(\mu,\nu)$, and assume that $\tau>0$ a.s. Then \begin{itemize}
    \item For Type (I),  $\eta$ solves \eqref{St1nu} with 
  initial data $(\mu, E(\eta))$. 
 \item For Type (II), $\eta$ solves \eqref{St2nu} with  initial data $(\mu, E(\eta))$.
 \end{itemize}
\end{theorem}
We will later show that for some solutions of $\eqref{St1nu}$, the initial trace can be characterized as the positivity set of the elliptic obstacle problem (see Theorem \ref{nu_identify}) for the solutions we generate in Section \ref{connection_stefan}. {Also for $\eqref{St2nu},$ when $\mu$ has no mushy region,  the initial trace of  solutions we construct later will be shown to be the support of $\mu$.}

\subsection{Solutions to Weighted Stefan Problems}  \label{solution_stefans_1}

In this section, we interpret the solution to the weighted Stefan problem as the Eulerian variables $(\eta,\rho)$ from Lemma \ref{Eulerian_Variables} associated to $\mathcal{P}_0(\mu,\nu)$ from \eqref{Var_Skoro}. Assume that 
\begin{equation}\label{initial_0}
0 \leq \nu \in L^{\infty}(\R^d) \cap L^1(\R^d), \quad   0 \leq \eta_0 \in L^{\infty}(\R^d) \hbox{ of compact support.}
\end{equation} 
\begin{theorem}[Consistency of \eqref{St1nu}] \label{St_1_Consistent}
For $\nu$ and $\eta_0$ satisfying  \eqref{initial_0}, let $\eta$ be a weak solution to \eqref{St1nu} with initial data $(\eta_0,E)$. Define $s(x) := \sup\{t: \eta(t,x)>0\}$, and let $\kappa_1 := -[(-\Delta)^s v_1]$ where $v_1(t,x) = \int_t^{\infty} \eta(a,x) da$ is from Definition \ref{weighted_stefan_def}. Now we define $\rho \in \mathcal{M}(\R^+ \times \R^d)$ as follows: for any $t>0$,
\begin{align} \label{def}
    \rho((t,\infty),\cdot):= 
\nu \chi_{ \{\eta(t,\cdot)>0\} } \chi_{ \{s(\cdot)<\infty\} } + \kappa_1(t,\cdot) \chi_{ \{\eta(t,\cdot)=0\} } \quad \hbox{ with }\rho(0,\cdot) \equiv 0. 
\end{align}
 Then, $(\eta,\rho)$
 is the optimal Eulerian variables from Lemma \ref{Eulerian_Variables} for a cost of Type (I) between $\eta_0$  and  $\tilde{\nu}$ whose density is given by
 \begin{equation}\label{tilde_nu_2}
 \tilde{\nu}(x) :=  \rho([0,\infty),x) = \nu(x) \chi_{ \{0<s(x)<\infty\} } + \kappa_1(0,x) \chi_{ \{ s(x) = 0 \}  }.
 \end{equation}
 \end{theorem}

\begin{proof}
First, let us verify the second identity in \eqref{tilde_nu_2}. Since   $E$ is the collection of limit points of $\{\eta(t,\cdot)>0\}$ as $t\to 0^+$ (see Definition \ref{weighted_stefan_def}), by the fact $\rho(\{0\} \times \R^d)=0$ along with the dominated convergence theorem,
$$\rho( [0,\infty) \times A) = \lim_{t \rightarrow 0^+} \rho( (t,\infty) \times A) = \int_A \left( \nu(x) \chi_{E } \chi_{ \{s(x)<\infty\} } + \kappa_1(0,x) \chi_{E^c} \right) dx. $$ 
This implies \eqref{tilde_nu_2}, since $E = \{x\in \R^d: s(x)>0\}.$

Next note that, since the support of $\eta(t,\cdot)$ is non-increasing in time and $E$ is compact, there is  $r>0$ such that $B_r(0)$ contains the support of $\eta(t,\cdot)$ for all $t \ge 0$. We show that $(\eta,\rho)$ solve \eqref{non_local_heat} in $(0,\infty) \times B_r(0)$. By the definition of $\rho$, for any $\varphi \in C^{\infty}_c(\R^+ \times B_r(0))$,
\begin{align*}
 \iint \p_t \varphi & \cdot  \rho( (t,\infty),x) dtdx  \\
 &= \iint \p_t \varphi \cdot( \nu \chi_{ \{\eta>0\} } \chi_{ \{s<\infty\} } + \kappa_1 \chi_{ \{\eta=0\} }) dtdx \\
 &= \iint \p_t \varphi \cdot( \nu \chi_{ \{\eta>0\} }  + \kappa_1 \chi_{ \{\eta=0\} }) dtdx - \iint \p_t \varphi \cdot  \nu \chi_{ \{\eta>0\}} \chi_{\{s=\infty\}  } dtdx\\ 
 &= \iint \p_t \varphi \cdot( \nu \chi_{ \{\eta>0\} }  + \kappa_1 \chi_{ \{\eta=0\} }) dtdx + \int \varphi(0,\cdot) \nu \chi_{ \{s=\infty\} } dx,
 \end{align*}
 where the last equality follows from the following reason: As $\{\eta>0\}$ is open in space-time  and  $\{\eta(t,\cdot)>0\}$ is monotone in time (see Definition \ref{weighted_stefan_def}), the definition of $s(x)$ yields  $\{\eta(t,x)>0\} = \{ t < s(x) \}$ and thus
 $$
 \int \partial_t  \varphi  \cdot\chi_{ \{\eta>0\} } dt = \int_0^{s(x)} \partial_t \varphi (t,x) dt = \varphi(s(x),x)-\varphi(0, x).
 $$
Observe that one can also write 
\begin{align*}
    \iint \p_t \varphi \cdot  \rho( (t,\infty), x ) dtdx &= \iint \p_t \varphi \cdot  \tilde{\nu} dtdx - \iint \p_t \varphi \cdot  \rho( (0,t],x) dtdx \\
&=  -\int \varphi(0,\cdot )\tilde{\nu} dx + \iint \varphi \rho(dt,dx),
\end{align*}
where the first equality is due to the definition of $\tilde{\nu}$ and the second equality follows from the integration by parts (Lemma \ref{diff_rho_t}), along with the fact  $\rho( \{0\} \times \R^d)=0$.

Putting the above displays together, 
\begin{align} \label{tilde_nu_22}
     \iint \varphi \rho(dt,dx) &= \iint \p_t \varphi \cdot (\nu \chi_{ \{\eta>0\} } + \kappa_1 \chi_{ \{\eta=0\} }) dtdx+\int \varphi(0,\cdot) \nu \chi_{ \{s=\infty\} } dx  + \int \varphi(0,\cdot )\tilde{\nu} dx  \nonumber \\ 
&=\iint \p_t \varphi \cdot (\nu \chi_{ \{\eta>0\} } + \kappa_1 \chi_{ \{\eta=0\} }) dtdx + \int  \varphi(0,\cdot) 
\nu \chi_E dx + \int \varphi(0,\cdot) \kappa_1(0,\cdot) \chi_{E^c} dx,
\end{align}
where the last identity follows from the fact that  $E = \{ s(\cdot)>0\}$ is a disjoint union of $  E \cap \{s(\cdot)< \infty\}$ and $ E \cap  \{s(\cdot)=\infty\} = \{s(\cdot)=\infty\}$. 
Therefore, by the weak form of \eqref{St1nu}, namely  \eqref{weak_eq_St_1}, we conclude that $(\eta,\rho)$ solves \eqref{non_local_heat} with initial data $\eta_0$ in $(0,\infty) \times B_r(0)$. \\

Finally, we show that $(\eta,\rho)$ is the Eulerian variable associated to the optimizer of $\mathcal{P}_0(\eta_0,\tilde{\nu})$ for a Type (I) cost. Let us define $\tau := \inf \{t \ge 0 : t \geq s(X_t) \}$. Since $\tau$ is bounded by the exit time of   $B_r(0)$, it has finite moments. 
Let $(\overline{\eta},\overline{\rho})$  be the  Eulerian variables associated to $(\eta_0,\tau)$. Then as $\{\eta=0\}$ is closed, $\overline{\eta}[ \{\eta=0\} ] = 0$ and $ \overline{\rho}[ \{\eta=0\} ] = 1$.  Hence by Theorem \ref{uniqueness},  $\overline{\eta}=\eta$ and $\overline{\rho}=\rho$, concluding the proof.
\end{proof}

\begin{remark} The formula of $\rho$ in Theorem \ref{St_1_Consistent} is consistent with the one in \eqref{main_rho}. This is because for the cost of Type (I), we have $\nu = 0$ on $\{s=\infty\}$. 
\end{remark}

For the consistency result of \eqref{St2}, the corresponding result to Theorem~\ref{St_1_Consistent} is implied from Theorem \ref{melting_solve} when there is no initial mushy region.


\subsection{Enthalpy variable} \label{enthalpy_section} 
In this section, we show that in the melting scenario, the weighted temperature-based Stefan problem \eqref{St2nu} can be rewritten in terms of the  associated enthalpy variable. This connects our approach to the several recent papers  \cite{del2021one,del2017distributional,del2017uniqueness} which consider a particular case of  the enthalpy equation, \eqref{Ste}, to study fractional Stefan problems.  Our approach furthermore connects \eqref{Ste} to the parabolic obstacle problem studied in \cite{caffarelli2013regularity, ros2021optimal, barrios2018free, borrin2021obstacle},  extending the well-known connection in the local case to the non-local case for the first time.

\medskip

\begin{definition}[Enthalpy variable] 
Assume that $\nu$ is a  bounded and nonnegative function in $\R^d$. Let   $\eta$ be the solution to \eqref{St1nu} or \eqref{St2nu} with initial domain $E$. Then the enthalpy variable $h$ is defined as 
\begin{align*}
     h(t,x) := \begin{cases}
    \eta(t,x) + \rho( [0,t),x) - \nu(x) &\hbox{ for \eqref{St1nu}}, \\
    \eta(t,x) +  \tilde{\rho}( [0,t),x) + \nu(x) \chi_E  &\hbox{ for \eqref{St2nu}},
\end{cases} 
\end{align*} 
where $\rho$ for \eqref{St1nu} is from \eqref{def} and  $\tilde{\rho}$ for \eqref{St2nu} is defined  via
\[ \tilde{\rho}( [0,t),\cdot) := \nu(\cdot) \chi_{ \{\eta(t,\cdot)>0\} } \chi_{E^c} + \kappa_2(t,\cdot) \chi_{ \{\eta(t,\cdot)=0\}  } \text{ with } \tilde{\rho}(0,\cdot) \equiv 0,  \] where $\kappa_2:= [-(-\Delta)^s v_2]$ and $v_2(t,x) := \int_0^t \eta(a,x) da$. Our  definition of $\tilde{\rho}$ is motivated from \eqref{main_rho_2}  and Lemma \ref{nu_1_E}.
    \label{enthalpy_def}
\end{definition}

\begin{remark} For Type (II),  $\eta$ from Lemma \ref{Eulerian_Variables} solves \eqref{St2nu} due to Theorem \ref{solution_eta_stefan}. In this case  \eqref{main_rho_2} yields $\tilde{\rho}([0,t), \cdot ) = \rho([0,t),\cdot ) - \nu \chi_E$. 
Hence, our definition {of the enthalpy} coincides with \eqref{enthalpy0} .
We have chosen to remove the $\nu \chi_E$ term from $\rho$ to emphasize its role as part of the initial enthalpy (see Theorem \ref{stefan_enthalpy} below for details).


\end{remark}

Recall that \eqref{main_rho} and \eqref{main_rho_2} (for Type (I) and (II) respectively) allow us to rewrite $h$ in terms of $\eta$ when $\tau>0$ a.s.  Below we show that the converse holds for the melting case \eqref{St2nu} and in particular for \eqref{St2}. Namely in this case we can recast $\eta$ in terms of the enthalpy variable, leading to an enthalpy-based formulation of \eqref{St2nu}. When $\nu \equiv 1$, our definition is equivalent to that of \cite{del2021one}.

\begin{definition}[Enthalpy form of \eqref{St2nu}] Let $\nu$ and $h_0$ be two bounded and nonnegative functions in $\R^d$. We say that $h \in L^{\infty}(\R^+\times \R^d)$ is a weak solution to
\begin{equation}
\tag{$St_{h,\nu}$}
    \p_t h = -(-\Delta)^s (h-\nu)_+\quad \text{ in } (0,\infty) \times \R^d, \quad h(0,\cdot ) = h_0,    \label{weighted_st2_h} 
\end{equation} if for any $\varphi \in C^{\infty}_c(\R^+ \times \R^d)$ we have
\begin{equation} \int_0^{\infty}\int_{\R^d} (-h\cdot \p_t \varphi +  (h-\nu)_+  (-\Delta)^s \varphi  )dxdt = \int_{\R^d} \varphi(0,\cdot)h_0 dx  .\end{equation} 
In particualr, if $h$ is a weak solution to \eqref{weighted_st2_h} with $\nu \equiv 1$, then we say that $h$ is a weak solution to  \eqref{Ste}. \label{enthalpy_weighted_def}
\end{definition}

Note that in our definition of solutions to \eqref{St2nu} and \eqref{weighted_st2_h}, we view $\nu$ as a target measure. In contrast, in the equations \eqref{St2} and \eqref{Ste}, we regard  {$\nu\equiv 1$ as a weight.}

\begin{theorem} Let $\nu$ be a  bounded and nonnegative function in $\R^d$.
\begin{enumerate}
    \item Assume that $\eta$ solves \eqref{St2nu} with initial data $(\eta_0,E)$. Then the  enthalpy $h$ from Definition \ref{enthalpy_def} is a weak solution of \eqref{weighted_st2_h} with  $h_0=\eta_0 + \nu \chi_E$. 
    \item  Assume that $\eta$ solves \eqref{St2} (see Definition~\ref{def_St}). Then the  enthalpy $h$ from Definition \ref{enthalpy_def} is the unique weak solution of \eqref{Ste} with $h_0 = \eta_0+ \chi_E$, and $\eta$ is the unique weak solution of \eqref{St2} with initial data $(\eta_0,E)$.
\end{enumerate}

\label{stefan_enthalpy}
\end{theorem}

\begin{proof} Assume that $\eta $ solves \eqref{St2nu}. Let us first show that $\eta(t,x) = (h(t,x)-\nu(x))_+$. To check this, {fix $t>0$. Then if $ \eta(t,x)>0$}, then we have from Definition \ref{enthalpy_def} that $\tilde{\rho}([0,t),x) = \nu(x) \chi_{E^c}(x)$ and thus $h(t,x) = \eta(t,x) + \nu(x)$. Next {if $ \eta(t,x)=0$}, then we have $h(t,x) = \kappa_2(t,x) \leq \nu(x)$ (see Definition \ref{weighted_stefan_def}), and thus $(h-\nu)_+(t,x)=0$. Hence we have verified $\eta=(h-\nu)_+$, implying that  for any test function $\varphi \in C^{\infty}_c(\R^+ \times \R^d)$, we have from \eqref{weak_eq_St_2} that 
\begin{equation} \int_0^{\infty}\int_{\R^d} (h  ( -\p_t \varphi) +  (h-\nu)_+  (-\Delta)^s \varphi  )dxdt = \int_{\R^d} \varphi(0,\cdot )(\eta_0 + \nu \chi_E) dx.  \label{weak_distrbutional_form_h} \end{equation}

\noindent This shows that $h$ is a weak solution to \eqref{weighted_st2_h} with initial data $\eta_0 + \nu \chi_E$.   

~

Next, we consider the equation \eqref{St2}. We claim that $(h-1)_{+} = \eta$. Fix $t>0.$ Recalling Definition \ref{def_St},  $\nu(x)=1$  {when $\eta(t,x)>0$}, implying $h(t,x) = \eta(t,x)+1$. Also since $\kappa_2(t,x) \leq \nu(x) \leq 1$, if  $\eta(t,x)=0$, then  $(h(t,x)-1)_{+} = (\kappa_2(t,x)-1)_{+} = 0$. Hence we have verified $(h-1)_{+} = \eta$. Since {$E\subset  \{x \in \R^d:\eta(t,x)>0 \hbox{ for some } t>0\}$,} recalling Definition \ref{def_St} again,   we have $\nu \chi_E = \chi_E$. Thus  by \eqref{weak_eq_St_2}, for any test function $\varphi$,
\begin{equation} \int_0^{\infty}\int_{\R^d} (h  ( -\p_t \varphi) +  (h-1)_+  (-\Delta)^s \varphi  )dxdt = \int_{\R^d} \varphi(0,\cdot )(\eta_0 +  \chi_E) dx. \end{equation} In particular, $h$ is a weak solution to \eqref{Ste} with initial data $\eta_0 + \chi_E$. The uniqueness of weak solutions of \eqref{Ste} follows from \cite[Theorem 2.3]{del2021one}. Lastly observe that this yields the uniqueness of $\eta =(h-1)_+$.

\end{proof}

Based on our probabilistic characterization of $h$, we are able to prove a rate of convergence for $h$ as $t \rightarrow \infty$ for Type (II). This type of result does not appear to be known for $\eqref{St2}$ in the literature. 

\begin{lemma}
    Let $\mu,\nu \in \mathcal{M}(\R^d)$ satisfy Assumption \ref{assume1}. Let
     $(\eta,\rho)$ be the Eulerian variables associated to $\tau$, the optimal stopping time for $\mathcal{P}_0(\mu,\nu)$ with a  Type (II) cost. 
There exist  constants {$\gamma=\gamma(r,d,s)>0$ and $C=C(r,d,\mu,s)>0$ (here, $r>0$ is from (4) in Assumption \ref{assume1}) such that if $\tau > 0$ a.s., then for any $t \ge 0,$}
\[ ||h(t,\cdot)-\nu||_{L^1(\R^d)} \leq Ce^{-\gamma t},\]
where $h$ is the enthalpy variable from Definition~\ref{enthalpy_def}.
\label{asymptotic_h}
\end{lemma}

\begin{proof}
From Proposition \ref{w_s_laplacian_remark}, we have that $\kappa_2(t,x) = \rho( (0,t),x)$, and thus by definition of $h$,
\begin{equation} h(t,x)-\nu(x) = \eta(t,x) + \left( \nu(x) \cdot \chi_{ \{\eta>0\} }(t,x) + \rho( (0,t),x) \cdot \chi_{ \{\eta=0\} }(t,x) - \nu(x) \right). \label{diff_enthalpy_nu} \end{equation} Now from Lemma \ref{Eulerian_Variables} and in particular \eqref{non_local_heat}, we have $\{x:\eta(t,x)>0\} \subset B_r(0)$ for all $t \geq 0$. By Cauchy-Schwarz and Lemma \ref{regularity_lemma},  there exists $K=K(r,d),\gamma = \gamma(r,d,s) > 0$ such that
\[ ||\eta(t,\cdot)||_{L^1(\R^d)} \leq K ||\eta(t,\cdot)||_{L^2(\R^d)} \leq K e^{-\gamma t} || \mu ||_{L^2(\R^d)}. \] Using the triangle inequality in \eqref{diff_enthalpy_nu} along with the above inequality, we obtain
$$
\begin{array}{lll}
||h(t,\cdot)-\nu||_{L^1(\R^d)} &\leq& K e^{-\gamma t} || \mu||_{L^2(\R^d)} +  || (\nu-\rho( (0,t),\cdot) ) \chi_{ \{\eta=0 \} } ||_{L^1(\R^d)}  \\
 &\leq& K e^{-\gamma t} || \mu ||_{L^2(\R^d)} +  || \nu-\rho( (0,t),\cdot)   ||_{L^1(\R^d)}.
 \end{array}
 $$
 Since $\nu(x) = \rho( (0,\infty),x)$, we have  $\nu(x)-\rho( (0,t),x)  = \rho( [t,\infty),x)$. Hence we arrive at 
 \begin{equation}\label{h_nu_bound}   ||h(t,\cdot)-\nu||_{L^1(\R^d)} \leq  K e^{-\gamma t} || \mu ||_{L^2(\R^d)} +  \rho([t,\infty) \times \R^d). 
 \end{equation}
Recalling that $\rho \sim (\tau,X_{\tau})$ from Lemma~\ref{Eulerian_Variables}, we have  $\rho([t,\infty) \times \R^d) = \mathbb{P}[\tau \geq t]$. Since $\mathbb{E}[e^{\gamma \tau}] \leq \mathbb{E}[e^{\gamma \tau_r}] < \infty$ (see Remark \ref{exp_moments_exit_time}), by Markov's inequality $\mathbb{P}[\tau \geq t] \leq e^{-\gamma t} \mathbb{E}[e^{\gamma \tau_r}] $.
 Putting  these observations into \eqref{h_nu_bound}, we see that
\[ ||h(t,\cdot)-\nu||_{L^1(\R^d)} \leq e^{-\gamma t}( K ||\mu||_{L^2(\R^d)} + \mathbb{E}[e^{\gamma \tau_r}]), \]
concluding the proof.
 \end{proof}

\subsubsection{A remark on Continuity of Enthalpy variable}  In the case of equation \eqref{Ste}, it was shown \cite{del2021one} that for self-similar solutions, the enthalpy variable is continuous. This is rather surprising because in the local case, the enthalpy variable has a jump discontinuity across the free boundary.  It would be thus interesting to consider regularity of the enthalpy variable in the nonlocal setting. 

\medskip

One possible way to obtain a regularity of the enthalpy variable is to use the newly obtained connection of \eqref{St1nu} and \eqref{St2nu} with the parabolic nonlocal obstacle problem \eqref{w_obstacle}. It was shown in \cite[Theorem 2.1]{caffarelli2013regularity} that  if $\nu-\mu$ is sufficiently smooth, that the solution $w$ of \eqref{w_obstacle} for Type (II) has a  continuous $(-\Delta)^s w$. This is in contrast to  the local case, where the optimal spatial regularity for $w$ is $C^{1,1}$.  Continuity of $(-\Delta)^sw$  implies from our enthalpy formula for Type (II) and Proposition \ref{w_s_laplacian_remark} that the enthalpy for \eqref{St2nu} is continuous. In addition, the free boundary regularity for \eqref{w_obstacle} with Type (II) is  well studied when $\nu-\mu$ is sufficiently smooth (see \cite{ros2021optimal, barrios2018free, borrin2021obstacle}). 

\medskip

These results unfortunately do not apply to \eqref{St2} since we will see in Theorem \ref{melting_solve} later that {we require $\mu>1$ in its support, which is equivalent to the absence of an initial mushy region}. This is rather natural since one expects irregularity near the initial data. As in the local case, one would need a localization argument to show continuity of the enthalpy variable away from the initial time. It is unclear to the authors at the moment how such localization can be carried out for the parabolic nonlocal obstacle problem.
 
\medskip

Definition \ref{enthalpy_def} suggests that one can might be able to show the continuity of the enthalpy using probabilistic arguments. Let us focus on Type (I) case, where the optimal stopping time is given by $\tau = \inf\{ t\ge 0 : t \geq s(X_t)  \}$  for some measurable function $s : \R^d \rightarrow \R^+$ (see Theorem \ref{barrier_existence}). By Lemma \ref{open_barrier},  $\{\eta(t,x)>0\} = \{ t < s(x) \}$.  Thus, recalling $h = \eta - \nu \chi_{ \{\eta>0\} } - \kappa_1 \chi_{ \{\eta=0\} }$, if $\eta$ is continuous, then by using a version of $\kappa_1$ for Type (I), we have 
\[ \lim_{t \rightarrow s(x)^+} -h(t,x) = \nu(x) \mathbb{P}[ \tau > s(x) | X_{\tau}=x ] \text{ and } \lim_{t \rightarrow s(x)^-} -h(t,x) = \nu(x). \]  
Hence, if $\mathbb{P}[\tau=s(x)|X_{\tau}=x]=0$ and $s(x)$ is continuous, then $h(t,x)$ is continuous in time across the free boundary.

\medskip

We expect that under mild assumptions on $s(x)$, such as all of its level sets having measure zero, we have $\mathbb{P}[\tau=s(x) \mid X_{\tau}=x]=0$. This is because it seems unlikely for the particles to land on the graph of $s(x)$ under this assumption. There are some recent investigations in computing $\mathbb{P}[\tau=s(x) \mid X_{\tau}=x]$ for one dimensional Levy processes under some strong assumptions on $s(x)$ (see \cite{chaumont2022creeping,chi2016exact}).

 \section{Optimal Target Problem}\label{sec:target}

In this section, we consider the optimal target problem which will yield weak solutions of \eqref{St1} and \eqref{St2}. Recall that $\mathcal{C}(\tau)$ denotes the cost defined in \eqref{Cost}, i.e. $\mathcal{C}(\tau) = \mathbb{E} \left[ \int_0^{\tau} L(s,X_s) ds \right]$, where the Lagrangian  $L$ satisfies Assumption \ref{assum2}. For the family of measures  $\mathcal{A}_{\mu,M, R} $ given in \eqref{A_mu_e_R}, we associate the corresponding optimal stopping time:
\begin{align*}
    \tilde{\mathcal{A}}_{\mu,M, R} := \{ (\nu, \tau) : \nu \in  \mathcal{A}_{\mu,M, R},  \ \tau 
 \ \text{is an optimizer of}   
 \ \mathcal{P}_0 (\mu,\nu) \text{ in } \eqref{Var_Skoro}\}.
\end{align*}
Throughout this section, we assume that the initial measure $\mu \in L^{\infty}(\R^d)$ is compactly supported and  $f$ satisfies Assumption \ref{assume3}.

~

Consider the constrained family of target measures and stopping times
\[ \mathcal{A}_{R,f}(\mu) := \{ (\nu, \tau) \in \bigcup_{M>0}\tilde{\mathcal{A}}_{\mu,M, R} \hbox{ with } \nu \leq f\},\]
and define
\[ \mathcal{A}_f(\mu) := \bigcup_{R>0}\mathcal{A}_{R,f}(\mu). \]

Now, we consider the optimal target problem associated with $f$:
\begin{equation} \mathcal{P}_f(\mu) := \inf\{ \mathcal{C}(\tau) : (\nu,\tau) \in \mathcal{A}_f(\mu) \}.
\label{opt_target} \end{equation}
It will also be useful to define the truncated problem
\[ \mathcal{P}_{R,f}(\mu) := \inf \{ \mathcal{C}(\tau) : (\nu,\tau) \in \mathcal{A}_{R,f} (\mu) \}. \] By Lemma \ref{non_empty_P_f},  $\mathcal{A}_{R,f}(\mu)$ is non-empty for sufficiently large $R>0$, and thus $\mathcal{A}_f(\mu)$ is non-empty.

\subsection{Existence and Uniqueness}

In this section, we establish the well-posedness of  the variational problem $\mathcal{P}_f(\mu)$. To accomplish this, we first show the monotonicity property of the variational problem, a central ingredient for the analysis in the local case \cite{kim2021stefan}. 
 
\begin{theorem}[Monotonicity]   \label{monotonicity}
Assume that $\mu_1 \leq \mu_2$ and $\nu_i \leq f$  for $i=1,2$. Let  $\tau_i$ be the optimal stopping time for $\mathcal{P}_0(\mu_i,\nu_i)$. Then, the following holds:

\begin{enumerate}
    \item If $(\nu_1,\tau_1)$ is optimal for $\mathcal{P}_f(\mu_1)$, then as stopping times from the initial distribution $\mu_1$, we have that $\tau_1 \leq \tau_2$ a.s.

    \item If $(\nu_i,\tau_i)$ is optimal for $\mathcal{P}_f(\mu_i)$ for $i=1,2$, then $\nu_1 \leq \nu_2$ a.e.
    
\end{enumerate} 

   If we further assume that $\tau_i \leq \tau_R = \inf \{ t : X_t \notin B_R(0) \}$  ($R>0$ is some constant) for $i=1,2$, then the conclusions $(1)$ and $(2)$ also hold for  the version of $\mathcal{P}_{R,f}(\mu_i)$.
	
\end{theorem}

\begin{proof} The proof is analogous to that of \cite[Theorem 7.1]{kim2021stefan}. The only difference is that if $R_i$ is the 
 barrier associated  to $\tau_i$, then one can use {Lemma \ref{stop_inside_barrier}} to conclude that $(\tau_i,X_{\tau_i}) \in R_i$ (if $\tau_i>0$ a.s. for Type (II)).
\end{proof}

\begin{theorem}[Existence and Uniqueness of Optimizer]
There exists an optimal pair $(\nu,\tau)$ for $\mathcal{P}_f(\mu)$ with a $R=R(\mu,d,s,f)>0$ such that $\tau \leq \tau_R$ almost surely. In addition, the optimizer is unique. \label{exist_unique_Pf}
\end{theorem}

\begin{proof} 
We  first show the existence of an optimizer for the truncated problem $\mathcal{P}_{R,f}(\mu)$ for large $R$. Since  $f$ satisfies Assumption \ref{assume3}, $\mathcal{A}_{R,f}(\mu)\neq \emptyset$ for large enough $R>0$. Let $(\nu_n,\tau_n)_{n \ge 1}$ be the minimizing sequence for $\mathcal{P}_{R,f}(\mu)$.  Since  $\{\nu_n\}_{n \ge 1}$ is tight by Lemma \ref{nu_decay}, up to the subsequence, $(\nu_n)_{n \ge 1}$ converges weakly to  some $\nu_{R,f}$.  {Denoting by $\tau_{R,f}$ the optimal stopping time for $\mathcal{P}_0(\mu,\nu_{R,f})$, we have  $(\nu_{R,f},\tau_{R,f}) \in \mathcal{A}_{R,f}(\mu)$}.

\medskip

We claim that $\nu_{R,f}$ is the optimal target measure for $\mathcal{P}_{R,f}(\mu)$.
Note that  by Appendix \ref{opt_sko_embed}, 
$\mathcal{P}_0(\mu,\nu_n) = \mathcal{D}_0(\mu,\nu_n)$ ($\mathcal{D}_0$ denotes the dual problem of  $\mathcal{P}_0$, see  \eqref{d0}). Also, since the map $\nu \mapsto \mathcal{D}_0(\mu,\nu)$ is convex, it is lower semi-continuous with respect to the weak topology. Thus,
\[ \mathcal{P}_0(\mu,\nu_{R,f}) = \mathcal{D}_0(\mu,\nu_{R,f}) \leq \liminf_{n \rightarrow \infty} \mathcal{D}_0(\mu,\nu_n) = \liminf_{n \rightarrow \infty} \mathcal{P}_0(\mu,\nu_n) = \liminf_{n \rightarrow \infty} \mathcal{C}(\tau_n), \]
verifying the claim. Hence, $(\nu_{R,f},\tau_{R,f})$ is optimal  for $\mathcal{P}_{R,f}(\mu)$.  

\medskip

Observe that for $R'>R$,  the optimal pair $(\nu_{R,f},\tau_{R,f})$ is admissible for $\mathcal{P}_{R',f}(\mu)$ as well. Thus, by Theorem \ref{monotonicity}, $\tau_{R',f} \leq \tau_{R,f}$ a.s., implying that  $\nu_{R'}$ is admissible for $\mathcal{P}_{R,f}(\mu)$ and $(\nu_{R',f},\tau_{R',f}) \in \mathcal{A}_{R,f}(\mu)$.  Theorem \ref{monotonicity} again yields that $\tau_{R,f} \leq \tau_{R',f}$, thus $\tau_{R,f} = \tau_{R',f}$ a.s. and in particular $\nu_{R,f} = \nu_{R',f}$ a.e. This shows that $(\nu_{R,f}, \tau_{R,f})$ is an optimal pair for $\mathcal{P}_f(\mu)$ as long as $R>0$ is sufficiently large so that $\mathcal{A}_{R,f}(\mu) \neq \emptyset$.

\medskip
 
The uniqueness of the optimal pair $(\nu, \tau)$  follows from Theorem \ref{monotonicity}.
    \end{proof}

By following the argument in \cite[Theorem 7.6]{kim2021stefan}, we obtain the universality result for the optimal target measure. 

\begin{theorem}[Universality] 
Let $\nu_i$ be the optimal target measure for $\mathcal{P}_f(\mu)$ with costs $\mathcal{C}_i$, $i=1,2$, with  different Lagrangians  and of possibly different Types $(I)$ or (II). Then $\nu_1=\nu_2$ a.e.
\label{universal}
\end{theorem}

In addition, the monotocity result (Theorem \ref{monotonicity}) implies the following $L^1$ contraction:

\begin{theorem}[$L^1$ Contraction] Assume that $\nu_i$ is optimal for $\mathcal{P}_f(\mu_i)$, $i=1,2$. Then,
\[ || (\nu_1-\nu_2)_{+} ||_{L^1(\R^d)} \leq ||  (\mu_1 - \mu_2)_{+} ||_{L^1(\R^d)}.  \]
\end{theorem}

\begin{proof} This is a consequence of Theorem \ref{monotonicity} (see \cite[Theorem 7.8]{kim2021stefan}).
\end{proof}

This also implies a BV bound when $f$ is constant:

\begin{theorem}[BV Bound] Assume that $f$ is constant and that $\nu$ is optimal for $\mathcal{P}_f(\mu)$. Then,
\[ ||\nu||_{\textup{BV}} \leq ||\mu||_{\textup{BV}}. \]
\end{theorem}

\begin{proof} This follows from the $L^1$ contraction and that the problem is homogenous with respect to spatial shifts (see \cite[Theorem 7.10]{kim2021stefan}).
\end{proof}

\subsection{Identification the optimal target measure}

In this section, we  show that the optimal target measure  saturates the constraint for $\mathcal{P}_f(\mu)$. 
 
\begin{theorem}[Saturation on the Active Region] Let ($\nu^*,\tau^*$) be the optimizer for $\mathcal{P}_f(\mu)$ and $(\eta,\rho)$ be the associated Eulerian variables. Then $\nu^* = f$
a.e. in $\{x \in \R^d:\eta(t,x)>0 \hbox{ for some } t>0\}$.
 \label{saturation}
\end{theorem}
\begin{proof}
     Let us consider the Type (I) case since parallel arguments work for Type (II).  
Since $\nu^*\leq f$, we have  $\nu^*=f=0$ when $f=0$. Thus, it suffices to show that  for any $t>0$, 
 \begin{equation}\label{claim}
 | \{\nu^*<f\} \cap \{f>0\} \cap \{\eta(t,\cdot)>0\}  |=0.
 \end{equation}  
For $\delta>0$, define $H_{\delta} := \{  \nu^* \leq f-\delta \} \cap \{  f  > 2\delta \}$. As $\{\nu^*<f\} \cap \{f>0\}= \bigcup_{n\in\N} H_{1/n}$, it reduces to show that  for any $t>0$,
\begin{align}
    |H_{\delta} \cap  \{\eta(t,\cdot)>0\}|=0.
\end{align} 

 Assume that $H_{\delta} \cap  \{\eta(a,\cdot)>0\}$ has a positive Lebesgue measure for some $a>0$. Since the positive set of $\eta$ is non-increasing in time,  
\[ \iint_{[0,a]\times H_{\delta} } \eta dt dx > 0. \] 
By a characterization of $\eta$ (Lemma \ref{Eulerian_Variables}) and Fubini's theorem,
\[ \int_0^{\infty} \mathbb{P}[ t < \tau^*, (t,X_t) \in [0,a]\times H_{\delta}  ] dt > 0. \] Therefore, there exists $\overline{t}>0$ such that
\begin{equation} \mathbb{P}[\overline{t} < \tau^*, X_{\overline{t}} \in H_{\delta} ] > 0. \label{new_stopping_1} \end{equation} 
One can deduce that this violates the optimality of $\tau^*$. Indeed, if \eqref{new_stopping_1} holds, then one can construct a 
 stopping time $\overline{\tau}$ which attains the strictly smaller value of $\mathcal{P}_f(\mu)$, by stopping some of the active particles at time $\overline{t}$. We refer to \cite[Theorem 8.3]{kim2021stefan} (Case (II)) for details in the case of Brownian motion.  
\end{proof}

Given Theorem \ref{saturation}, one can characterize the optimal target measure $\nu$ for $\mathcal{P}_f(\mu)$ in terms of the non-local obstacle problem.
Note that by Theorem~\ref{universal}, $\nu$ is the same for the cost of either Type (I) or Type (II).
\begin{theorem}\label{nu_identify}  Let $u:\R^d \rightarrow \R$ be the unique continuous viscosity solution of
\begin{equation}
     \min \{ (-\Delta)^s u + f - \mu, u  \} = 0, \quad \lim_{|x|\to\infty} u(x)=0 \text{ in } \R^d. \label{initial_set_expansion} 
\end{equation} 
Then the optimal target measure $\nu$ for $\mathcal{P}_f(\mu)$ is given by $\nu = \mu - (-\Delta)^s u$ in $\R^d$. Also it satisfies that $\nu = f$ in $E := \{x\in \R^d: u(x)>0\}$.
\end{theorem}

\begin{proof} 
By Theorem~\ref{universal},
it suffices to only consider Type (I).  Let  $(U_{\mu_t})_{t\ge 0}$ be the  potential flow associated to the optimizer $(\nu,\tau)$ of $\mathcal{P}_f(\mu)$. By \eqref{relation} and Theorem \ref{saturation}, $\nu =f$ a.e. in $\{x \in \R^d:w(t,x)>0 \hbox{ for some } t>0\}$ (recall that $w$ is defined in \eqref{w}). Since $w$ is continuous and non-increasing in time,
setting  
\begin{align} \label{def e}
    E':=\{x\in \R^d: w(0,x)>0\},
\end{align}
we obtain
\[ \nu  = f \quad \hbox{ in } E'. \] Hence  $u(x):=w(0,x) = U_{\mu}(x) - U_{\nu}(x)$ satisfies  
\[ (-\Delta)^su=  \mu-\nu=\mu-f \quad \text{ in } E'. \] 
In addition, since  $\nu \leq f$,
\[ 0 = (-\Delta)^s u + \nu - \mu \leq (-\Delta)^s u + f - \mu \quad \text{ in } \R^d.  \] This shows that
\[ \min \{  (-\Delta)^s u + f - \mu,u  \} = 0 \quad \text{ in } \R^d. \] 
Since $\nu$ has a compact active region, by Lemma \ref{nu_decay},    $\nu$  satisfies the decay condition \eqref{decay}. As $\mu$ is compactly supported, $u = U_{\mu} - U_{\nu}$ decays at infinity.

\medskip

Hence we deduce that $u$ satisfies \eqref{initial_set_expansion}. Due to the uniqueness of \eqref{initial_set_expansion},  which follows from the standard comparison principle  (see \cite{fernandez2023integro}), we deduce $E=E'$ and  conclude  the proof.

\end{proof}

\begin{remark}  
Theorem \ref{nu_identify} allows the classical  variational problem to yield the optimal target measure $\nu$ for $\mathcal{P}_f(\mu)$. Assume that $\varphi : \R^d \rightarrow \R$ satisfies $(-\Delta)^s \varphi = f-\mu$ with $\lim_{|x| \rightarrow \infty} \varphi(x)=0$. Then it is known that $v := u  - \varphi$ minimizes the functional 
\[ \int_{\R^d} |(-\Delta)^\frac{s}{2} w|^2  \] in $\{w \in H^s(\R^d): w \geq \varphi\}.$ (see  for instance \cite{fernandez2023integro}). 
   
\end{remark}

\begin{lemma} Let $(\eta,\rho)$ be the   Eulerian variable associated to the optimizer of $\mathcal{P}_f(\mu)$ for Type (I). Then the initial domain $E(\eta)$ in \eqref{initial} is same as  the set $E$ in Theorem \ref{nu_identify} (up to zero-measure).\label{initial_trace_identification}
\end{lemma}

\begin{proof} 
Since $\{\eta(t,\cdot)>0\}$ is non-increasing in $t$, we have $E(\eta)=\cup_{t > 0} \{\eta(t,\cdot)>0\}$. By \eqref{relation}, this set is same as $\cup_{t > 0} \{w(t,\cdot)>0\} = \{w(0,\cdot)>0\}$ (up to zero-measure). This together with \eqref{def e} imply that $E(\eta)=E$ (recall that $E=E'$ in the proof of Theorem \ref{nu_identify}).  
\end{proof}

Finally, with the aid of our characterization of $\nu$, we show that the conversion of particles from the distribution $\mu$ to $\nu$ is  completed in a finite time for Type (I), if $f$ is uniformly positive.

\begin{theorem}
\label{everything_freezes} 
Further assuming that $f(x) \geq \delta$ for all $x\in \R^d$ ($\delta>0$ is a constant),
let $\tau$ be the  optimal stopping time for $\mathcal{P}_f(\mu)$ for Type (I). Then for some constant  $\overline{T}= \overline{T}(\delta,||\mu||_{L^1})<\infty,$ 
\[ \tau \leq \overline{T} \ \text{ a.s.} \]   
\end{theorem}

\begin{proof} 
By Lemma \ref{initial_trace_identification}, denoting by $s: \R^d \rightarrow \R^+$  the barrier function associated to $\tau$, we have  $s = 0$ a.e. on $E^c$. Hence, by Theorems \ref{equality_hitting_time} and  \ref{saturation}, we obtain the result by a parallel argument, similar to the one in \cite[Theorem 8.6]{kim2021stefan}.
\end{proof}

 \section{Connection to Fractional Stefan Problem} \label{connection_stefan}
In this section, given any function $f$ satisfying Assumption \ref{assume3}, we construct a solution to the (weighted) fractional Stefan problem \eqref{St1f} and \eqref{St2f}, as defined in   Definition~\ref{weighted_stefan_def};
\begin{equation}
      \partial_t h   +(-\Delta)^s \eta = 0, \qquad h=h_1(\eta):=\eta -f \chi_{\{\eta>0\}} - \kappa_1 \chi_{ \{\eta=0\} }\label{St1f} \tag{$St_{1,f}$} 
\end{equation}
and
\begin{equation}
\p_t h + (-\Delta)^s \eta = 0, \quad \quad h =h_2(\eta):=\eta + f\chi_{\{\eta>0\}} +  \kappa_2 \chi_{ \{\eta=0\} }\tag{$St_{2,f}$}     \label{St2f}
\end{equation}
with $\kappa_1$ and $\kappa_2$ defined in \eqref{Type_I_jump} and \eqref{Type_II_jump} respectively with the property $\kappa_i \leq f$ in $\{\eta=0\}$. 

We use the letter $f$ to emphasize {both its role as a weight} and the role of the saturation result (Theorem \ref{saturation}) 
in constructing our solutions.
Let us briefly outline the strategy. By Theorem \ref{exist_unique_Pf}, there exists  a unique optimizer $\nu$ for $\mathcal{P}_f(\mu)$. Also by  Theorem \ref{saturation}, $\nu=f$ on the active region, which allows us to apply Theorem \ref{solution_eta_stefan} to obtain the result. 

\medskip

In the special case of $f \equiv 1$ we obtain solutions to the unweighted fractional Stefan problem \eqref{St1} and \eqref{St2}.  The results are parallel to the local case obtained in \cite{kim2021stefan}.  For  \eqref{St2}, we obtain a unique solution with an active region which traces back to that of the initial data.  While for \eqref{St1}, by allowing $f$ to be a characteristic function, we will produce a unique solution in terms of the \emph{insulated region} that has a prescribed initial data, but with an enlarged initial trace (see Theorem \ref{Solve_Insulated} later).

\medskip

Let us first construct solutions to the weighted melting Stefan problem \eqref{St2f} when $\{0<\mu\leq f\}$ is empty.

\begin{theorem}\label{melting_solve} Let $f$ satisfy Assumption \ref{assume3} {and further assume that $f>0$ a.e.}  Consider $\mu = (f + \eta_0) \chi_{\Sigma}$, where $\eta_0\in L^{\infty}(\R^d)$ is positive and $\Sigma$ is a bounded Borel set  in $\R^d$ having a positive Lebesgue measure. Let $(\nu,\tau)$ be the optimizer for $\mathcal{P}_f(\mu)$ with a cost of Type (II), and let $(\eta,\rho)$ be the associated  Eulerian variables. Then $\eta$ is a weak solution to \eqref{St2f} with initial data $(\eta_0 \chi_{\Sigma}, \Sigma)$.
\medskip

In particular when $f \equiv 1$, $\eta$ is the unique weak solution to  \eqref{St2} with initial data $(\eta_0\chi_{\Sigma}, \Sigma)$, and the corresponding enthalpy $h$ from Definition \ref{enthalpy_def} is the unique weak solution of \eqref{Ste} with initial data $h_0 = (\eta_0+1)\chi_{\Sigma}$.  In addition, there exist  constants $\gamma,C>0$ depending only on $\mu,d$ and $s$ such that for any $t \ge 0,$
\begin{equation}\label{rate_h} ||h(t,\cdot)-\nu||_{L^1(\R^d)} \leq Ce^{-\gamma t}. \end{equation}

\end{theorem}

\begin{proof} Recall from Theorem \ref{barrier_existence} that the optimal stopping time $\tau$ is a hitting time to some backward barrier {with a potential randomization on $\{\tau=0\}$}. Denoting by $s:\R^d \rightarrow \R^+$ the associated barrier function, we have $s=0$ a.e. on  $\Sigma$, since $\mu>f$ on $\Sigma$. In particular, the stopping time randomizes to stop the initial $f \chi_{\Sigma}$ particles. The remaining particles $\tilde{\mu} := \eta_0 \chi_{\Sigma}$ are transported for a positive time to yield a final distribution $\tilde{\nu}$, the optimal target measure for $\mathcal{P}_{ f (1-\chi_{\Sigma}) }(\tilde \mu)$.
 
Let $\tilde{\tau}>0$ be the stopping time $\tau$ restricted to the initial distribution $\tilde{\mu}$, and let $(\tilde{\eta},\tilde{\rho})$ be the Eulerian variables associated to $(\tilde{\mu},\tilde{\tau})$.   
{Recalling that the initial particles $f \chi_{\Sigma}$ 
 stop immediately,  not affecting the active distribution $\eta$ and thus $\tilde{\eta}=\eta$.}
Also, as $\tilde{\rho} + (f \chi_{\Sigma})  \delta_{ \{t=0\} } \otimes dx  = \rho$,  the spatial marginal distribution of  $\tilde{\rho}$ is $\tilde{\nu}:= \nu-f \chi_{\Sigma}$. 

\medskip

Since $\tilde{\tau}>0$ a.s., Theorem \ref{solution_eta_stefan} implies  that for any $\varphi \in C^{\infty}_c(\R^+ \times \R^d)$,  
\[  \int_{0}^{\infty}  \int_{\R^d} (-\p_t \varphi\cdot (\eta + \tilde{\nu} \chi_{ \{\eta>0\} } + \kappa_2 \chi_{ \{\eta=0\}} ) +  (-\Delta)^s \varphi \cdot\eta )dx dt = \int_{\Sigma} \varphi(0,\cdot) \eta_0  dx  \] 
($\kappa_2$ denotes the quantity \eqref{Type_II_jump}, recall that $\tilde{\eta}=\eta$). Recall that the  saturation result (Theorem \ref{saturation}) implies $\nu \chi_{ \{\eta>0\} }  = f \chi_{ \{\eta>0\} } $.  Since $\tilde{\nu} = \nu - f\chi_{\Sigma}$, it follows that  $\tilde{\nu} \chi_{ \{\eta>0\} } = f \chi_{\Sigma^c} \chi_{ \{\eta>0\}}$. Therefore, 
\[  \int_0^{\infty} \int_{\R^d} (-\p_t \varphi \cdot (\eta +  f \chi_{ \{\eta>0\} } + \kappa_2 \chi_{ \{\eta=0\}} ) + (-\Delta)^s \varphi\cdot \eta)dxdt = \int_{\Sigma} \varphi(0,\cdot ) \eta_0 dx + \mathcal{A} \]
with
\begin{align*}
\mathcal{A}:=\int_0^{\infty}\int_{\Sigma}  -\p_t \varphi\cdot f  \chi_{ \{\eta>0\} } dx dt = \int_0^{\infty}\int_{\Sigma} -\partial_t \varphi  \cdot  f dx dt = \int_{\Sigma} \varphi(0,\cdot ) f dx,
\end{align*} where the second equality is obtained from the fact that $\Sigma \subset \{\eta(t,\cdot)>0\}$ for all $t>0$. 

\medskip

Now, we verify that  $ E (\eta) = \Sigma$ a.e. Since $\Sigma \subset E(\eta)$ a.e., we aim to show the reverse inclusion.  By Lemma \ref{nu_1_E},  we have $\tilde{\nu} \chi_{E(\eta)} = 0$ a.e. Also by Theorem \ref{saturation},
\[ \tilde{\nu} = {f} \text{ a.e. on } \bigcup_{t>0} \{\eta(t,\cdot)>0\} \setminus \Sigma. \]
Thus, {as $f>0$ a.e.,} we have that  $E(\eta)  \cap \left( \bigcup_{ t>0 } \{\eta(t,\cdot)>0\} \setminus \Sigma \right) $ is a null set. Hence, as $E(\eta) \subset \bigcup_{t>0}  \{\eta(t,\cdot)>0\}$, we conclude that $\eta$ solves \eqref{St2f} with  initial data $(\eta_0 \chi_{\Sigma}, \Sigma)$. \\

When $f \equiv 1$, \eqref{St2f} simplifies to \eqref{St2} (see Definition \ref{def_St}). This implies $\eta$ is a weak solution to \eqref{St2}. The uniqueness of $\eta$ to \eqref{St2} along with the uniqueness  of  $h$ to \eqref{Ste} follow from Theorem \ref{stefan_enthalpy}.\\

Lastly we show \eqref{rate_h}. Recall that $\tilde{\tau}>0$ a.s. and $(\eta,\tilde{\rho})$ are its associated Eulerian variables.
By definition $\tilde{\nu} = \nu  - \chi_{\Sigma}$ and $\tilde{\rho}([0,t),x) = \rho( [0,t),x) - \chi_{\Sigma}(x) $, and thus $\tilde{h}(t,x)-\tilde{\nu}(x)  = h(t,x)-\nu(x)$,  where $\tilde{h}(t,x) := \eta(t,x) + \tilde{\rho}([0,t),x)$. Now by Theorem \ref{exist_unique_Pf}, there is an $R=R(\mu,d,s)>0$ such that $\tau \leq \tau_R$ almost surely and $\mu$ is supported on $B_R(0)$. So as $\tilde{\tau}>0$ a.s., we may apply Lemma \ref{asymptotic_h} on the pair $(\mu,\tilde{\nu})$ to deduce there are constants $\gamma=\gamma(\mu,d,s),C=C(\mu,d,s)>0$ such that
\[ ||\tilde{h}(t,\cdot)-\tilde{\nu}(\cdot)||_{L^1(\R^d)} \leq Ce^{-\gamma t}. \] Recalling that  $\tilde{h}-\tilde{\nu} = h-\nu$, we conclude the proof.

 \end{proof}

One can also show the Duvuat transform of our solution solves the parabolic obstacle problem:

\begin{theorem}[Melting Obstacle Problem] \label{associated_obstacle_melt}
Let $f,\mu$ and $\eta$ be as given in Theorem \ref{melting_solve}. Then $w(t,x) := \int_0^t \eta(a,x)da$ solves the parabolic obstacle problem
\begin{equation}
    \min\{ \p_t w + (-\Delta)^s w + f - \mu, w \} = 0, \qquad w(0,\cdot)=0. \label{stefan_obstacle_melt}
\end{equation} 
\end{theorem}

\begin{proof}
By  Theorem \ref{theorem 3.13} and \eqref{w_explicit}, it suffices to show that  $\nu$ in \eqref{w_obstacle} can be replaced with $f$. Indeed, the  saturation result (Theorem \ref{saturation}) along with \eqref{relation} imply that $\nu = f$ in $\{x\in \R^d:w(t,x)>0 \text{ for some } t >0\}$. As $\nu \leq f$,  we may apply the arguments in Theorem \ref{nu_identify} to conclude the proof.
\end{proof}

\begin{remark} The regularity of $(-\Delta)^s w$ and the free boundary of \eqref{stefan_obstacle_melt} have been studied in the literature \cite{caffarelli2013regularity, barrios2018free, ros2021optimal, figalli2023regularity, ros2023semiconvexity} when $f-\mu$ is smooth. For instance, if $f-\mu$ is smooth, then $(-\Delta)^s w$ is continuous for $0<s<1$, which would imply that the enthalpy is continuous. This would be an interesting contrast to the local case where the enthalpy  jumps by $1$ across the free boundary. Unfortunately this condition does not apply for us because our assumption in Theorem~\ref{melting_solve}, $\mu>f$ in $\{\mu>0\}$, yields a discontinuous $\mu$ for most choices of $f$. While the discontinuity of $\mu$ only occurs at the initial free boundary, the local regularity result for \eqref{stefan_obstacle_melt} is not available at the moment.
\end{remark}

Next, we analyze the freezing problem \eqref{St1}. 

\begin{theorem} \label{freeze_solve} Assume that $f$ satisfies Assumption \ref{assume3}. Consider $\mu = (f + \eta_0) \chi_{\Sigma}$, where $\eta_0\in L^{\infty}(\R^d)$ is positive and $\Sigma$ is a Borel set  in $\R^d$ having a positive Lebesgue measure. Let $(\nu,\tau)$ be the optimizer for $\mathcal{P}_f(\mu)$ with a cost of Type (I), and let $(\eta,\rho)$ be the associated Eulerian variables. Let $u$ be a solution to \eqref{initial_set_expansion}, and set  $E: = \{ u > 0 \}$. Then $\eta$ is a solution of \eqref{St1f} with initial data $(\mu, E)$. In addition, $\Sigma \subseteq E$.
\end{theorem}

\begin{proof} 

Note that by  Lemma \ref{open_barrier}  and Theorem \ref{equality_hitting_time},  the initial trace of $\eta$ is $\{ s>0\}$, where $s:\R^d \rightarrow \R^+$ denotes the barrier function  associated  to the optimal stopping time $\tau$. By Lemma \ref{initial_trace_identification}, this set is equal to $E$ a.e.  We claim that $\Sigma \subseteq E$ a.e. Since $\tau$ is the hitting time to the epigraph of $s$, if $s=0$ on some $A \subseteq \Sigma$ with $|A|>0$, then the initial $\mu \chi_A$ particles stop immediately. Thus $\nu \chi_A \geq \mu \chi_A > f \chi_A$, which yields a contradiction since $\nu \leq f$. Therefore $\Sigma \subseteq E$.
\medskip

Observe that by our choice of $\mu$, the optimal stopping time $\tau$ satisfies $\tau>0$ a.s., since the above argument implies that  $s>0$ a.e. on $\Sigma$.  Also by the saturation result (Theorem \ref{saturation}),  $\nu = f$ a.e. on $\{x \in \R^d:\eta(t,x)>0 \text{ for some  } t>0 \}$. Hence by Theorem \ref{solution_eta_stefan} {and Lemma \ref{initial_trace_identification}}, $\eta$ solves \eqref{St1f} with initial data $(\mu,E)$.
\end{proof}

Next, we state a connection between \eqref{St1f} and the parabolic obstacle problem, which is a direct consequence of Theorem~\ref{theorem 3.13}. Based on the instability of the supercooled problem, this indicates that $w$ has a low regularity compared to the one for the melting case. At the level of the obstacle problem, this is due to the fact that $\p_t w = -\eta\leq 0$. We refer to the examples constructed in \cite[Remark 3.7]{caffarelli2013regularity} in the case $s=1/2$.

 \begin{theorem}  Let $f,\mu$ and $\eta$ be as given in Theorem \ref{freeze_solve}. Then $w(t,x) := \int_t^{\infty} \eta(a,x) da$ solves the parabolic obstacle problem
 \begin{equation} \min\{ \p_t w + (-\Delta)^sw + \nu ,w \}= 0, \qquad w(0,\cdot) = U_{\mu} - U_{\nu}. \label{stefan_obstacle_freeze}    \end{equation} Here $\nu$ is the optimal target measure of $\mathcal{P}_f(\mu)$ and is also given by Theorem \ref{nu_identify}.  \label{associated_freezing_obstacle}
\end{theorem}

Note that, in contrast to \eqref{St2f}, the initial domain $E$ depends on the choice of $f$ in the case \eqref{St1f}. Hence $f\equiv 1$ no longer corresponds to a unique solution of \eqref{St1} with initial data $\mu$. In fact uniqueness  may not hold even with the given initial data and domain, see the example for the local case in \cite{kim2021stefan}.  We can however uniquely characterize solutions of \eqref{St1} in terms of the \emph{insulated region}, a region where the particles never freezes.

\begin{definition}[Insulated Region] For a solution $\eta$ of \eqref{St1}, its insulated region is defined as
\begin{equation}
     \Sigma(\eta) := \bigcap_{t \geq 0} \{\eta(t,\cdot)>0\}.
\end{equation}
\end{definition}

{By the argument of Remark \ref{inital_domain_well_posed}, we see that if $\eta_1$ and $\eta_2$ are two versions of $\eta$ (in $\R^+ \times \R^d$) whose  positive sets are open and  {non-increasing} in time, then  $\Sigma(\eta_1)=\Sigma(\eta_2)$ up to zero-measure.} 

\begin{theorem}
Assume that $\mu \in L^{\infty}(\R^d)$ with a  compact support which is not identically zero. Let $G$ be a non-empty bounded open set in $\R^d$ that contains $\{ x \in \R^d: 0 < \mu(x) \leq 1 \}$.  Let $\eta$ and $E$ be given in Theorem \ref{freeze_solve} with  $f:= 1-\chi_{G}$. Then,  $\eta$ solves \eqref{St1} with initial data $(\mu,E)$ and an insulated region $G$. 
  \label{Solve_Insulated}
\end{theorem}

\begin{proof} Our first goal is to show that $G$ is the insulated region of $\eta$. Before verifying this, we show that for any time $t>0$, $\eta(t,\cdot)$ is not identically zero. Since  $\rho = 0$ on  $\R^+ \times G$, by the Eulerian variable equation \eqref{non_local_heat}, $\eta$ solves the fractional heat equation $\p_t \eta  = -(-\Delta)^s \eta$ on the open set $\R^+ \times G$ with initial data $\mu$ and the Dirichlet condition $\eta$. As $\mu(\cdot) = \eta(0,\cdot)$ is not identically zero, the maximum principle implies that $\eta(t,\cdot)$ is not identically zero for any $t \ge 0$.

\medskip

Now we show that the insulated region of $\eta$ is $G$. To see that $G \subseteq \Sigma(\eta)$, recall by \eqref{relation} that $\{\eta>0\} = \{w>0\}$, where $w$ solves \eqref{w_obstacle}. Since {$\nu=0$ on $G$},
\[ \p_t w + (-\Delta)^sw \geq 0\quad  \text{ on }  {(0,\infty) \times G.} \] 
As $\eta(t,\cdot)$ is not identically zero for any $t$, the same holds for $w$. Thus on its global minimum set $\{w=0\}$, we have $\p_t w + (-\Delta)^s w < 0$. Therefore, $w>0$ on $(0,\infty) \times G$, which implies $G \subseteq \Sigma(\eta)$. 

For the reverse inclusion, we first show that $\{\nu=0\} \subseteq G$. To see this, recall from \eqref{w_obstacle} that
\[ \p_t w + (-\Delta)^s w \geq 0\quad  \text{ on }  (0,\infty) \times \{\nu=0\} . \] 
Hence using the parallel argument as above, 
\begin{align*}
    \{\nu=0\} \subseteq \{x\in \R^d:w(t,x)>0 \text{ for some  } t>0 \} = \{x\in \R^d:\eta(t,x)>0 \text{ for some  } t>0 \}.
\end{align*}
By the saturation result (Theorem \ref{saturation}), we have $\{\nu=0\} \subseteq  G$. As $\Sigma(\eta) \subseteq \{\nu=0\}$, we deduce that $\Sigma(\eta) \subseteq 
 G$. Therefore, we establish  $\Sigma(\eta) =
 G$.

\medskip
 
    Finally, we  verify that  $\eta$ solves \eqref{St1}. Observe that for any test function $\varphi,$
    \[ \int_0^{\infty}\int_G \p_t \varphi \cdot \chi_{ \{\eta>0\}} dx dt = -\int_G \varphi(0,x) dx. \] 
    By Theorem \ref{freeze_solve},  $\eta$ is a solution to \eqref{St1f} with initial data $(\mu,E)$. Thus recalling $f:= 1-\chi_{G}$ and adding the above identity to the weak form of \eqref{St1f}, setting  $h := \eta - \chi_{ \{\eta>0\} } - \kappa_1 \chi_{ \{\eta=0\} }$,
\[     \iint (-\p_t \varphi \cdot  h +   (-\Delta)^s \varphi \cdot   \eta ) dtdx  = \int \varphi(0,\cdot) (\mu - \chi_E -  \kappa_1 (0, \cdot)\chi_{E^c}) dx.  \] 
Hence we conclude the proof.
\end{proof}

Finally, we show that the initial data and the insulated region are enough to characterize \eqref{St1}. In particular, our choice of the initial trace coming from the elliptic obstacle problem is \emph{necessary}.

\begin{theorem} Assume that $\mu \in L^{\infty}(\R^d)$ with a  compact support.  Then for any bounded measurable set $G$ in $\R^d$, there is at most one solution to \eqref{St1} with initial data $\mu$ and an insulated region $G$. If the solution exists (which we call $\eta$), then $\eta$ is the Eulerian variable associated to the optimal stopping time of $\mathcal{P}_f(\mu)$ with $f:=1-\chi_{G}$.
In addition, the initial domain $E=\{\eta( 0^+,\cdot)>0\}$ is given by a positivity set of the obstacle solution solving \eqref{initial_set_expansion}, and $E$ contains the support of $\mu$.

\label{uniqueness_St1}
\end{theorem}

  \begin{proof} By the consistency theorem (Theorem \ref{St_1_Consistent}), we have that $(\eta,\rho)$ (where $\rho$ is defined in \eqref{def}) are the optimal Eulerian variables for a cost of Type (I) between $\mu$ and $\nu = \chi_{G^c \cap E}  + \rho([0,\infty),\cdot ) \chi_{E^c}$.  We show that this is enough to characterize the target measure $\nu$ even on $E^c$. 

First it follows from Section \ref{barrier_property} that $w(t,x) := \int_t^{\infty} \eta(a,x) da$ solves the obstacle problem \eqref{w_obstacle} for Type (I). In addition, we have $u(\cdot) := w(0,\cdot) = U_{\mu}-U_{\nu}$, and thus $\nu = \mu - (-\Delta)^s u$. Now let us show that $u$ is uniquely specified by our assumptions. 

\medskip

Note that $G \subseteq E$, thus $\nu \leq 1-\chi_G$ with an equality on $E$ (recall $\rho([0,\infty),\cdot ) \leq 1$ due to Definition \ref{weighted_stefan_def}). Hence, by the arguments in the proof of Theorem \ref{nu_identify}, we deduce that  $u$ is a unique continuous viscosity solution of \eqref{initial_set_expansion} with $f = 1-\chi_{G} $. In particular, any solution of \eqref{St1} satisfying the assumptions in Theorem \ref{uniqueness_St1} induces a unique target measure $\nu$. 

Since the choice of $\nu$ is unique, a solution of \eqref{St1} satisfying the assumptions in Theorem \ref{uniqueness_St1} is unique. This is because $w$ solves the obstacle problem \eqref{w_obstacle} for Type (I), which has a unique solution due to the comparison principle (see \cite{fernandez2023integro}). 

Note that $\nu$ is the same target measure as $\mathcal{P}_{1-\chi_G}(\mu)$ due to Theorem \ref{nu_identify}, and then Lemma \ref{initial_trace_identification} implies that the initial trace is $E = \{u>0\}$. Lastly, the fact that $E$ contains the support of $\mu$ follows from Theorem~\ref{freeze_solve}.

\end{proof}

A special case is when $G$ is an empty set and $f\equiv  1$.  As a consequence of Theorems~\ref{everything_freezes}, \ref{freeze_solve} 
and  \ref{uniqueness_St1}, we establish the following corollary.
\begin{corollary}
    \label{vanish_finite_time}
 Let $\mu$ and $\eta$ be as in Theorem~\ref{freeze_solve} with  $f \equiv 1$. Then  $\eta$ is the unique solution to \eqref{St1} with initial data $\mu$ that vanishes in a finite time. The initial domain of $\eta$ is a positivity set  of the obstacle solution solving \eqref{initial_set_expansion} with $f \equiv  1$.

\end{corollary}
  
\appendix

\section{Preliminaries  on Measures} \label{measures}
 
 For any Borel set $A$ in the Euclidean space,  we denote  by $\mathcal{M}(A)$ (resp. $\mathcal{M}_1(A)$) the space of  {finite {(signed)}} Radon measures (resp. probability measures) on $A$. As mentioned in the introduction, we are often interested in expressions of the form $\pi([0,t],x)$, where $\pi$ is a space-time measure. We now give a precise definition of this expression:

\begin{definition}[Disintegration and Density]\label{disintegration}
Assume that both $\pi \in \mathcal{M}(\R^+ \times \R^d)$ and $f \in L^1(\R^d) \cap L^{\infty}(\R^d)$ are non-negative. We say that $\pi$ has a \emph{spatial marginal} $f$ if for any Borel set $A \subset \R^d$, $\pi([0,\infty) \times A) = \int_A f(x) dx$.
 This particularly implies that for any Borel set $I \subset \R^+$, the spatial measure $A \mapsto \pi(I \times A)$ is absolutely continuous w.r.t. Lebesgue measure on $\R^d$.  We denote the associated density by $\pi(I,x)$. 

One can disintegrate the measure $\pi$ into {probability measures} $\{\pi_x\}_{x \in \R^d}$ on $\R^+$ such that
\begin{equation}
    \pi(I \times A) = \int_{A} \pi_x(I)f(x) dx,
\end{equation} 
where $I \subset \R^+$ is Borel. In other words, $\pi(I,x) = \pi_x(I)f(x)$ $x$-a.e. 
\end{definition}
 We show that the choice of including or excluding the endpoints of the interval $[0,t]$ in $\pi([0,t],x)$ does not affect the expression in the above definition:

\begin{lemma} \label{lambda_equality}
Let $\pi$ be as in Definition \ref{disintegration}. Set $\lambda_1(t,x) := \pi([0,t],x)$ and $\lambda_2(t,x) := \pi([0,t),x)$. Then $\lambda_i \in L^{\infty}(\R^+; \R^d)$ for $i=1,2$. Also    for any $x \in \R^d$, we have $\lambda_1(t,x) = \lambda_2(t,x)$ $t$-a.e.
\end{lemma}

\begin{proof} The $L^{\infty}$ property follows from the fact that  the  spatial marginal  of $\pi$ is  $f \in L^{\infty}(\R^d)$.
The equality of $\lambda_1$ and $\lambda_2$ follows from the fact that for any fixed $x\in\R^d$,  $\lambda_1(t,x)$ is monotone in $t$, and $\lambda_1(t,x) \neq \lambda_2(t,x)$ if and only if $\lambda_1(t,x)$ jumps at time $t$. As the set of times where monotone functions can jump is countable, we conclude the proof.
\end{proof}

Next, we state a general integration by parts formula for $\pi([0,t],x)$.

\begin{lemma}  \label{diff_rho_t} 
Let $\pi$ be as in Definition \ref{disintegration} and $0 \leq a < b$. Then for any $\varphi \in C^{\infty}_c( \R^+  \times \R^d)$,
\begin{align}
    \int_{\R^d} \int_{a}^{b} &\p_t \varphi(t,x) \cdot \pi([0,t],x) dtdx \nonumber \\
&= \int_{\R^d} \varphi(b,x) \pi([0,b],x) - \int_{\R^d} \varphi(a,x) \pi([0,a],x) -\iint_{[a,b] \times \R^d } \varphi(t,x) \pi(dt,dx). \label{diff_rho[t]}
\end{align} 
\end{lemma}

\begin{proof}
   
We use the notation $\pi_x([0,t])$ from Definition \ref{disintegration}. Since $\pi([0,t],x) = \pi_x([0,t])f(x)$, by the integration by parts formula for the Riemann–Stieltjes integral, for each $x\in \R^d,$
\[ \int_{a}^{b} \p_t\varphi(t,x) \cdot \pi([0,t],x)  dt = \varphi(b,x) \pi([0,b],x) - \varphi(a,x) \pi([0,a],x) -\int_{a}^{b} \varphi(t,x) f(x) \pi_x(dt). \] Integrating this expression over $\R^d$ concludes the proof.
\end{proof}

\section{Fractional Sobolev Space \texorpdfstring{$H^s$}{Hs} and Optimal Control}
\label{hs_info}

\subsection{Fractional Sobolev space $H^s$} \label{Hs_review}

In this section, we briefly recall  the fractional Sobolev spaces and the fractional Laplacian  on $\R^d$.  For $s \in (0,1)$, define the Gagliardo seminorm
\[ [u]_{H^s(\R^d)} := \iint_{\R^d \times \R^d} \frac{|u(x)-u(y)|^2}{|x-y|^{d+2s}} dxdy. \]
With the aid of Fourier transform,  Gagliardo seminorm can be written as 
\begin{equation}  \label{gali_norm_and_laplacian}
[u]_{H^s(\R^d)} = 2C_{d,s}^{-1} || (-\Delta)^{\frac{s}{2}} u||^2_{L^2(\R^d)} , \end{equation}
where $C_{d,s}$ denotes the normalization factor in the fractional Laplacian \cite[Proposition 3.6]{di2012hitchhiker}. 
\begin{definition} We say that $u \in H^s(\R^d)$ if 
\[ ||u||^2_{H^s(\R^d)} := ||u||^2_{L^2(\R^d)} + [u]_{H^s(\R^d)} < \infty. \]
Also, given a bounded and open set $\Om$ in $ \R^d$, we say that $u \in H^s_0(\Om)$ if $u = 0$ \text{ a.e. in } $\Om^c$ and $||u||_{H^s_0(\Om)} := ||u||_{H^s(\R^d)} < \infty$.
\end{definition}

We state the following Sobolev embedding theorem \cite[Theorem 7.1]{di2012hitchhiker}.

    \begin{theorem}[Sobolev embedding]
    Let $\Om$ be a bounded and open  set in $\R^d$ with a   Lipschitz boundary. Suppose that $\mathcal{F}\subset L^2(\Om)$ satisfies
\[ \sup_{f \in \mathcal{F}} ||f||_{L^2(\Om)} < +\infty \quad \textup{and}\quad  \sup_{f \in \mathcal{F}} [f]_{H^s(\R^d)} < +\infty. \] 
Then, $\mathcal{F}$ is pre-compact in $L^2(\Om)$.
\end{theorem}

We introduce some relationships between the Gagliardo seminorm and the fractional Laplacian (see for instance \cite[Proposition 9]{servadei2013variational}):
\begin{proposition}[Poincare's Inequality]  \label{poincare_ineq} 
 Let $\Om$ be a bounded and open  set in $\R^d$ with a Lipschitz boundary. Then, there is a constant $C>0$ such that for any $u \in H^s_0(\Om)$, \[ ||u||_{L^2(\Om)}^2 \leq C[u]_{H^s(\R^d)}. \]
 The optimal constant $C$ above is given by $C_{d,s}/(2\lambda)$, where $C_{d,s}$ is the normalization constant in \eqref{gali_norm_and_laplacian} and $\lambda$ is the principal eigenvalue of the Dirichlet problem
\begin{equation}
	\begin{cases} (-\Delta)^s u = \lambda u &\text{ in } \Om, \\ u = 0 &\text{ on } \Om^c. \end{cases}
\end{equation} 
\end{proposition}

  In particular, \eqref{gali_norm_and_laplacian} together with Poincar\'e's inequality imply that for any $u \in H_0^s(\Om)$, the following fractional Poincar\'e inequality  holds: 
  \begin{align}  \label{pi}
      ||u||_{L^2(\Om)}^2 \leq \frac{1}{\lambda} || (-\Delta)^{\frac{s}{2}} u||^2_{L^2(\R^d)}. 
  \end{align}
 This immediately implies that the norm $||u||_{H^s_0(B)} := ||(-\Delta)^{\frac{s}{2}} u||_{L^2(\R^d)}$ is equivalent to the usual $H^s(\R^d)$ norm.

\subsection{Viscosity Solutions with respect to  $(-\Delta)^s$} \label{visc_optimal}

In this section, we review some crucial properties of the fractional Laplacian and viscosity solutions.

\begin{definition}[Viscosity solution]
  Let $K>0$ be a constant and  $\Om$ be a bounded and open  set in $\R^d$.
  We say that a lower semi-continuous function $\psi$  is a \emph{viscosity
super solution} to the equation $(-\Delta)^s \psi  =  -K$  in $\Om$ if for any $x_0\in \Om$ and $\varphi \in C^2_b(\R^d)$  such that $\psi  - \varphi$
has a global minimum at $x_0$,
\begin{align*}
    (-\Delta)^s \varphi(x_0) \geq -K.
\end{align*}
We often say that $\psi$ is  a viscosity solution to  $(-\Delta)^s u \geq -K$ if
 $\psi$ is a viscosity super solution to  $(-\Delta)^s u = -K$.
\end{definition}

We show that viscosity solutions to $(-\Delta)^s u \geq -K$ are in fact also distributional solutions. 

\begin{lemma}
Let $K>0$ be a constant and  $\Om$ be a bounded and open  set in $\R^d$.
Assume that  $(-\Delta)^s  u \geq -K$ in $\Om$ in the viscosity sense, $u(x) = 0$ for $x \notin \Om$, and $u$ is continuous. Then, for any non-negative function $\varphi \in C^{\infty}_c(\Om)$,
\[ \int_{\R^d} u \cdot (-\Delta)^s \varphi \geq -K \int_{\R^d} \varphi.  \]  \label{visc_dist}
\end{lemma}

\begin{proof} 
 For $\e>0$, let $u_\e$ be the inf-convolution of $u$, i.e. 
\[ u_{\e}(x) = \inf_{y\in \R^d} \left\{ u(y) + \frac{|x-y|^2}{2\e} \right\}. \] Then, there is some constant $c_{\e}$ with $ \lim_{\e \rightarrow 0} c_{\e} =0$  such that   $(-\Delta)^s u_{\e} \geq -K  - c_{\e}$ in $\Om$ in the classical sense 
(see  \cite[Proposition 5.4]{caffarelli2009regularity}).  By using a mollification, we may further assume that $u_{\e}$ is smooth. Then, by integration by parts, for any non-negative $\varphi \in C^{\infty}_c(\Om)$,
\[  \int_{\R^d} u_{\e} \cdot (-\Delta)^s \varphi = \int_{\R^d} (-\Delta)^s u_{\e}\cdot \varphi  \geq (-K-c_{\e}) \int_{\R^d} \varphi.  \]  
The standard properties of the inf-convolution implies that  $u_{\e} \rightarrow u$ locally uniformly.
Since $(-\Delta)^s \varphi(y) = O( |y|^{-d-2s})$ as $y \rightarrow \infty$,  by sending  $\e \rightarrow 0$, we conclude the proof. 
\end{proof}

\section{Optimal Skorokhod Embeddings} \label{opt_sko_embed}

In this section, we consider the Skorokhod's embedding problem which involves the variational problem $\mathcal{P}_0(\mu,\nu)$ in \eqref{Var_Skoro}. The variational problem \eqref{Var_Skoro} is a linear minimization problem over a convex set, and thus $\mathcal{P}_0(\mu,\nu)$ admits a dual problem. 
The associated dual problem is given by
\begin{equation} \label{d0}
    \mathcal{D}_0(\mu,\nu) := 
    \sup_{\substack{\psi \in C_c(\R^d)\\ G=(G_t)_{t\ge 0} \in \mathcal{K}^{+}_{-\gamma} }}\left\{ \int_{\R^d} \psi(z) \nu(dz) - \mathbb{E}[G_0] :  G_t - \psi(X_t) \geq -\int_0^t L(s,X_s) ds \text{ for } t \leq \tau_r \right\},
\end{equation} where $\mathcal{K}^{+}_{-\gamma}$ denotes the set of progressively measurable continuous supermartingales with respect to our $2s$-stable process with $\gamma$-exponential growth  (see \cite{ghoussoub2019pde} for the local case).

Note that   in the above variational problem  \eqref{d0},  by the Doob-Meyer decomposition,  we may assume that $G=(G_t)_{t\ge 0}$ is a martingale. In addition, we may further assume that
\begin{align} \label{support}
    \psi(x) = 0, \quad \forall x \notin B,
\end{align} where we recall $B=B_r(0)$ with $r$ from Assumption \ref{assume1}. To see this, for $\psi \in C_c(\R^d)$, let $\hat{\psi}$ be a solution (the notions of weak solution and of viscosity solution agree, see \cite[Remark 2.11]{ros2014dirichlet} for explanations) to the equation
\[
    \begin{cases}
        (-\Delta)^s \hat{\psi} = 0 &\text{ in } B, \\ \hat{\psi}= \psi &\text{ on } B^c.
    \end{cases}
\] Note that $\hat{\psi} \in C^{\infty}(B) \cap C_b(\R^d)$ due to \cite{ros2014dirichlet}, and thus by Ito's formula 
$\hat{\psi}(X_{t\wedge \tau_r})$ is  martingale. Thus,  if $(\psi,(G_t)_{t\ge 0})$ is an admissible pair in \eqref{d0}, then  $(\psi-\hat{\psi}, (G_t-\hat{\psi}(X_t))_{t\ge 0})$ is also an  admissible pair with the same value on the action functional, since $\mu \leq_{s\text{-SH}}  \nu$. 
Hence, one can assume \eqref{support} in the variational problem \eqref{d0}. 

Therefore this implies that  by \cite[Theorem 2.3]{ghoussoub2021optimal}, for any $\mu$ and $\nu$ satisfying Assumption \ref{assume1} and $L$ satisfying Assumption \ref{assum2}, 
\begin{align} \label{eq0}
    \mathcal{P}_0(\mu,\nu) = \mathcal{D}_0(\mu,\nu).
\end{align}
We refer to \cite[Theorem 1.2]{beiglbock2017optimal} and  \cite[Theorem A.1]{ghoussoub2019pde} for the Brownian motion case.  
In addition,  by \cite[Theorem 5.7]{ghoussoub2021optimal}, there exists an optimizer for $\mathcal{D}_0(\mu,\nu)$. 

\medskip

Now, we apply \cite[Theorem 6.2]{ghoussoub2021optimal} to establish Theorem \ref{barrier_existence}.\\

\noindent \emph{Proof of Theorem \ref{barrier_existence}.}  
First, we verify the assumptions \textbf{(A0)}, \textbf{(A1)}, \textbf{(B0)}, \textbf{(B1)}, \textbf{(C0)}-\textbf{(C2)}, \textbf{(D0)} and \textbf{(D1)}  in \cite[Theorem 6.2]{ghoussoub2021optimal} (under the additional disjointness assumption  on $\mu$ and $\nu$ for a Type (II) cost).  

\medskip

 Let $S_t := \int_0^t L(s,X_s) ds$, then $S_t$ satisfies Assumption \textbf{(A0)}, due to the continuity of $t \mapsto S_t$ which is a consequence of $L \in L^{\infty}$. To see \textbf{(A1)}, note that our compact active region assumption  and non-negativity of $L$  imply $S_t \leq S_{\tau_r} \in L^1$ ($r$ is from Assumption \ref{assume1}), and thus $S$ is  uniformly integrable over
stopping times. 

To verify the condition \textbf{(B0)},  note that if $\psi \in C(\R^d)$  satisfies $\psi(x)=0$ for $x \notin B_r(0)$, then the r{\'e}duite $\hat{\psi}$ of $\psi$ on $B_r(0)$ is the viscosity solution  to
\[ \begin{cases} \min \{ (-\Delta)^s \hat{\psi}, \hat{\psi} - \psi \} = 0 &\text{ in } B_r(0), \\  \hat{\psi}(x)=\psi (x) = 0  &\text{ on } [B_r(0)]^c .
\end{cases} \] 

Also, this solution is known to be continuous \cite{fernandez2023integro}.  Note that the condition \textbf{(B0)} is about the continuity of $\hat{\psi}$ for $\psi \in C_b(\R^d)$, but from our above argument of reducing the optimization set of $\mathcal{D}_0(\mu,\nu)$ to have $\psi(x)=0$ for $x \notin B_r(0)$, we only need to check \textbf{(B0)} for $\psi \in C_b(\R^d)$ with $\psi(x)=0$ for $x \notin B_r(0)$. Next, as $L \geq 0$, $S_t$ is a submartingale,  and thus \textbf{(B1)} holds. 

Assumption \textbf{(C0)} and \textbf{(C1)} are satisfied due to Proposition \ref{poincare_ineq} and  Lemma \ref{visc_dist} respectively. Also \textbf{(C2)} follows from the condition $L \in L^{\infty}$.

Finally Assumptions \textbf{(D0)} and \textbf{(D1)} are shown to be satisfied in \cite{ghoussoub2021optimal} (if further $\mu$ and $\nu$ are disjoint in Type (II) case). This was proved in  the example after the definition of \textbf{(D0)} and \textbf{(D1)} in \cite{ghoussoub2021optimal}. 

\medskip

Now with these assumptions having been verified, one can apply \cite[Theorem 6.2]{ghoussoub2021optimal} to obtain for costs of Type (I) that the optimal stopping time $\tau^*$ is given by 
\begin{equation} \label{optimal_hitting_time_form} \tau^* = \inf \{t  \ge 0: \varphi(t,X_t) = \psi(X_t)  \} \leq \tau_r , \end{equation} where $\psi$ denotes the  optimizer in a dual problem \eqref{d0}, which is bounded and lower semi-continuous that is zero outside $B_r(0)$, and 
\begin{equation} \varphi(t,x) := \sup_{ \tau \in \mathcal{S}, \tau \geq t } \mathbb{E} \left[ \psi(X_{\tau}^{t,x}) - \int_t^{\tau} L(a,X^{t,x}_a) da \right], \label{H_psi} \end{equation} where $X^{t,x} = (X_a^{t,x})_{a \ge t}$ denotes the $2s$-stable process started from $x$ at time $t$. 

We verify that  $ R_1:=\{ ( t,x)\in \R^+ \times \R^d : \varphi(t,x) = \psi(x) \}$ is a forward barrier in Type (I) case.
Since $L$ is strictly increasing in time, the dynamic programming principle \cite[Proposition 4.2]{ghoussoub2019pde} implies that $\varphi$ is non-increasing in time. In addition,  by taking $\tau=t$ in  \eqref{H_psi}, we have $\varphi(t,x) \geq \psi(x)$.
Therefore, it follows that if $\varphi(t,x) = \psi(x)$, then $\varphi(t+\delta,x) = \psi(x)$  for any $\delta > 0$, that is,  $R_1$ is a forward barrier. In other words,
\[ \tau^* = \inf\{t  \ge 0: t \geq s_1(X_t) \}, \]
where $s_1(x) := \inf\{ t \geq 0 : \varphi(t,x) = \psi(x) \}.$ \\

One can argue similarly for Type (II) case, with the only exception being that $\mu$ and $\nu$ are required to be disjoint to apply the result \cite[Theorem 6.2]{ghoussoub2021optimal}.  To achieve this, consider instead $\mathcal{P}_0(\tilde{\mu},\tilde{\nu})$ with $\tilde{\mu} := \mu - \mu_0$ and $\tilde{\nu} := \nu - \mu_0$, where $\mu_0$ is the shared mass between $\mu$ and $\nu$. It follows from optimality that if $\tau$ is the optimal stopping time for $\mathcal{P}_0(\mu,\nu)$, {then on $\{\tau>0\}$ we have that} $\tau$ is the optimal stopping time for $\mathcal{P}_0(\tilde{\mu},\tilde{\nu})$. $\qed$ \\    

Next, we verify that the stopped particles are inside the barrier. 

\begin{lemma} \label{stop_inside_barrier} Let $\tau$ be the optimizer of  $\mathcal{P}_0(\mu,\nu)$ in \eqref{Var_Skoro} (assume additionally that $\tau>0$ a.s. for Type (II)). Then if $R$ denotes the associated barrier from Theorem \ref{barrier_existence}, then   $(\tau,X_{\tau}) \in R$ a.s. In other words, denoting by   $s$ the associated barrier function, we have $\tau \geq s(X_{\tau})$ for Type (I) and $\tau \leq s(X_{\tau})$ for Type (II).
\end{lemma}

\begin{proof} 
We only consider the  Type (II) case since the argument is similar for Type (I) case. As $\tau>0$ a.s., $\mu$ and $\nu$ are disjoint, {since otherwise the optimizer $\tau$ would randomize at $t=0$ to instantly stop the shared mass.} Thus the optimal stopping time $\tau$ is given by \eqref{optimal_hitting_time_form}. Then by \cite[Theorem 2.4]{ghoussoub2021optimal},  $\varphi(\tau,X_{\tau}) = \psi(X_{\tau})$ a.s., implying that $(\tau,X_{\tau}) \in R$ a.s.

\end{proof}

Finally a useful fact for this paper is the following uniqueness theorem for Eulerian variables for Type (I), which we will use to uniquely characterize the freezing Stefan problem. The statement is identical to \cite[Lemma 4.5]{ghoussoub2019pde} which was originally proven for the Brownian motion.

\begin{theorem}[Uniqueness of Eulerian Variables for Type (I)] \label{uniqueness} Assume that $R$ is a measurable forward barrier in $\R^+ \times \R^d$.
Then, for any $\mu \in L^{\infty}(\R^d)$ with a compact support, there is at most one solution  $(\eta,\rho)$ to \eqref{non_local_heat} satisfying $\eta(R)=0$ and $\rho(R)=1$.

\end{theorem}

\begin{proof} 
The argument is based on \cite[Lemma 4.5]{ghoussoub2019pde} which deals with the case of Brownian motion. Let  $(\eta_i,\rho_i)$, $i=1,2$, be solutions to \eqref{non_local_heat}. Then, define $\eta := \eta_1 - \eta_2$ and $\rho := \rho_1 - \rho_2$.  Consider the potential function $\varphi$
\begin{equation}
     \varphi(t,\cdot) := N*(\eta(t,\cdot) + \rho([0,t),\cdot)),
\end{equation} where $N$ denotes the Riesz potential (see \eqref{riesz}). 
Then by similar arguments as in Lemma \ref{open_barrier},
\begin{equation}
    \begin{cases}
        \p_t \varphi = -\eta, \\ (-\Delta)^s \varphi(t,\cdot) = \eta(t,\cdot)  + \rho([0,t),\cdot) \label{C1_identity}.
    \end{cases}
\end{equation} 
By taking $\varphi$ as a test function in
 \eqref{non_local_heat} for  $i=1,2$, 
 \[ \iint \varphi \rho(dt,dx) = \iint  -2\eta^2 - \eta \cdot \rho([0,t),\cdot) .  \]
 Note that since $R$ is a measurable forward barrier, $\eta(t,x) \cdot \rho([0,t),x)  = 0$ for all $(t,x) \in \R^+ \times \R^d$.  
 
 Also by \eqref{C1_identity}, $\varphi(t,x) = \lim_{s \rightarrow \infty} \varphi(s,x) =: \varphi_{\infty}(x)$ on $R$. Hence, as $\rho$ is supported on $R$,
\[ \iint \varphi \rho(dt,dx) = \iint \varphi_{\infty} \rho(dt,dx) = \int \varphi_{\infty}(x) \rho([0,\infty),dx) =  \int | (-\Delta)^{s/2} \varphi_{\infty}(x)|^2, \] where in the last equality we used the fact $(-\Delta)^s \varphi_{\infty}(x) = \rho([0,\infty),x)$. By the above displays, 
 \[ \iint 2 \eta^2 + \int |(-\Delta)^{s/2} \varphi_{\infty}|^2 = 0. \] 
 Therefore we obtain $\eta_1 = \eta_2$, which yields $\rho_1 = \rho_2$ as well.
\vspace{0mm}
\end{proof}

\bibliographystyle{amsplain}
\bibliography{stefan_bib}

\end{document}